\documentclass[a4paper, hidelinks, 11pt]{article}
\usepackage{graphicx} %
\usepackage{amsmath, relsize} 
\usepackage{amsthm} 
\usepackage{amsfonts}
\usepackage{amssymb}
\usepackage{url}
\usepackage{amsmath}
\usepackage{amssymb}
\usepackage{amsthm}
\IfFileExists{bbm.sty}{\usepackage{bbm}}{\let\mathbbm\mathbf}
\usepackage{tocloft}

\usepackage{enumerate}
\usepackage{setspace}
\usepackage[utf8]{inputenc}
\usepackage{graphicx,color,hyperref}
\newcommand{\rev}[1]{{\color{black}#1}}
\newcommand{\bis}[1]{{\color{black}#1}}
\newcommand{\blu}[1]{{\color{black}#1}}
\usepackage[margin=0.8in]{geometry}
\parskip \medskipamount
\newtheorem{theorem}{Theorem}[section]
\newtheorem{proposition}[theorem]{Proposition}

\newtheorem{lemma}[theorem]{Lemma}
\theoremstyle{definition}
\newtheorem{definition}[theorem]{Definition}
\newtheorem{remark}[theorem]{Remark}
\newtheorem*{correctionremark}{Remark}

\newcommand{\N}     {\mathbb{N}}

\newcommand{\E}     {\mathbb{E}} 
 
\newcommand{\T}     {\mathbb{T}}

\newcommand{\Acal}   {{\mathcal A }}

\newcommand{\Ecal}   {{\mathcal E }} 
 
\newcommand{\Gcal}   {{\mathcal G }}

\newcommand{\Wcal}   {{\mathcal W }}

\numberwithin{equation}{section}

\begin{document}
\title{Uniformly positive correlations in the dimer model and \\   phase transition in lattice permutations \\ in $\mathbb{Z}^d$, $d > 2$,  via reflection positivity}
\author{Lorenzo Taggi\thanks{Weierstrass Institute for Applied Analysis and Stochastics, Berlin, DE; taggi@wias-berlin.de}}
\date{\today}
\maketitle

\begin{abstract}
Our first main result is that correlations between monomers in the dimer model in $\mathbb{Z}^d$ do not decay to zero when $d > 2$. This is the first rigorous result about correlations in the dimer model in dimensions greater than two and shows that the  model behaves drastically differently than in two dimensions, in which case it is  integrable and correlations are known to decay to zero polynomially.
Such  a result is implied by our more general, second main result, which states the occurrence of a phase transition in  the model of lattice permutations, 
which is related to the quantum Bose gas.
More precisely, we consider a self-avoiding walk interacting with lattice permutations and we prove that, in the regime of fully-packed loops,
such a walk is `long' and the distance between its end-points grows linearly with the diameter of the box. 
These results follow from the derivation of a version of the infrared bound from {a new general probabilistic setting}, with coloured loops and walks interacting at sites and walks entering into the system from some `virtual' vertices. 
In dimension $d=3$ we  use  a transfer matrix argument to obtain a sharper infrared bound.
\end{abstract}

\begin{figure}
  \centering
     \includegraphics[width=0.50\textwidth]{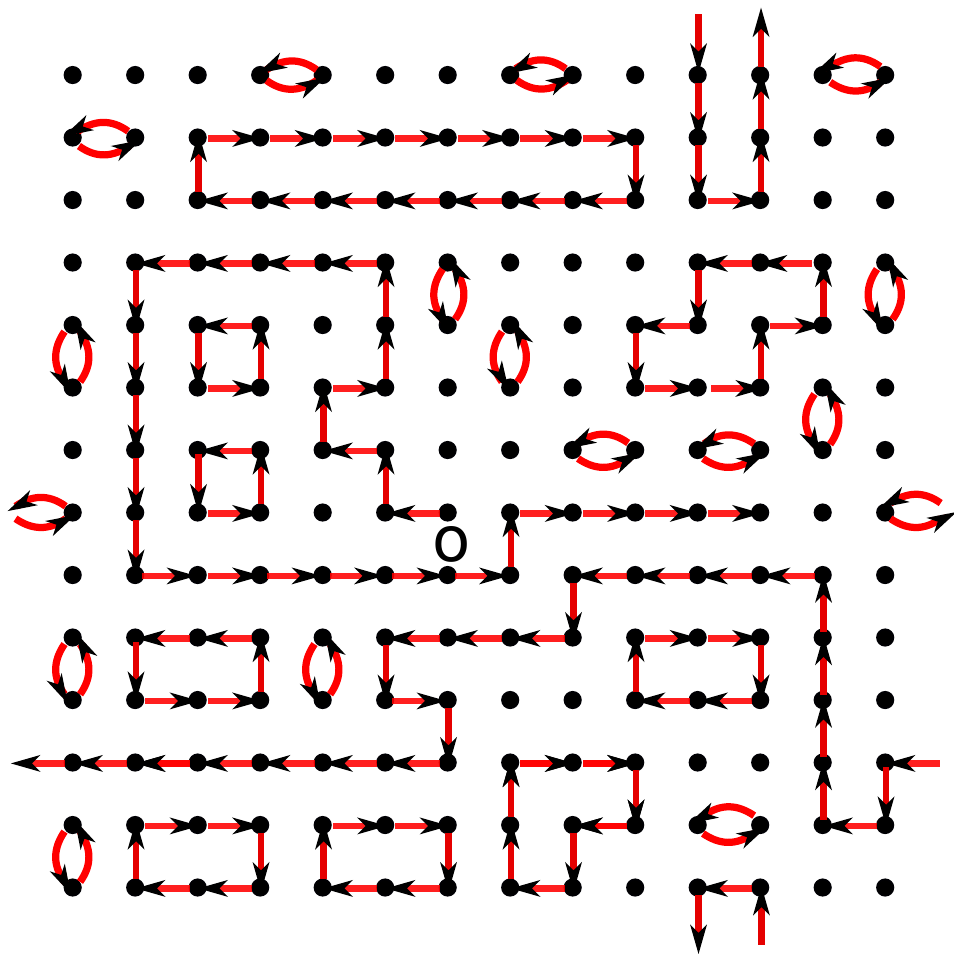}
      \caption{A realisation $\pi \in \Omega$ on the torus.}\label{Fig:examplepi}
\end{figure}

\section{Introduction}
\label{sect:Introduction}
This paper considers two models related to each other, the dimer model and lattice permutations.

The dimer model is a classical statistical mechanics model whose configurations are perfect matchings of a graph, namely subsets of edges which cover every vertex precisely once.
The model attracts interest from a wide range of perspectives, which include combinatorics, statistical mechanics, and algorithm complexity studies.
Its rigorous mathematical study  achieved a breakthrough with  the  works of Kasteleyn, Temperley and Fisher,  \cite{Fisher, Kasteleyn, TemperleyFisher} in 1961, who showed that on planar graphs the dimer problem is \textit{exactly solvable}.
{Since then, various aspects}  of the dimer model have been explored:
For example its close relation to the critical Ising model  \cite{AuYang, Kasteleyn},
a characterisation of  the model's correlations \cite{FisherStephenson},
the arctic circle phenomenon \cite{Cohn2},
their continuous limits and the emergence of conformal symmetry
\cite{Giuliani, Kenyon2, Kenyon3}.

Despite so much progress on planar graphs,  the rigorous mathematical understanding of the dimer model in higher dimensional graphs is  still very poor.  Indeed, as it was formalised by Hammersley \textit{et al.} \cite{Hammersley1}, the method of  Kasteleyn, Temperley and Fisher, which consists of reducing the problem of enumerating the number of dimer covers to the  problem of computing the Pfaffian of the so-called Kasteleyn matrix, cannot be naturally extended to  $\mathbb{Z}^d$, $d > 2$, in which case  it was shown  \cite{Jerrum} that the dimer model is  computationally intractable.

This paper presents the first result about correlations in the dimer model in $\mathbb{Z}^d$, when $d > 2$. 
More precisely, we consider the monomer-monomer correlation, i.e., the ratio between  the number of dimer covers with two monomers and 
the number of dimer covers with no monomers,  which is a central quantity in the study of this model. In dimensions $d = 2$, it was shown that it decays to zero \textit{polynomially} with the distance between the two monomers   \cite{Dubedat, FisherStephenson}. 
Our first main result, Theorem \ref{theo:theo1} below, states that such a function \textit{does not decay to zero} with the distance when $d > 2$. 
This is in agreement with {physicists' predictions}  \cite{Huse} based on heuristic arguments.  
As a by-product of our technique we also deduce that, in the infinite volume limit, the correlation between monomers along the cartesian axis {at odd distances} equals 
$\frac{1}{2d}$ up to {non-positive correction terms} of order ${
O(d^{-3/2})
}$,
which are uniform with respect to the distance between such monomers.

Our first main result is implied by our more general main result about the model of \textit{lattice permutations},  which,
in the form as we define it, can be viewed as a generalisation of the 
\textit{double dimer model}  \cite{Dubedat2, Kenyon5}.
The configuration space of the model can be viewed as the set of directed multi-graphs {whose vertex set is the set of vertices} of a box in $\mathbb{Z}^d$ and such that any connected component is either a `monomer' (a single vertex with no edges which are incident to it), a `double edge' (a connected component consisting of two vertices and  two parallel edges pointing opposite directions), or a directed self-avoiding loop. A measure which assigns to each such  graph a weight which depends on two parameters,  $\rho \in [0, \infty)$, the \textit{monomer activity}, and  $N \in [0, \infty)$, the \textit{number of colours}, is introduced. The parameter $\rho$ rewards  the number of monomers, while the parameter $N$ rewards the number of loops and double edges.

The study of lattice permutations has been proposed in \cite{Biskup2, Betz3, Gandolfo, Grosskinsky} in view of their  connections to  \textit{Bose-Einstein condensation}
\cite{Feynman}, which is an important unsolved statistical mechanics problem.
Contrary to  these papers, where jumps of  arbitrary length are allowed and penalised according to a Gaussian weight (and no multiplicity factor for the number of loops and double edges is considered),  here  we  only allow jumps of length one or zero; this feature gives the model a combinatorial flavour and allows the connection with the dimer model.
The relevance of lattice permutations for the study of Bose-Einstein condensation is that, contrary to other spatial random permutation models which were studied before (for example \cite{BU1, Betz0, BU3, BU4, Bogachev, ElboimPeled, Ferrari}) and  similarly to the  \textit{interacting} quantum Bose gas,
a spatial interaction  which depends on the mutual distance of the loops takes place (loops interact by mutual exclusion).
This feature makes the techniques which have been employed in such previous works ineffective for the rigorous analysis of lattice permutations and the model interesting and challenging.
The central question for the quantum Bose gas is whether Bose-Einstein condensation takes place. In  \cite{UeltschiRelation} it is shown that, in a random loop model which is related to  lattice permutations, the two-point function, namely the ratio of the partition functions of a system with a forced `open' cycle and one without, can be used to detect Bose-Einstein condensation: If this ratio stays positive uniformly in the volume and in the spatial separation of the two endpoints of the forced cycle, this is equivalent to the presence of off-diagonal long range order \cite{Penrose}, which itself is equivalent to Bose-Einstein condensation.
{For lattice permutations, Theorem~\ref{theo:theo3} establishes positivity of the odd-sublattice average and positive correlations at odd separations along coordinate axes, in the fully-packed regime and for the values of $N$ specified in that theorem.}

The relevance of lattice permutations goes even beyond  their connection to the dimer model and Bose-Einstein condensation, which holds  when $N = 2$.
Indeed, they are an intriguing mathematical object for any value of $N \in [0, \infty)$ and can be viewed as a version of the \textit{loop O(N) model},
which is in turn related to spin systems with continuous symmetry for integer values of $N$  (see  \cite{PeledSpinka} for an overview). The  difference  between our setting and  the model considered in \cite{PeledSpinka} is that we also allow double edges and that the loop containing the origin is `open', namely it is a self-avoiding walk which starts from the origin and ends at an arbitrary vertex of the box. One of the most important questions for this class of models is the  identification of regions of the phase diagram where the loop length \textit{does not} admit exponential decay. This was recently accomplished  for the loop O(N) model on the hexagonal lattice using various techniques, for example parafermionic observables, planar spin representations, and Russo-Welsh estimates \cite{DuminilCopinParafermionic, Glazman}, see also further references in \cite{PeledSpinka}. Although very powerful, these techniques are specific for planar graphs and cannot be naturally extended to  $\mathbb{Z}^d$, $d > 2$, in which case only results stating exponential decay have been derived \cite{Chayes, Taggi} and techniques are missing. Our Theorem \ref{theo:theo2} below states that,  in any dimension $d > 2$, in the regime of fully-packed loops, the length of the  self-avoiding walk in lattice permutations grows unboundedly with the size of the box and the  distance between its two end-points is of the same order of magnitude as the diameter of the box. Hence, not only we rule out exponential decay in any dimension $d > 2$,  but we also identify the correct scaling of the distance between the end-points of the self-avoiding walk.

Our proof method is of independent interest and can be {viewed as a reformulation} of the famous approach of   Fr\"ohlich, Simon and Spencer  in \cite{Frohlich} in the space of paths.
In  \cite{Frohlich},  the property of \textit{reflection positivity} of a system of spins with continuous symmetry, known as spin O(N) model, was employed for the derivation of the so-called \textit{infrared bound}, which implies that correlations do not decay in such a spin system.
Such an approach was further developed in \cite{FrohlichI, FrohlichII} and implemented in  several other research works in the framework of  quantum and classical spin systems.
Here we implement such an approach in  a completely different  setting which does not involve spins, but a general probabilistic model of  interacting coloured loops and walks.
Our framework includes the (loop representation of) the spin O(N) model as a special case, and other random loop models for which no spin representation exists  or is  easy to derive, for example lattice permutations (see also Remark \ref{rem:extension} below).
Hence, our method can  be viewed as an extension of \cite{Frohlich, FrohlichI, FrohlichII}.

{
\begin{correctionremark}\hypertarget{remark:correction}{}
While working on an extension of the current paper~\cite{KlippelTaggiWu},
the author identified an error
confined to the computation of a constant in the long-range-order
argument in version~3 of the preprint~\cite{TaggiPreprint},
subsequently published in \textit{Communications on Pure and Applied
Mathematics} in 2022~\cite{TaggiCPAM}.
Specifically, the limit of $\mathcal I_L(d)$, defined in (5.9) of that
preprint, is incorrectly identified with $r_d/(4d)$ in Theorem~5.1,
equation~(5.4), and in (5.10).
Consequently, the lower bound $(2d)^{-1}(1-r_d/2)$ in (2.3) is not
justified, leaving the original proof of long-range order incomplete
in dimension three.
Repairing this constant estimate for $d=3$ requires a new argument,
presented in Section~\ref{rev3d-section}; the correction for $d>3$
is given in Section~\ref{sect:corrected-higher-d}.
The literature review reflects the state of the literature at the time
of acceptance of the original article on 10 December 2020, and has not
been updated to cover subsequent developments,
except for the references added to explain this correction.
\end{correctionremark}
}

\section{Definitions and main results}
\label{sect:rigorousresults}
We now provide a precise definition of the dimer model and of lattice permutations and we state our main results formally.
This section is divided into three paragraphs with each paragraph stating a main theorem. 
Our third theorem, Theorem \ref{theo:theo3} below, involves lattice permutations and it can be viewed as a reformulation of our Theorem \ref{theo:theo2} and as a generalisation of Theorem \ref{theo:theo1}, which involves the dimer model.

\paragraph{The Dimer model.}
\begin{figure}
  \centering
    \includegraphics[width=0.30\textwidth]{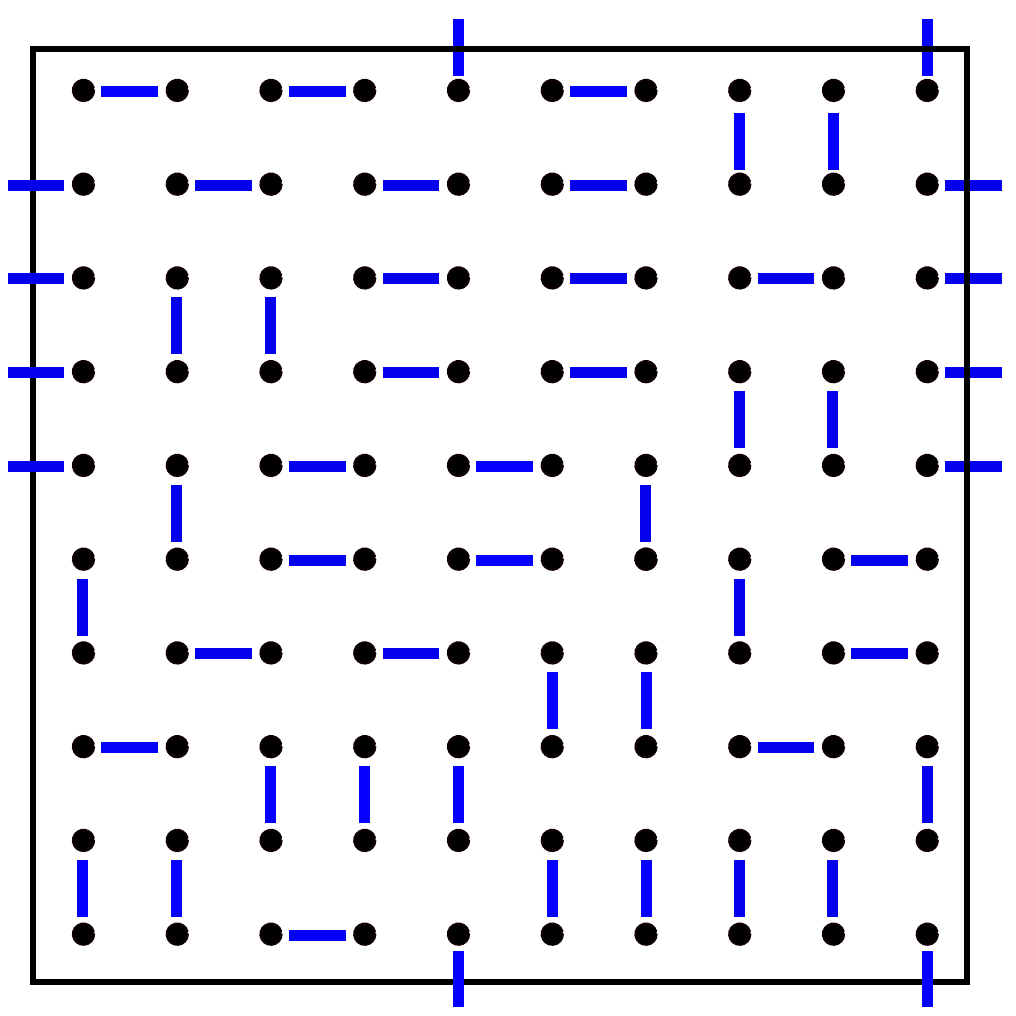}
        \includegraphics[width=0.30\textwidth]{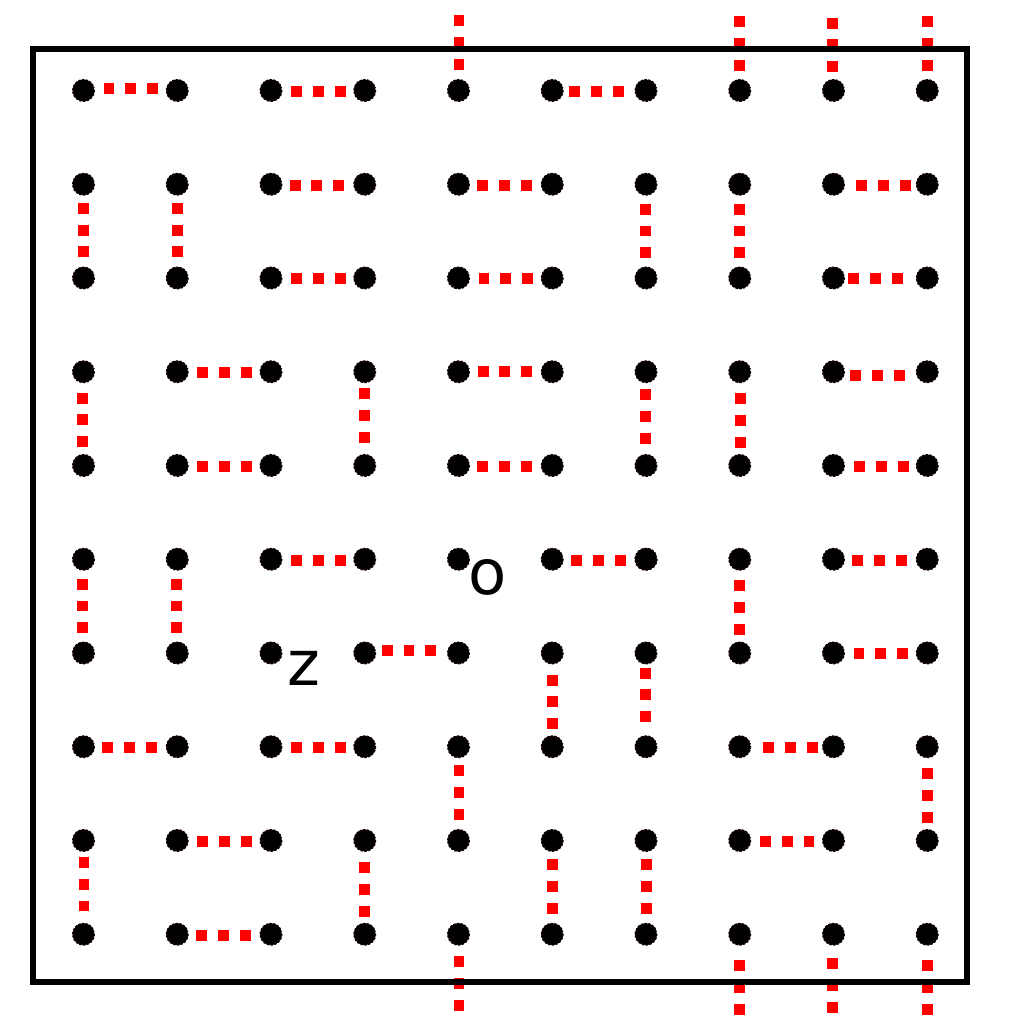}
            \includegraphics[width=0.30\textwidth]{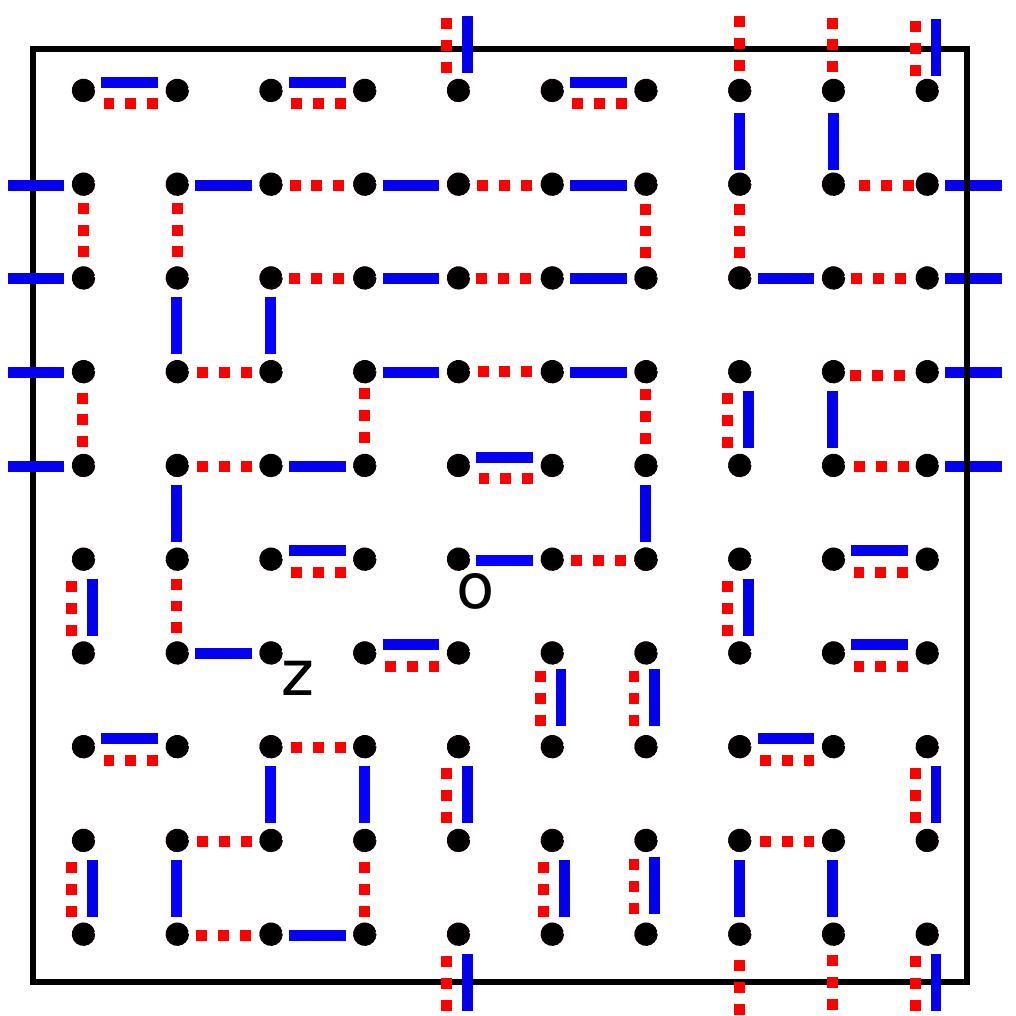}
      \caption{Left: A dimer cover in $\mathcal{D}(\emptyset)$. Centre: A dimer cover in  $\mathcal{D}(\{o,z\})$. Right: superposition of the dimer cover on the left and the dimer cover in the centre. }\label{Fig:dimerssuperposition}
\end{figure}
A \textit{dimer cover} of the graph $\mathcal{G} = (\mathcal{V}, \mathcal{E})$ is a spanning sub-graph\footnote{A spanning sub-graph of a graph $\mathcal{G}= (\mathcal{V},\mathcal{E})$ is a sub-graph of $\mathcal{G}$ whose vertex set is $\mathcal{V}$.} of $\mathcal{G}$ such that every vertex has degree one. 
Let $(\T_L, \E_L)$ be a graph  with vertex set 
$
\T_L$ $ : = \big \{ (x_1, \ldots, x_d) \in \mathbb{Z}^d $ $ : \, \, 
x_i \in ( - \frac{L}{2}, \frac{L}{2}  ] \big \}
$  and edges connecting nearest-neighbour vertices and boundary vertices so that  $(\T_L, \E_L)$ can be identified with the torus 
$\mathbb{Z}^d / L \mathbb{Z}^d$, where {$L\geq4$ is even. Throughout the paper torus sizes satisfy this convention.}
For any set of sites $M \subset \T_L$, 
let   $\mathcal{D}(M)$ be the (possibly empty) set of dimer covers of 
the graph which is obtained from  $(\T_L, \E_L)$ by removing all the sites which 
are in $M$ and from $\E_L$ all the edges which are incident to at least one vertex in $M$.
The \textit{monomer-monomer correlation} is a fundamental quantity for the analysis of the dimer model and it corresponds to the ratio between the number of dimer covers with two monomers and the number of {dimer covers} with no monomer,
\begin{equation}\label{eq:monomermonomercorr}
\forall x \in \T_L \quad  \Xi_L(x) : = \frac{ | \mathcal{D}(\{o,x\}) |}{  |  \mathcal{D}( \emptyset) |},
\end{equation}
where $o$ is used to denote the origin, $o = (0,  \ldots, 0)\in \T_L$.
See also Figure \ref{Fig:dimerssuperposition}.
This function equals zero if $L \in 2\mathbb{N}$ and $x$ belongs to the even sub-lattice $\T_L^{e} \subset \T_L$, which is now defined together with the odd sub-lattice,
\begin{equation}\label{eq:oddevensublattice}
\T_L^{e} : = \{  x \in \T_L \, \, : \, \,  d(o,x) \in 2 \mathbb{N}   \},
\quad \quad \quad 
\T_L^{o} : = \{  x \in \T_L \, \, : \, \,  d(o,x) \in 2 \mathbb{N}+1   \},
\end{equation}
where $d(o,x)$ is the graph distance in $(\T_L, \E_L)$.
\begin{theorem}\label{theo:theo1}
Suppose that $d>2$. There exists a constant $c_{\rm dim}=c_{\rm dim}(d)>0$ such that
\begin{equation}\label{eq:dimerpositivityaverage}
 \liminf_{\substack{L\to\infty\\L\text{ even}}}
 \frac1{|\T_L^o|}\sum_{x\in\T_L^o}\Xi_L(x)\ge c_{\rm dim},
\end{equation}
where $$c_{\rm dim}(d) =  \frac{1}{2d} \big ( 1 - O(\frac{1}{\sqrt{d}}) \big )$$ in the limit of large $d$. 
Moreover, for every $\varphi\in(0,c_{\rm dim})$ there exists a constant
$c_1=c_1(\varphi,d)>0$ such that, for every sufficiently large even $L$
and every odd integer $n\in(0,c_1L)$,
\begin{equation}\label{eq:dimerpointwise}
 \Xi_L(n\boldsymbol e_i)\ge\varphi,\qquad i\in\{1,\ldots,d\}.
\end{equation}
\end{theorem}

From the general site-monotonicity properties which were derived in  
\cite[Remark 2.5]{Lees}  we deduce that,
\begin{equation}\label{eq:pointwiseupper}
\forall L \in 2 \mathbb{N}, \quad \forall x \in \T_L, \quad \Xi_L(x) \leq \frac{1}{2d}.
\end{equation}
Hence,  the lower bound established in Proposition~\ref{prop:corrected-higher-d} below
gets closer to the  point-wise upper bound (\ref{eq:pointwiseupper}) as the dimension increases. Hence, the  larger is the dimension, the more uniform is the correlation between  monomers  across the odd sites of the torus.
For $x \in \mathbb{Z}^d$, define now $\Xi( x ) : =  \liminf_{L \rightarrow \infty } \, \, \Xi_{2L}  ( x )$.
Our bound (\ref{eq:dimerpointwise}) and the point-wise upper bound
(\ref{eq:pointwiseupper})  
imply that, when $d >2$, for any  integer $n \in 2 \mathbb{Z}+1$,
\begin{equation}\label{eq:liminf}
{
c_{\rm dim}(d)\leq
} 
\Xi( n \, \boldsymbol{e}_i ) \leq \frac{1}{2d},
\end{equation}
where $\boldsymbol{e}_i$ is any  cartesian vector.
Contrary to  (\ref{eq:liminf}), $\Xi(n \boldsymbol{e}_i)$  was shown by Fisher and Stephenson  \cite{FisherStephenson} to decay like $n^{ - \frac{1}{2}}$
when $d = 2$.
{Moreover,  we further deduce from (\ref{eq:liminf}) that the asymptotic behaviour
} of the monomer-monomer correlation in the limit of large dimension,  i.e., for any odd integer  $n$,
$$
\Xi( n \, \boldsymbol{e}_i )  = \frac{1}{2d} + {
O(d^{-3/2})
},
$$
where the error  term in the right-hand side  is uniform in $n$.

\paragraph{Lattice permutations.}
We  now introduce the model of lattice permutations.
Recall that $(\T_L, \E_L)$ denotes the torus, with edges connecting nearest neighbour vertices.
To begin, for any pair of sites $x, y \in \T_L$ such that $x \neq y$, let 
$\Omega_{x,y}$ be the set of directed multi-graphs 
$\pi = (\T_L, E_\pi)$  such that:
\textbf{(i)} the edges of $E_\pi$ connect nearest-neighbour vertices in the torus,
\textbf{(ii)} the in-  and the out-degree of every vertex in $\T_L \setminus \{x,y\}$ are equal and their value {is either zero or one},
\textbf{(iii)}  the out-degree of $x$ is one and  its
in-degree is zero, the out-degree of $y$ is zero and  its
in-degree is one.
This implies that the connected component of the graph $(\T_L, E_\pi)$
which contains $x$ is a walk which starts at $x$ and ends at $y$ and that any other connected component is either a monomer, a double edge or a loop, which we now define:  
a \textit{walk}  is a sub-graph which is isomorphic to a simple open curve in $\mathbb{R}^d$  and it is directed (and self-avoiding),
a \textit{monomer} is a connected component consisting of a single vertex with no edges  incident to it,
a \textit{double edge} is a connected component corresponding to a pair of nearest neighbour vertices, $z, w \in \T_L$, with an edge  directed from $z$ to $w$ and an edge directed from $w$ to $z$; 
a  \textit{loop} is a sub-graph which is isomorphic to a simple closed curve  in $\mathbb{R}^d$ and it is directed (and self-avoiding).
See also  Figure \ref{Fig:examplepi} for an example.
When $x  = y$, we define $\Omega_{x,y}$ as the set of 
directed multi-graphs 
$\pi = (\T_L, E_\pi)$ such that:
\textbf{(i)} the edges of $E_\pi$ connect nearest-neighbour vertices in the torus,
\textbf{(ii)} the in- and the out-degree of every vertex in $\T_L \setminus \{x\}$ are equal and their value is  either zero or one,
{\textbf{(iii)} the vertex $x=y$ is isolated and is marked as the degenerate walk, consisting of one vertex and no edges; it is not counted as a monomer of the open configuration.}
We define the \textit{configuration space} $\Omega : = \cup_{x \in \T_L} \Omega_{o,x}$. Each such $\pi \in \Omega$ can be viewed as a system of monomers, loops and double edges with a  walk which starts from the origin and ends at one unspecified vertex of the torus
and all these objects are mutually disjoint.
For any $\pi \in \Omega$, let  $\mathcal{M}(\pi)$ be the number of monomers of  $\pi$.
{Open configurations carry the marked walk and are distinguished from the corresponding unmarked closed configurations.}
Furthermore, for any $\pi \in \Omega$, let $\mathcal{L}(\pi)$ be the \textit{number of loops and double edges} in $\pi$.
We introduce the probability measure  $\mathbb{P}_{L,N, \rho}$ in $\Omega$,
which depends on two   parameters $\rho \in [0, \infty)$, the \textit{monomer activity,} and $N \in [0, \infty)$, the \textit{number of colours,} as follows:
\begin{equation}\label{eq:definitionmeasure}
\forall \pi \in \Omega \quad \quad  \mathbb{P}_{L,N, \rho} \big (   \pi   \big ) : = \frac{\rho^{\mathcal{M}(\pi) } \, \, (\frac{N}{2})^{ \mathcal{L}(\pi)  } }{Z_{L,  N, \rho}},
\end{equation}
where $Z_{L, N, \rho}$ is a normalisation constant.
Let $X : \Omega \rightarrow \T_L$ be the end-point of the walk,
which we call \textit{target point.} More precisely,
for any $\pi \in \Omega$, we define $X(\pi) \in \T_L$
as the unique vertex such that  $\pi \in \Omega_{o,X(\pi)}$.
It is known that, if the {monomer activity is large enough}, the length of the walk admits uniformly bounded exponential moments \cite{Betz2, Taggi}. 
This implies that the distance between the target point and the origin does not grow unboundedly with the size of the system. Our Theorem \ref{theo:theo2} below states that, 
contrary to the case of high  monomer activity,
when the monomer activity is zero, the 
distance between the target point and the origin grows with the size of the system and scales linearly with  the diameter of the box.
In other words, a phase transition takes place at a finite, possibly zero value of the monomer activity.
{
For the statements about lattice permutations, let $N_+ = \sum_{n>0} \mathbbm{1}\{ S_n = o \}$ be the number of returns to the origin of a simple random walk, $S_n$, in $\mathbb{Z}^d$, whose probability measure and expectation are denoted by $P^d$ and $E^d$, respectively. Define $r_d := E^d(N_+)$, the expected number of returns to the origin. We use $\boldsymbol{e}_i \in \mathbb{R}^d$ to denote the cartesian vectors, where $i \in \{1,\ldots,d\}$.
Define
\begin{equation}\label{eq:admissible-N}
h_d:=r_d+\sqrt{r_d(1+r_d)} \quad \quad  \mathcal N_d:=\{N\in\mathbb N_{>0}:Nh_d<2\}\cup\{1,2\}.
\end{equation}
The value $N=2$ in dimension three is covered by the separate argument
in Section~\ref{rev3d-section}.

}
{
\begin{theorem}\label{theo:theo2}
Suppose that $d>2$ and $N\in\mathcal N_d$.
There exists a constant $c_2=c_2(N,d)>0$ such that, for every sufficiently large even $L$,
\begin{equation}\label{eq:mainestimate}
 \forall A\subset\T_L,\qquad
 \mathbb P_{L,N,0}(X\in A)\le c_2\frac{|A|}{L^d}.
\end{equation}
\end{theorem}
}

For example, by choosing $A = \T_{\lfloor \epsilon L  \rfloor}$ for a small enough $\epsilon$, we see that with uniformly positive probability the target point is at a distance at least $\epsilon L$ from the origin.
The restriction of our result to not-too-large values of $N$ is not a limitation of our technique: It was shown by Chayes \textit{et al.} \cite{Chayes} that, in any dimension $d \geq 2$, if $N$ is a large enough integer, the loop length admits uniformly bounded exponential moments for any value $\rho \in [0, \infty)$ (in 
\cite{Chayes} a different setting than ours is considered,
with only loops, which are allowed to overlap a bounded number of times, and no walk;  the proof of \cite{Chayes}  can be adapted to our setting implying that  length of the self-avoiding walk does not grow unboundedly with the size of the system and admits uniformly bounded exponential moments).
Hence, not only we prove the occurrence of a phase transition with respect to the variation of $\rho$ for {
$N\in\mathcal N_d$,
} 
but we also prove the occurrence of a phase transition with respect to the variation of $N$ when we fix $\rho = 0$.

\paragraph{Uniform positivity.}
Our third main theorem, Theorem \ref{theo:theo3} below, can be viewed as a generalisation of Theorem \ref{theo:theo1} and  states that the two point function of lattice permutations is bounded away from zero point-wise  when the points lie along the same cartesian axis and `on average' across all  points, uniformly with respect to the system size.
To define the two-point function we need to introduce the set of multi-graphs $\Omega^\ell$, whose connected components are  loops, double edges or monomers and no walk is present.  Thus, let $\Omega^\ell$ be the set of directed multi-graphs $\pi = (\T_L, E_\pi)$ such that: 
\textbf{(i)} the edges connect nearest-neighbour vertices in the torus and, \textbf{(ii)} the  in- and the out-degree of every vertex in $\T_L$ are equal and their value is either zero or one.\footnote{Alternatively, $\Omega^\ell$ could be defined as the set of \textit{permutations} of the elements of $\T_L$ such that every vertex is mapped either to itself or to a nearest neighbour, the same as  in \cite{Betz2}. Here we keep the name `permutations', but we define the realisations as  multi-graphs.}
It follows from this definition that every connected
component of the graph  $\pi \in \Omega^\ell$ is either a monomer, a loop or a double edge, which we defined before.
We extend the definition of the number of monomers, $\mathcal{M}(\pi)$,  and of the number {of} loops and double edges, $\mathcal{L}(\pi)$,
which were provided before, to the graphs $\pi \in \Omega^{\ell}$.
For any $L \in \mathbb{N}$,  $\rho, N \in [0, \infty)$,  we define the \textit{loop partition function,}
\begin{equation}\label{eq:partitionloop}
\mathbb{Z}_{L,  N, \rho}^{\ell} := \sum\limits_{\pi \in \Omega^\ell} \rho^{\mathcal{M}(\pi)} (\frac{N}{2})^{  \mathcal{L}(\pi) },
\end{equation}
and, for  any $x, y \in \T_L$, we define 
the \textit{directed partition function,} 
\begin{equation}\label{eq:partitiondirected}
\mathbb{Z}_{L, N, \rho}(x,y) := \sum\limits_{\pi \in \Omega_{x,y}} \rho^{\mathcal{M}(\pi)} (\frac{N}{2})^{  \mathcal{L}(\pi) },
\end{equation}
We use $0^0=1$ in the partition sums. Whenever $\mathbb{Z}_{L,N,\rho}^{\ell}>0$ (this excludes only the parameter choice $N=\rho=0$), we define the \textit{two point function} as
\begin{equation}\label{eq:twopoint}
\mathbb{G}_{L, N, \rho}(x,y) :=     \frac{ \mathbb{Z}_{ L, N, \rho}(x,y)}{\mathbb{Z}_{L, N, \rho}^{\ell}},
\end{equation}
and we define $\mathbb{G}_{L, N, \rho}(x) : = \mathbb{G}_{L, N, \rho}(o, x)$.
In the special case of $N = 2$ and $\rho = 0$,  the two-point function of lattice permutations corresponds to the monomer-monomer correlation function of the dimer model,
\begin{equation}\label{eq:relationdimerpermutationstat}
\forall x \in \T_L \quad  \mathbb{G}_{L, 2, 0}(x) = \Xi_L(x).
\end{equation}
Indeed, as we prove in (\ref{eq:relationdimerpermutation}) below, 
the set of configurations which are obtained by superimposing two independent dimer covers, like in Figure \ref{Fig:dimerssuperposition},  
are in a one-to-one correspondence with the set of  fully-packed lattice permutations and this leads to (\ref{eq:relationdimerpermutationstat}).
In light of (\ref{eq:relationdimerpermutationstat}), our 
Theorem \ref{theo:theo3} below, which holds for
 {
$N\in\mathcal N_d$ (not necessarily equal to $2$)
}, can be viewed as a generalisation of Theorem \ref{theo:theo1}.
 {
\begin{theorem}\label{theo:theo3}
Suppose that $d>2$ and $N\in\mathcal N_d$. There exists a constant
$c_{\rm av}=c_{\rm av}(d,N)>0$ such that
\begin{equation}\label{eq:permpositivityaverage}
 \liminf_{\substack{L\to\infty\\L\text{ even}}}
 \frac1{|\T_L^o|}\sum_{x\in\T_L^o}\mathbb G_{L,N,0}(x)
 \ge c_{\rm av}.
\end{equation}
Moreover, for every $\varphi\in(0,c_{\rm av})$ there exists a constant
$c_3=c_3(\varphi,d,N)>0$ such that, for every sufficiently large even $L$
and every odd integer $n\in(0,c_3L)$,
\begin{equation}\label{eq:permpointwise}
 \mathbb G_{L,N,0}(n\boldsymbol e_i)\ge\varphi,
 \qquad i\in\{1,\ldots,d\}.
\end{equation}
\end{theorem}
}

Similarly to the case of the dimer model, from the site-monotonicity properties which were derived in \cite{Lees} we deduce that, for any integer $N \in \mathbb{N}_{>0}$ {at $\rho=0$,}
\begin{equation}\label{eq:uniformpointwisesupper}
\forall x \in \T_L \quad  {\mathbb{G}_{L, N, 0}(x)} \leq \frac{1}{d N}.
\end{equation}
{
Since $h_d=O(d^{-1/2})$, the lower bound proved in Proposition~\ref{prop:corrected-higher-d}
} 
and the uniform point-wise upper bound (\ref{eq:uniformpointwisesupper}) on the two-point function get closer to each other as $d$ is larger. {For fixed $N$, this shows that the distribution of the target point becomes more uniform across the odd sites of the torus as the dimension increases.}

\section{Proof description}
\label{sect:mainingredients}
Most of the paper is devoted to the proof of (\ref{eq:permpositivityaverage}), from which all our main results follow. 
 The {first part} is devoted to the  derivation of the Key Inequality, Theorem \ref{theo:keyinequality} below,  from the analysis of a general soup of loops and walks, to which we refer as \textit{random path model.} 
The random path model was introduced in \cite{Lees} and it is a generalisation of the random wire model  \cite{Benassi}, which,  in turn, can be viewed as a reformulation of the random walk representation of the spin O(N) model  \cite{Brydges}. In \cite{Lees} it was shown that the random path model satisfies the important property of reflection positivity (which will be stated later). The property of reflection positivity for random loop models was used also in  \cite{Chayes, Lees2, Ueltschi}.  However, in  such works
the additional structures which allow the derivation of the Key Inequality  \textit{directly} from the space of loops and walks (i.e., without employing any spin representation) have not been introduced. The most important technical novelty of this paper is the introduction of such structures. This allows the extension of the method of  {\cite{Frohlich, FrohlichI, FrohlichII}} to random loop models for which no spin representation exists or  is easy to derive, for example lattice permutations (and, consequently,  the dimer model). More precisely, our analysis involves the study of the  random path model with  appropriate weights   in an `extended' graph, which is obtained from the original torus by adding `virtual' vertices on the `top'  of each  vertex of the `original'  torus; such virtual vertices serve as a source for the walks, and the walks  get a different weight depending on where they start from; the whole setting is designed in such a way that the reflection positivity property, which was proved to hold true in the torus \cite{Lees}, is preserved.

The second part is devoted to the derivation of a version of the so-called infrared bound from such a Key Inequality.  Here we use Fourier transforms  similarly to the case of   spin systems with continuous symmetry  \cite{Frohlich, FrohlichI, FrohlichII}, in which case the two point function corresponds to the correlation between two spins. Our analysis differs from such a classical case  for some non-trivial aspects. The most important difference is that, in our case, the two-point function vanishes at any even site as $\rho \rightarrow 0$. In other words, the model exhibits a sort of anti-ferromagnetic ordering, similarly to  \cite{Feldheim}. {
The proof first uses the parity symmetry of the Fourier transform. 
For $d>3$ this gives the required estimate after averaging over coordinate directions. 

In the special case  $d=3$, $N=2$ the key inequality does not suffice. 
We then use \blu{a} different argument based on the transfer matrix for dimer configurations.
This argument is presented in Section~\ref{rev3d-section}.
}

We now describe the two parts of the proof in greater detail and state Theorem \ref{theo:keyinequality} and Lemma \ref{lemma:relationmodus}. In the third and last subsection, {we deduce Theorem \ref{theo:theo1} from the results for lattice permutations.}

\subsection{Description of part I: Derivation of the Key Inequality}
\label{sect:ingredientspart1}
The first part of the proof, which is presented in Section \ref{sect:therandom}, is devoted to the derivation of Theorem \ref{theo:keyinequality}, which is stated below.
For an arbitrary vector of real numbers, $\boldsymbol{v} = (v_z)_{z \in \T_L}$,
define the \textit{discrete Laplacian} of $\boldsymbol{v}$,
\begin{equation}\label{eq:discretelaplacian}
\forall x \in \T_L \quad \quad (\triangle v)_x : = \sum\limits_{\substack{ y \in \T_L : \\ x \sim y}} ( v_y - v_x ).
\end{equation}
\begin{theorem}[\textbf{Key Inequality}]\label{theo:keyinequality}
For any $N \in \mathbb{N}_{>0}$, 
$\rho \in \mathbb{R}_{ \geq 0}$,
 $L \in 2 \mathbb{N}_{>0}$,  any real-valued vector
 $\boldsymbol{v} = (v_x)_{x \in \T_L}$, we have that,
\begin{equation}\label{eq:starting point infrared}
\sum\limits_{x,y \in \T_L} \mathbb{G}_{L, N, \rho}(x,y) (\triangle v )_x \, 
(\triangle v)_{y} \,  
\leq  \, \sum\limits_{\{x,y\} \in \E_L} \big (v_y - v_x \big )^2.
\end{equation}
\end{theorem}
The proof of Theorem \ref{theo:keyinequality} uses several ingredients
which we now describe informally. We deal with the random path model, namely a  probabilistic model of coloured closed and open paths,
which interact at sites through a weight function,
which will be denoted by $U$. 
The model depends on an edge parameter 
$\lambda  \in [0,\infty)$ which, informally, has the effect of increasing the typical length of the paths  as $\lambda$ is larger.

We introduce a new setting which  is reminiscent of the random current representation of the Ising model \cite{ProceedingsCopin}. Such a setting involves the random path model on a graph 
 $(\mathcal{T}_L, \mathcal{E}_L)$, which is obtained from the torus  $(\T_L, \E_L)$ by adding a new vertex (which will be referred to as \textit{virtual}) on the top of each vertex  in $\T_L$ and by connecting such a new vertex to the one which is below it by an edge, like in Figure \ref{Fig:exampledoubletorus}. 
We refer to such a new graph $(\mathcal{T}_L, \mathcal{E}_L)$ as \textit{extended torus} and to the graph $(\T_L, \E_L) \subset (\mathcal{T}_L, \mathcal{E}_L)$, which was defined previously, as  \textit{original torus}.
Virtual vertices play the role of sources for open paths and closed paths  are not allowed to `touch'  any virtual vertex.
Such a  setting is designed in such a way that the measure associated to the random path model on such a graph satisfies two fundamental properties \textit{at the same time}. The first fundamental property is  \textit{reflection positivity}. 
The second fundamental property involves a central quantity, $\mathcal{Z}_{L, N, \lambda, U}(\boldsymbol{v})$, where $\boldsymbol{v} = (v_z)_{z \in \mathbb{T}_L}$ is an array of real numbers, with each number being associated to a vertex of the original torus. {In this overview, $\mathcal Z_{L,N,\lambda,U}(\boldsymbol v)$ abbreviates $\mathcal Z_{L,N,\lambda,U}(\boldsymbol h^{\boldsymbol v})$, where $h_z^{\boldsymbol v}=v_z$ and $h_{z+\boldsymbol e_{d+1}}^{\boldsymbol v}=-2dv_z$ for $z\in\mathbb T_L$, as in (\ref{eq:hasfunctionofv}).}
The quantity $\mathcal{Z}_{L, N, \lambda, H}(\boldsymbol{v})$ is defined as the  average of a function which assigns a multiplicative weight $v_z$ every time that a walk starts (or ends) at a vertex of the original torus $z \in \T_L$
and a multiplicative weight $ - 2 d v_z$ every time that a walk starts (or ends) at the virtual vertex  which is `on the top' of $z\in \T_L$.
Such a fundamental property is stated in  (\ref{eq:expansion1}) below and involves the infinitesimal variation of the function  $\mathcal{Z}_{L, N, \lambda, U}(\boldsymbol{v})$ around the point  $\boldsymbol{v} = 0$
when a specific choice of the weight function, $U = H$, is made.
More precisely, for an arbitrary choice of $\boldsymbol{v} \in \mathbb{R}^{\T_L}$ and  $\varphi \in \mathbb{R}$,  in the limit as $\varphi \rightarrow 0$
\begin{multline}\label{eq:expansion1}
\mathcal{Z}_{L, N, \lambda, 
H}(\varphi \boldsymbol{v}) =
 \lambda^{   |\T_L|  }\, \,  \mathbb{Z}^\ell _{L, N, \frac{1}{\lambda}}
- \varphi^2 \,  \frac{\lambda N}{2} \,  \lambda^{   |\T_L|  }\, \, 
\mathbb{Z}^\ell_{L, N, \frac{1}{\lambda}} \sum\limits_{\{x,y\} \in \E_L} \big (v_y - v_x \big )^2
\, \, \\     \,  + \varphi^2 \,  \frac{\lambda N}{2} \, \lambda^{   |\T_L|  }\sum\limits_{x,y \in \T_L} \mathbb{Z}_{L, N, \frac{1}{\lambda}}(x,y) (\triangle v )_x \, 
(\triangle v)_{y} \, + o(\varphi^2).
\end{multline}
To derive (\ref{eq:expansion1}) we introduce a map which maps configurations of the random path model to configurations of lattice permutations and compare their weights.
Here we use in an essential way the structure of the extended torus: the walks which enter into the original torus 
from a virtual vertex are weighted differently than the walks which start from a vertex of the original torus and the weights are chosen appropriately so that we get the discrete Laplacians and the sum involving factors  $(v_y - v_x)^2$ 
in (\ref{eq:expansion1}).
Also the properties of the random path model and of the weight function $H$, which allows the walk to be vertex-self-avoiding at every vertex except for its end-points, are used in an essential way.
We refer to this central step of the proof as \textit{Polynomial expansion.}
The reason why the expansion (\ref{eq:expansion1}) is so important is that it is possible to deduce the Key Inequality by showing that, for any vector $\boldsymbol{v} \in \mathbb{R}^{\T_L}$, the term of order $O(\varphi^2)$ in (\ref{eq:expansion1})
is non-positive.
Indeed, the reader can verify that, from the non-positivity of the term  of order $O(\varphi^2)$ and  from the definition of two-point function, (\ref{eq:twopoint}), Theorem \ref{theo:keyinequality} follows immediately
after dividing {the quadratic coefficient} by $\frac{\lambda N}{2} \lambda^{   |\T_L|  } \mathbb{Z}_{L, N, 1 / \lambda}^\ell$.

{To prove this non-positivity of the quadratic coefficient at $\boldsymbol v=0$, we use an iterative sequence of reflections,}
which leads by reflection positivity to the  \textit{Chessboard estimate},
\begin{equation}\label{eq:gaussiansummary}
\mathcal{Z}_{L, N, \lambda, U}(\boldsymbol{v})  \, \leq \Big (  \prod_{x \in \T_L}       \mathcal{Z}_{L, N, \lambda, U}(\boldsymbol{v}^x) \Big )^{ \frac{1}{| \T_L| }  },
\end{equation}
where $\boldsymbol{v}^x =  (v^x_z)_{z \in \T_L}$ is a vector which is obtained from $\boldsymbol{v} := (v_z)_{z \in \T_L}$ by copying the value $v_x$ at each original vertex.
Since for each $x \in \T_L$, the term of order  $O(\varphi^2)$ 
is zero when we look at the   vectors $\boldsymbol{v}^x$, i.e., 
\begin{equation}\label{eq:expansion2}
\mathcal{Z}_{L, N, \lambda, H}(\varphi \boldsymbol{v}^x)  = \lambda^{  | \T_L|  } 
\, \, \mathbb{Z}^\ell _{L, N, \frac{1}{\lambda}} + o(\varphi^2),
\end{equation}
we deduce from 
(\ref{eq:expansion1}),  (\ref{eq:gaussiansummary}), and  (\ref{eq:expansion2})
and from a  Taylor expansion of the root in (\ref{eq:gaussiansummary}) that 
 \textit{the term of order $O(\varphi^2)$ in (\ref{eq:expansion1}) is non-positive.}

 \begin{figure}
  \centering
    \includegraphics[width=0.60\textwidth]{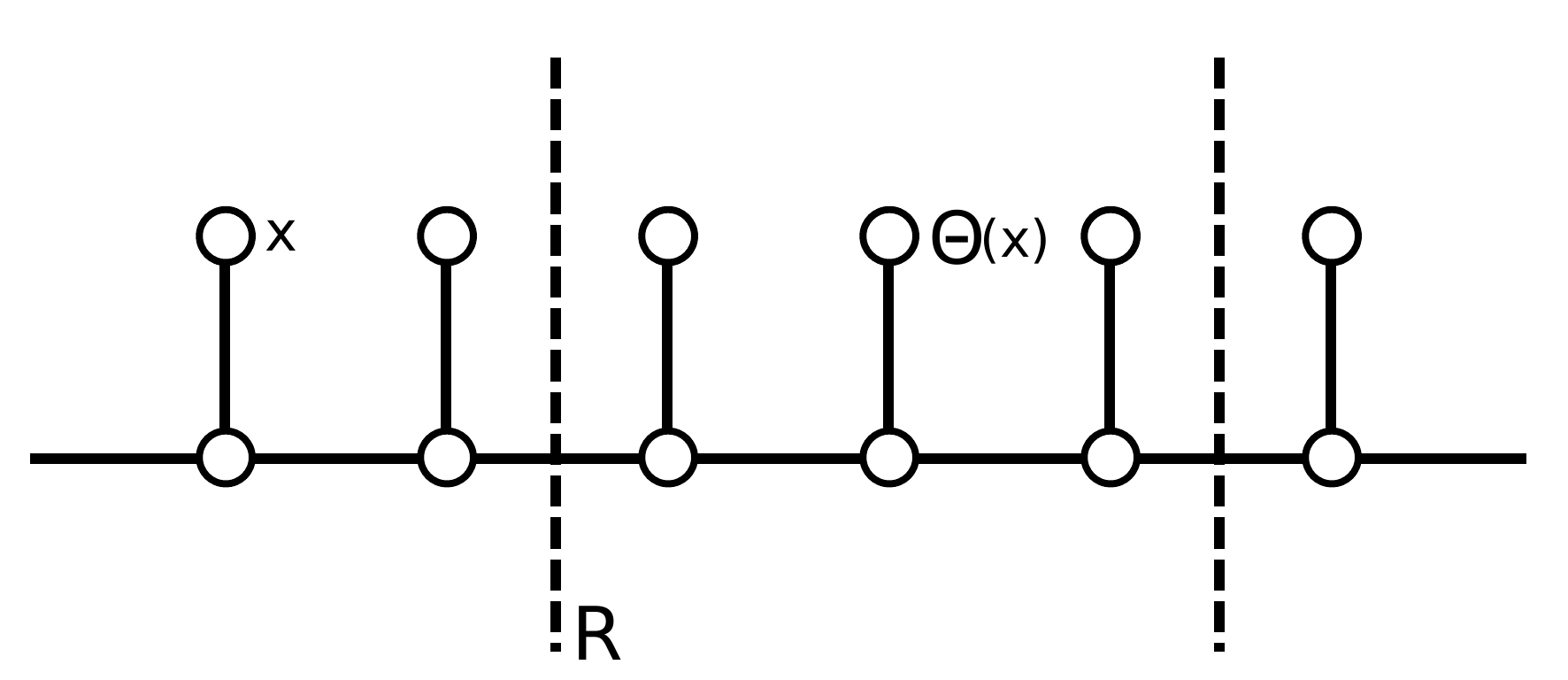}
      \caption{An \textit{extended torus} when  $d = 1$ and $L=6$. The left-most and the right-most horizontal edges are identified.
            The leftmost vertical dashed line represents a reflection plane, $R$, which, for example, maps the vertex $x$ to $\Theta(x)$. In the figure  $x$ is a virtual vertex, while the one which is `below it' is original.}\label{Fig:exampledoubletorus}
\end{figure}

 \begin{remark}\label{rem:extension}
The random path model, which depends on an arbitrary weight function $U$,
is related by the  expansion (\ref{eq:expansion1}) to lattice permutations when a specific choice for $U$ is made.
Our method can be adapted to any weight function $U$ satisfying the general assumptions in Definitions \ref{def:measure} and \ref{def:invariantunderreflection} below.
For example, there exists a special choice of the weight function $U$ which satisfies such assumptions such that the random path model is a \textit{representation} of the spin O(N) model \cite{Benassi, Lees} and our method can be used to derive   the famous result of Fr\"ohlich, Simon and Spencer \cite{Frohlich}, which involves the spin O(N) model, directly from its representation as a random loop model.   
Our method can also be adapted to random path models with weight function $U$ for which no spin representation is known, for example lattice permutations and the dimer model, and it can thus be viewed as an extension of \cite{Frohlich}.

\end{remark}

\subsection{Description of part II: Derivation of a version of the  Infrared bound}
\label{sect:ingredientspart2}
We now give a brief overview to the second part of the proof, which is presented in Section \ref{sect:InfraredBound} and uses Fourier transforms.
To begin, we define the  \textit{dual torus},
$$
\T_L^* := \big  \{ \frac{ 2 \pi }{L} \, (n_1, \ldots, n_d) \in  \mathbb{R}^d \, \, : \, \, 
n_i \in (- \frac{L}{2}, \frac{L}{2}] \cap  \mathbb{Z}  \, \big \}.
$$
We  denote the elements of $\T_L^*$ by $k = (k_1, \ldots, k_d)$
and we  keep using the notation $o$ for $(0, \ldots 0) \in \T_L $
or $(0, \ldots, 0) \in \T_L^*$. 
Given a function $f \in \ell^2(\T_L)$, we define its Fourier transform,
\begin{equation}\label{eq:Fourierdefinition}
\forall k \in \T_L^*, \quad \hat{f}(k) : = \sum\limits_{x \in \T_L} e^{- i k \cdot x} f(x).
\end{equation}
It follows from this definition that,
\begin{equation}\label{eq:InverseFourier}
\forall x \in \T_L, \quad f(x) = \frac{1}{|\T_L|} \sum\limits_{k \in \T_L^*} e^{i k \cdot x} \hat{f}(k).
\end{equation}
{The next lemma follows directly from Fourier inversion and allows us to explain the strategy of the proof.} 
\begin{lemma}\label{lemma:relationmodus}
Define the Fourier mode $p : = (\pi, \pi, \ldots, \pi) \, \in  \, \T_L^*$.
We have that, for any {even $L\geq4$ and $\rho,N\in[0,\infty)$ with $\mathbb Z^\ell_{L,N,\rho}>0$,}
\begin{equation}\label{eq:fourierlemma}
\frac{2}{|\T_L|} \sum\limits_{x \in \T^o_L} \mathbb{G}_{L,  N, \rho}(x)
=  \mathbb{G}_{L,  N, \rho}(\boldsymbol{e}_1) \, - \, \frac{1}{|\T_L|} \, \sum\limits_{k \in \T_L^* \setminus \{o, p\}} e^{  i k \cdot \boldsymbol{e}_1}  \, \hat{\mathbb{G}}_{L,  N, \rho}(k).
\end{equation}
\end{lemma}
{
\begin{proof}
Omitting subscripts, the definition of the Fourier transform gives
\[
 \widehat{\mathbb G}(o)-\widehat{\mathbb G}(p)
 =2\sum_{x\in\mathbb T_L^o}\mathbb G(x).
\]
Separating the modes $o,p$ in Fourier inversion at $\boldsymbol e_1$
and using $e^{ip\cdot\boldsymbol e_1}=-1$ proves the claim.
\end{proof}
}
The  goal is to bound away from zero uniformly in $L$ the quantity in the left-hand side of  (\ref{eq:fourierlemma}), obtaining  (\ref{eq:permpositivityaverage}). This quantity corresponds to the difference between the $(0,\ldots, 0)$ and the $(\pi, \ldots, \pi)$ Fourier mode of the two point function (note that the sum involves only odd vertices).
When $\rho = 0$, the first term in the right-hand side of (\ref{eq:fourierlemma})  satisfies
\begin{equation}\label{eq:boundfirstterm}
\mathbb{G}_{L,  N, 0}(\boldsymbol{e}_1) = \frac{1}{d N}
\end{equation}
for any even $L$, as we prove in Section \ref{sect:proofs} (and it is easy to show). 
{
Section~\ref{sect:InfraredBound} first derives the high-frequency bound from
the Key Inequality. Averaging over coordinate directions and pairing
Fourier modes differing by $p$ then gives
\begin{equation}\label{eq:boundsecondterm}
 \limsup_{\substack{L\to\infty\\L\text{ even}}}
 \frac1{|\T_L|}\sum_{k\in\T_L^*\setminus\{o,p\}}
 e^{ik\cdot\boldsymbol e_1}\widehat{\mathbb G}_{L,N,0}(k)
 \le\frac{h_d}{2d}.
\end{equation}
Together with \eqref{eq:boundfirstterm}, this proves positivity when
$Nh_d<2$, in particular for $N=2$ in every dimension $d>3$.

\bis{For $d=3$, $N=2$, Section~\ref{rev3d-section} derives the improved bound
$\widehat\Xi_L(k)\le 2/[3\varepsilon(k)]$ for $k\ne o$ from a transfer
matrix, without using the Key Inequality. This step was not present in the
published version~\cite{TaggiCPAM}. The factor $2/3$ suffices to obtain
positivity by the same pairing of Fourier modes.}

}
\subsection{From lattice permutations to dimers: proof of Theorem \ref{theo:theo1}}
We now prove (\ref{eq:relationdimerpermutationstat}) formally. 
{
The dimer result follows from its representation as a fully-packed lattice
permutation model with $N=2$. We shall return to the matching language in
Section~\ref{rev3d-section} for the additional three-dimensional estimate.

} In such a special case, lattice permutations can be viewed as a different formulation of the \textit{double dimer model} {\cite{Dubedat2, Kenyon5}.}
 Here, by `fully-packed' $\pi$ we mean that  $\pi$ is such that  $\mathcal{M}(\pi) = 0$. 

\begin{proof}[{{\textit{Proof of (\ref{eq:relationdimerpermutationstat})}.}}]
We claim that there exist two bijections,
\begin{align*}
\Pi^{1}  : \mathcal{D}(\emptyset) \times \mathcal{D}(\{o,z\})  & \mapsto   \{  \pi \in \Omega_{o,z} \, \, : \, \,\mathcal{M}(\pi) = 0 \} \\
\Pi^{2}  : \mathcal{D}(\emptyset) \times \mathcal{D}(\emptyset)  & \mapsto   \{  \pi \in \Omega^\ell \, \, : \, \,\mathcal{M}(\pi) = 0  \}.
\end{align*}
Indeed, note the following: If we superimpose two dimer covers, $\eta^1 \in \mathcal{D}(\emptyset)$, $\eta^2 \in \mathcal{D}(\{o,z\})$, which we call blue and red respectively,
we obtain a system of mutually-disjoint self-avoiding loops, double dimers and a self-avoiding walk from $o$ to $z$, like in Figure \ref{Fig:dimerssuperposition}, where the double dimer corresponds to the superposition of a blue and a red dimer on the same edge, while the loops and  walk consist of an alternation of blue and red dimers.
Note also that any loop might appear with two different alternations of blue and red dimers. Indeed, given a pair $(\eta_1, \eta_2)$ and 
some arbitrary loops of such a pair, one might obtain a new pair $(\eta^\prime_1, \eta^\prime_2)$ which is identical to $(\eta_1, \eta_2)$, except for the fact that the selected loops appear with the opposite alternation of blue and red dimers.
Thus, we can associate to $(\eta_1, \eta_2)$ an element $\pi \in \Omega_{o,z}$ which is such that $\pi$ has a double edge at $\{x,y\}$ if both $\eta_1$ and $\eta_2$ have  dimer at $\{x,y\}$ and every loop of 
$\pi$ corresponds to a loop of $(\eta_1, \eta_2)$ and fix a convention for which alternation of red and blue dimers of the loops in 
$(\eta_1, \eta_2)$ corresponds to which of the two  possible orientations of the loops in $\pi$. This defines the bijection $\Pi^1$. The bijection $\Pi^2$ is defined analogously (the only difference is that we have no walk starting at $o$ and ending at $z$).
Since we have two bijections, we deduce that
\begin{equation}\label{eq:relationdimerpermutation}
\forall z\in \T_L \quad  \mathbb{G}_{L, 2, 0}(o,z) =
 \frac{|\{ \pi \in \Omega_{o,z} \, : \, \mathcal{M}(\pi) = 0\}|}{  |\{ \pi \in \Omega^\ell \, : \, \mathcal{M}(\pi) = 0\}|} = 
 \frac{|\mathcal{D}(\{o,z\})| \, \, |\mathcal{D}(\emptyset)|}{ | \mathcal{D}(\emptyset)|^2} = \Xi_L(z),
\end{equation}
This leads to our claim.
\end{proof}

\begin{proof}[{\textit{Proof of Theorem \ref{theo:theo1}}}]
{
Let $c_{\rm dim}(d)$ be the lower limit on the left-hand side of
\eqref{eq:dimerpositivityaverage}. By Theorem~\ref{theo:theo3} with $N=2$
and \eqref{eq:relationdimerpermutationstat}, this constant is strictly positive.
For $d\ge4$, Proposition~\ref{prop:corrected-higher-d} and
\eqref{eq:uniformpointwisesupper} give
\[
 \frac{1-h_d}{2d}\le c_{\rm dim}(d)\le\frac1{2d}.
\]
Since $h_d=O(d^{-1/2})$, this proves the stated large-dimension asymptotic.
The site-monotonicity argument in Section~\ref{sect:proofs}, proving
\eqref{eq:permpointwise}, applies to any positive lower bound for the
lower limit of the spatial average. Applying it with $N=2$ and
$c_{\rm dim}(d)$ in place of $c_{\rm av}$, and using
\eqref{eq:relationdimerpermutationstat}, proves \eqref{eq:dimerpointwise}.
}
\end{proof}

\section*{Notation}

\begin{center}
	\resizebox{\textwidth}{!}{\begin{tabular}{ l l }
	
	   $ \boldsymbol{e}_i$  & cartesian vector, with $i \in \{1, \ldots d\}$ or   $i \in \{1, \ldots d+1\}$ \\ 

$\mathcal{G} = ( \mathcal{V}, \mathcal{E})$ &  an undirected, simple, finite  graph \\

$e \in \mathcal{E}$ or $\{x,y\} \in \mathcal{E}$ & undirected edges \\

$(x,y) \in \mathcal{E}$ &  edge directed from $x$ to $y$ \\

 $ ( \T_L, \E_L)$ & graph corresponding to the torus $\mathbb{Z}^d / L \mathbb{Z}^d$   \\
 
  $ ( \mathcal{T}_L, \mathcal{E}_L)$ & extended torus, with original and virtual vertices \\
  
      $ \T_L^{ (2)} \subset  \mathcal{T}_L$ & set of virtual vertices  \\
      
         $ \T_L^{*}$  & Fourier dual torus   \\

  $ o \in \T_L$, $o \in \mathcal{T}_L$  or $o \in \T_L^*$ & origin \\ 

    $x \sim y$ & pair of vertices in $\T_L$ which are connected by an edge in $ \E_L$    \\

 $N \in \mathbb{N}_{>0}, \, \, \lambda, \rho \in \mathbb{R}_{\geq 0}$& respectively number of colours, edge-parameter, and monomer activity  \\

$U = (U_x)_{x \in \mathcal{V}}$ &  {weight function} \\

$m = (m_e)_{e \in \mathcal{E}}$ & link cardinalities, with  $m_e$ {corresponding to the number of links} \textit{on} the edge $e$  \\

$c = (c_e)_{e \in \mathcal{E}}$ & link colourings, with  $c_e : \{1, \ldots, m_e\} \mapsto \{1, \ldots, N\} $ \\ 

$\gamma = (\gamma_x)_{x \in \mathcal{V}}$ & pairings, with $\gamma_x$ pairing the links touching the vertex $x$ \\

$\mathcal{W}_{\mathcal{G}}$ &  the set of  configurations in $\mathcal{G}$, with $w= (m,c, \gamma) \in {\mathcal W_{\mathcal G}}$ \\

$n_x$ & number of pairings at $x$ \\

$u_x$ & number of links touching  $x$ which are unpaired at $x$ \\

$\mathbb{Z}^\ell_{L,  N, \rho  }$ 
 & loop partition function\\

$\mathbb{Y}^\ell_{L,  N, \lambda  }$ 
 & loop partition function times an appropriate constant \\
 
 $\mathbb{Z}_{L,  N, \rho  }(x,y)$ 
 & directed partition function \\

$\mathbb{Y}_{L,  N, \lambda  }(x,y)$ 
 & directed partition function times an appropriate constant \\
 
 $\mathbb{G}_{L,  N, \rho  }(x,y)$ 
 & two-point function \\
 
  $\mathbb{G}_{L,  N, \rho  }(x)$ 
 & equivalent to $\mathbb{G}_{L,  N, \rho  }(o,x)$   \\
 
  $\hat{\mathbb{G}}_{L,  N, \rho  }(k)$ 
 & Fourier transform of  $\mathbb{G}_{L,  N, \rho  }(x)$  \\

$\boldsymbol{v} = (v_x)_{x \in \T_L}$ &  real-valued vector, with  coordinates associated to $\T_L$ \\

$\boldsymbol{h} = (h_x)_{x \in \mathcal{T}_L}$ & vector of real numbers, with  coordinates associated to  $\mathcal{T}_L$ \\

  $\mathcal{Z}_{L, N, \lambda, U}(\boldsymbol{h})$ 
 & partition function with  links unpaired at $x$ receiving a multiplicative weight $h_x$    \\
 
  $\mathcal{Z}^{(2)}_{L, N, \lambda, U}(\boldsymbol{h})$ 
 & second term of the polynomial expansion   \\

\end{tabular}}

\end{center}

\section{Derivation of the Key Inequality}
\label{sect:therandom}
This section is devoted to the proof of  Theorem \ref{theo:keyinequality}.
Before starting, it will be convenient introducing a different parametrisation of the partition functions. More precisely,  let  $x, y \in \T_L$ be arbitrary vertices, 
for any $\pi \in \Omega^\ell$ or  $\pi \in \Omega_{x,y}$, define $\mathcal{H}(\pi) : = \big |  E_\pi|$,  the number of directed edges in the graph  $\pi = (\T_L, E_\pi)$.
Define the \textit{edge-parameter} $\lambda \geq 0$ and define the partition functions parametrised by $\lambda$, 
\begin{equation}\label{eq:partitionlambda}
\mathbb{Y}_{L,  N, \lambda}^{\ell} := \sum\limits_{\pi \in \Omega^\ell} \lambda^{\mathcal{H}(\pi)} (\frac{N}{2})^{  \mathcal{L}(\pi) }, \quad \quad  \quad \mathbb{Y}_{L, N, \lambda}(x,y) := \sum\limits_{\pi \in \Omega_{x,y}} \lambda^{\mathcal{H}(\pi)} (\frac{N}{2})^{  \mathcal{L}(\pi) },
\end{equation}
which for any  $\lambda \in (0, \infty)$ and $L \in 2 \mathbb{N}$ are related to the partition functions (\ref{eq:partitionloop}) and
(\ref{eq:partitiondirected})
by 
$$
\mathbb{Y}_{L,  N, \lambda}^{\ell}  = 
\lambda^{ |\T_L|  } \, \, 
\mathbb{Z}_{L,  N, \frac{1}{\lambda}}^{\ell},  
\quad   \quad 
\mathbb{Y}_{L, N, \lambda}(x,y) = 
\lambda^{ |\T_L| -1 } \,  \, \mathbb{Z}_{L, N,\frac{1}{\lambda}}(x,y),
$$
(for this, we use that $\mathcal{H}(\pi) + \mathcal{M}(\pi) = |\T_L|$ if $\pi \in 
\Omega^\ell$  and that $\mathcal{H}(\pi) + \mathcal{M}(\pi) = |\T_L| - 1$
if $\pi \in \Omega${, also for the degenerate marked walk by the convention above})
 and thus satisfy for any $\lambda \in (0, \infty)$,
\begin{equation}\label{eq:twopointlambdarelation}
\mathbb{G}_{L, N, \frac{1}{\lambda}}(x,y)   = 
 \frac{\lambda \,  \mathbb{Y}_{ L, N, \lambda}(x,y)}{\mathbb{Y}_{L, N, \lambda}^\ell}.
\end{equation}
The edge parameter $\lambda$ will play a similar role to the inverse temperature in spin systems.

\subsection{The random path model}
\label{sect:definitionsrandompath}
In this section we introduce the random path model in an arbitrary graph (this section is similar to Section 2.1 in \cite{Lees}).
Let $\mathcal{G} = ( \mathcal{V}, \mathcal{E})$ be an undirected, simple, finite graph, and assume that $N \in \mathbb{N}_{>0}$. We  refer to  $N$ as the \textit{{number of colours}.}
A  realisation of the random path model can be viewed as a collection of undirected paths (which might be closed or open).
\begin{figure}
\includegraphics[scale=0.40]{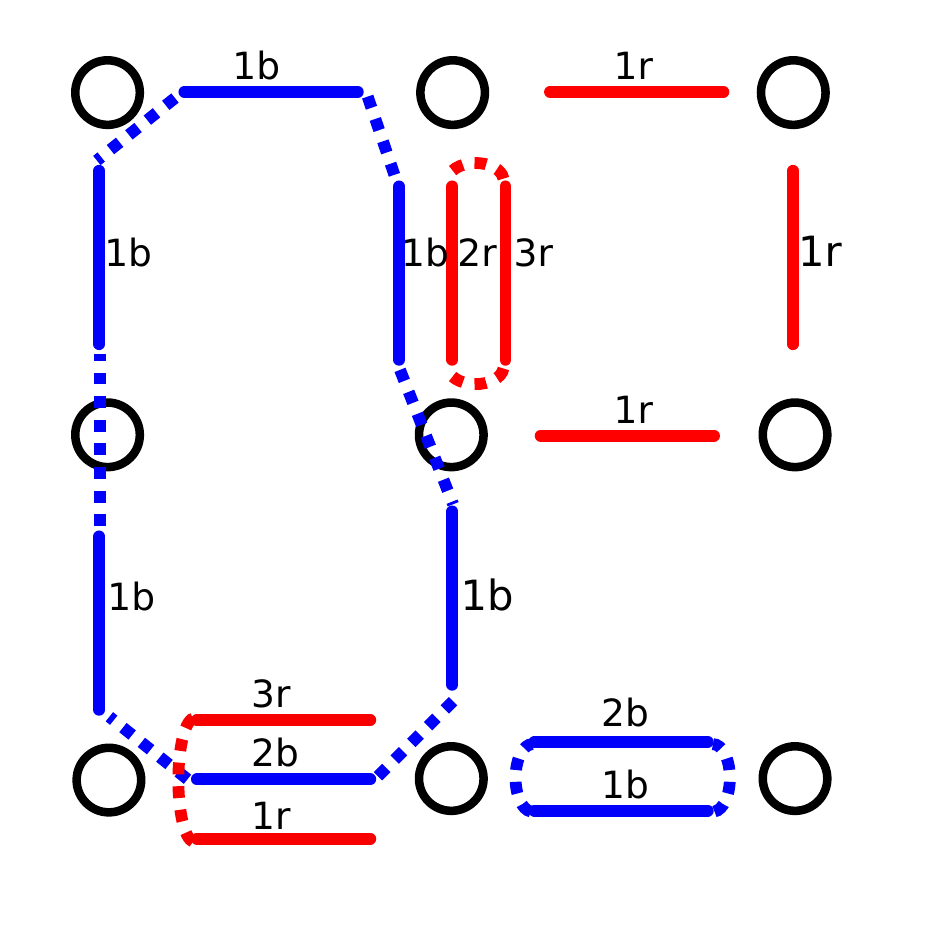}
\centering
\caption{
A configuration $w = (m, c, \gamma) \in \mathcal{W}_{\mathcal{G}}$, where $\mathcal{G}$ corresponds to the graph  $\{1, 2,  3\} \times \{1,2,3\}$ with edges connecting nearest neighbours and the lowest leftmost vertex corresponds to $(1,1)$.
On every edge $e$, the links are ordered and receive a label from $1$ to $m_e$.
In the figure, the numbers 1, 2, ...  are used for the identification of the links
and the letters $b$ and $r$ are used for the colours {which are assigned to the links} by $c$ (we assume that $N=2$ and that each link might be either blue or red).  Paired links are connected by a dotted line.
For example, the first link on the edge connecting the vertices $(1,1)$, $(2,1)$
is coloured by red and it is paired at $(1,1)$ with the third link on the same edge and it is unpaired at $(2,1)$. Moreover, both links touching the vertex $(3,3)$ are red and they are unpaired at $(3,3)$.
Finally, no link is on the edge which connects the vertices $(1,2)$ and $(2,2)$. 
}
\label{Fig:pairingexample}
\end{figure}

\paragraph{Links, colourings, pairings.}
To define a realisation we need to introduce {links}, {colourings} and pairings.
We represent a \textit{link configuration} by
 $m \in  \mathcal{M}_{\mathcal{G}} := \mathbb{N}^{\mathcal{E}}$. More specifically
$$m = \big ( m_e \big )_{e \in \mathcal{E}},$$
where $m_{e}\in\N$ represents the number of links \textit{on the edge} $e$.
Intuitively, a link represents a `visit' at the edge from a path.
The links are ordered and receive a label between $1$ and $m_e$.
See also Figure \ref{Fig:pairingexample}.
No constraint concerning the parity of $m_e$ is introduced.
If a link is on the edge $e = \{x,y\}$, then we  say that \textit{it touches} $x$ and $y$.

Given a link configuration $m \in \mathcal{M}_{\mathcal{G}}$, a \textit{colouring} $c \in \mathcal{C}_{\mathcal{G}}(m) := \{1, \ldots, N\}^{m}$ is a realisation which assigns an integer in $\{1, \ldots, N\}$ to each link, which will be called its \textit{colour}.
More precisely, 
$$
c = (c_e)_{e \in \mathcal{E}},
$$
is such that $c_e \in \{1, \ldots, N\}^{m_e}$, where $c_e(p) \in \{1, \ldots, N\}$ is the colour of the $p$-th {link which is on the edge} $e \in \mathcal{E}$, with $p \in \{1, \ldots, m_e\}$. 
See Figure \ref{Fig:pairingexample} for an example, where  $N=2$ and the colours are represented by a label in $\{r,b\}$.

Given a link configuration, $m \in \mathcal{M}_{\mathcal{G}}$, and a colouring $c \in \mathcal{C}_{\mathcal{G}}(m)$, a pairing $ \gamma = (\gamma_x)_{x \in \mathcal{V}}$ for $m$ and $c$ 
pairs links touching $x$  in such a way that,  if two links are  paired, then they have the same colour. A link touching $x$ can be paired to at most another link touching $x$, and
it is not necessarily the case that all links touching $x$ are paired to another link at $x$.
If a link touching $x$ is  paired at $x$ to no other link touching $x$, then we say that the link is \textit{unpaired at $x$.}
Given two links, if there exists a vertex $x$ such that such links are \textit{paired at $x$}, then we say that such links are \textit{paired}. 
It follows from these definitions that a link can be paired to at most two other links. We remark that, by definition, a link cannot be paired to itself. We denote by $\mathcal{P}_{\mathcal{G}}(m, c)$ the set of all such pairings for $m\in\mathcal{M}_\Gcal$, $c \in \mathcal{C}_{\mathcal{G}}(m)$. 

A \textit{configuration}  of the random path model is an element $w = ( m, c, \gamma)$
such that $m \in \mathcal{M}_{\mathcal{G}}$,
$c \in \mathcal{C}_{\mathcal{G}}(m)$,  
$\gamma \in \mathcal{P}_{\mathcal{G}}(m, c)$.
We let $\mathcal{W}_{\mathcal{G}}$ be the set of  such configurations. 
It follows from these definitions that any $w \in \mathcal{W}_{\mathcal{G}}$ can be viewed as a collection of closed and open paths.
These will be  defined in Section \ref{sect:relation} formally, and will be divided into four classes:  $\ell$-loops, double links, $\ell$-walks, and segments.

For any $w = (m, c, \gamma) \in \mathcal{W}_\mathcal{G}$, we use the notation  $m_{e}(w)$ for the random variable corresponding to the number of links on the edge $e$, i.e., the element of the vector  $m = (  m_{\tilde e} )_{\tilde e \in \mathcal{E}}$ such that  $\tilde e = e$.
For any $x \in \mathcal{V}$, let $u_x : \mathcal{W}_{\mathcal{G}} \mapsto \mathbb{N}$ be the number of links touching $x$ which are unpaired at $x$.
Moreover, let $n_x : \mathcal{W}_{\mathcal{G}} \mapsto \mathbb{N}$ be the number of pairings at $x$, namely
\begin{equation}\label{eq:numberofpairings}
n_x(w) : = \frac{1}{2}  \sum\limits_{ \substack{ (y,z)  \in \mathcal{E} : \\ y =x  }} m_{ \{y,z\}}(w) \, \,
- \, \, \frac{u_x(w)}{2}.
\end{equation}
which corresponds to the number of pairings at $x$ (i.e., the  number of links touching $x$ and paired at $x$  to another link  divided by two).

\paragraph{Domains, restrictions, measure.}
We now introduce the notion of domain and restriction and, after that, we introduce reflections. Intuitively, a function with domain $D \subset \mathcal{V}$ is a function which depends only on how $w \in \mathcal{W}_{\mathcal{G}}$ looks in $D$ or in a subset of $D$. More precisely, the function might only  depend on how many links are emanated from the vertices of $D$, on the direction in which  they are emanated, on which colour they have and on the pairings on vertices in $D$.
A function $f : \mathcal{W}_{\mathcal{G}} \mapsto \mathbb{R}$ has \textit{domain} $D \subset \mathcal{V}$ if, for any
pair of configurations $w = (m, c, \gamma), w^{\prime} = (m^{\prime}, c^{\prime}, \gamma^{\prime}) \in \mathcal{W}_{\mathcal{G}}$ such that 
$$
\quad \forall e \in \mathcal{E} :  \,e\cap D\neq \emptyset, \quad \forall z \in D, \quad m_e = m_e^{\prime} \quad c_e = c_e^{\prime} \quad \gamma_z = \gamma_z^{\prime} 
$$
one has that 
$f(w) = f(w^{\prime})$.
Moreover, for any  $w=(m, c, \gamma) \in \mathcal{W}_{\mathcal{G}}$ define the \emph{restriction} of $w$ to $D \subset \mathcal{V}$, $w_D=(m_D, c_D, \gamma_D)$ with 
$c_D \in \mathcal{C}_{\mathcal{G}}(m_D)$,
$\gamma_D\in \mathcal{P}_{\mathcal{G}}(m_D, c_D)$, by 
  \vspace{-0.1cm}
 \begin{enumerate}[i)]
  \vspace{-0.1cm}
 \item ${(m_D)_e=m_e}$ for any edge $e \in \mathcal{E}$ which has at least one end-point in $D$ and ${(m_D)_e=0}$ otherwise,
  \vspace{-0.1cm}
  \item $(c_D)_e = c_e$ for any edge $e$ which has at least one end-point in $D$ and $(c_D)_e= \emptyset$ otherwise,
 \vspace{-0.1cm}
  \item $(\gamma_D)_x = \gamma_x$ for any $x \in D$, and for $x\in  \mathcal{V} \setminus D$ we set $(\gamma_D)_x$ as the pairing which leaves all links touching $x$ unpaired (if any).
   \vspace{-0.1cm}
\end{enumerate}

We now introduce a measure  on $\mathcal{W}_{\mathcal{G}}$.
\begin{definition}\label{def:measure}
Let $N \in \mathbb{N}_{>0}$, let  $U = \big (  U_x  \big )_{x \in \mathcal{V}}$ be a sequence of real-valued functions such that, for any $x\in {\mathcal V}$, $U_x$ has domain $\{x\}$.
We refer to $U$ as \textit{weight function.}
We  introduce the (non-normalised, possibly signed)  measure of the random path model on $\mathcal{W}_{\mathcal{G}}$, which depends on the  parameter $\lambda \in [0, \infty)$ and on the weight function $U$, 
\begin{equation}\label{eq:measure}
\forall w = (m, c,  \gamma ) \in \mathcal{W}_{\mathcal{G}}
\quad \quad 
\mu_{\mathcal{G},  N, \lambda, U}(w) := 
  \prod_{e \in \mathcal{E}} \Big (\frac{ \lambda^{m_e}}{m_e!} \Big ) \prod_{x \in \mathcal{V}} 
  \Big ( 
  U_x(w) 
  \Big )
\end{equation}
Given {an absolutely integrable} function $f : \mathcal{W}_{\mathcal{G}} \rightarrow \mathbb{R}$, we represent its average by 
$
\mu_{ \mathcal{G}, N, \lambda, U} \big ( f \big ) =
\sum\limits_{w \in \mathcal{W}_{\mathcal{G}}} \mu_{\mathcal{G}, N, \lambda, U}(w) f(w).
$
\end{definition}
{We always assume that $\mu_{\mathcal G,N,\lambda,U}$ has finite total variation, namely $\sum_{w\in\mathcal W_{\mathcal G}}|\mu_{\mathcal G,N,\lambda,U}(w)|<\infty$.} The role played by the normalisation factor $\frac{1}{m_e!}$ in (\ref{eq:measure}) will be explained at the beginning of Section \ref{sect:technicallemma}.

\subsection{Reflection positivity and virtual vertices}
\label{sect:RP}
In this section we introduce the extended torus, a graph which is embedded in $\mathbb{R}^{d+1}$ and contains the torus $(\T_L, \E_L)$, which is embedded in  $\mathbb{R}^d$, and the important notion of reflection positivity.
From now on we  consider the random path model on such a graph.

\paragraph{Extended torus, virtual  and original vertices.}
Recall that $(\T_L, \E_L)$ was defined as the graph corresponding to a d-dimensional torus with edges connecting nearest neighbour vertices.
We will now view $(\T_L, \E_L)$ as the sub-graph of a larger graph embedded in $\mathbb{R}^{d+1}$, which will be denoted by 
$(\mathcal{T}_L, \mathcal{E}_L)$ and  will be referred to as \textit{extended torus}. The extended torus is obtained from the $d$-dimensional torus  by duplicating the vertex-set and by adding an edge between every vertex in $\T_L$ and its copy.
More precisely,   we define the vertex set of the extended torus as,
$$
\mathcal{T}_L : =  \big
\{ \,  (x_1, \ldots, x_{d+1} )  \in \mathbb{Z}^{d+1} \, : \,  x_i \in (- \frac{L}{2} , \frac{L}{2}  ] \,  \mbox{ for every } i \in \{1, \ldots, d \}, \,  \mbox{ and }  \,  x_{d+1} \in \{1,2\} \,   \big   \},
$$
where
$\T_L  = \{   (x_1,  \ldots, x_{d+1} )\in \mathcal{T}_L \, \, : \, \, x_{d+1} = 1    \}
 \subset \mathcal{T}_L$,  and $\T_L^{(2)} : = \mathcal{T}_L \setminus \T_L$. 
Recall that  $\E_L$ is defined as the set of edges connecting pairs of nearest neighbour vertices and boundary vertices in $\T_L$ so that  the  $(\T_L, \E_L)$ can be identified with the d-dimensional torus and define the edge-set,
$$
\mathcal{E}_L := \, \, \, \E_L \cup   \big \{ \{x,y\} \, 
\subset \mathbb{Z}^{d+1} : \, \,  x \in \T_L,  \, y = x + (0, \ldots,0, 1) \,        \big \}.
$$
This defines the extended torus $(\mathcal{T}_L, \mathcal{E}_L)$.
We will refer to  the vertices in $\T_L \subset \mathcal{T}_L$ as \textit{original} and to the vertices in  $\T_L^{(2)} \subset \mathcal{T}_L$ as \textit{virtual}.
From now on, we take  $\mathcal{G} = (\mathcal{T}_L, \mathcal{E}_L)$,
for $L \in \mathbb{N}_{>0}$, 
 and we  omit the sub-script $\mathcal{G}$  in all the quantities which were defined above or replace it by $L$ when appropriate.
In this setting we will keep referring to $o$, corresponding to the vertex $({0,\ldots,0,1}) \in  \T_L \subset \mathcal{T}_L \subset \mathbb{Z}^{d+1}$, as the origin.
From now on the current section is an adaptation of \cite{Lees}[Section 3] to the extended torus.

\paragraph{Reflection through edges.}
Recall that the graph $(\mathcal{T}_L, \mathcal{E}_L)$ is embedded in $\mathbb{R}^{d+1}$.
We say that the plane $R$ is  through the  edges of $(\mathcal{T}_L, \mathcal{E}_L)$ if it is orthogonal to one of the cartesian vectors $\boldsymbol{e_i}$ for $i \in \{1, \ldots, d\}$ (and not $i = d+1$)
and it  intersects the midpoint of
$L^{d-1}$ edges of the graph $(\mathcal{T}_L,\Ecal_L)$,  i.e.
$R = \{z \in \mathbb{R}^{d+1} \, \, : \, \, z \cdot \boldsymbol{e_i} = u      \}$, for some $u$ such that  $u - 1/2 \in \mathbb{Z} \cap (-\frac{L}{2},\frac{L}{2}]$
and $i \in \{1, \ldots, d\}$.
See Figure \ref{Fig:exampledoubletorus} for an example. Given such 
a plane $R$, we denote by 
$\Theta : \mathcal{T}_L \rightarrow \mathcal{T}_L$  the reflection operator which reflects the vertices of  $\,\mathcal{T}_L$ with respect to $R$, i.e.
for any $x = (x_1, x_2, \ldots, x_{d+1}) \in \mathcal{T}_L$, 
\begin{equation}
\Theta(x)_k : = 
\begin{cases}
x_k & \mbox{ if } k \neq i, \\ 
{2u} - x_k   \mod  L & \mbox{ if } k = i.
\end{cases}
\end{equation}
Let $\mathcal{T}_L^+, \mathcal{T}_L^- \subset \mathcal{T}_L$ be the corresponding partition of the extended torus into two disjoint halves
such that $\Theta(\mathcal{T}_L^\pm) = \mathcal{T}_L^{\mp}$,
 as in Figure \ref{Fig:exampledoubletorus}. 
Let $\Ecal^{+}_L, \Ecal^{-}_L \subset \Ecal_L$, be the set of edges $\{x,y\}$ with at least one of $x,y$ in $\mathcal{T}_L^+$ respectively $\mathcal{T}_L^-$. Moreover, let 
$\Ecal_L^R  := \Ecal^{+}_L \cap \Ecal^{-}_L$. Note that this set contains $2 L^{d-1}$ edges, half of them intersecting the plane $R$, and all of them  belonging to $\E_L$.
Further, let $\Theta  : \mathcal{W} \rightarrow \mathcal{W}$ denote the reflection operator reflecting the configuration $w=(m, c, \gamma)$ with respect to $R$ (we commit an abuse of notation by using the same letter). More precisely we define $\Theta w=(\Theta m, \Theta c, \Theta \gamma)$ where $(\Theta m)_{\{x,y\}}=m_{\{\Theta x,\Theta y\}}$, 
$(\Theta c)_{\{x,y\}}=c_{\{\Theta x,\Theta y\}}$, 
 $(\Theta \gamma)_x=\gamma_{\Theta x}$.
Given a function $f :\mathcal{W}\to \mathbb{R}$, we also use the letter $\Theta$ to denote the reflection operator $\Theta$ which acts on $f$ as $\Theta f(w) : = f(\Theta w)$.
We denote by $\Acal^\pm$ the set of functions with domain $\mathcal{T}_L^\pm$ and denote by $\mathcal{W}^\pm$ the set of configurations $w\in\mathcal{W}$ that are obtained as a restriction of some $w^\prime \in \mathcal{W}$ to $\mathcal{T}_L^\pm$. 

We remark  that, although the graph $(\mathcal{T}_L, \mathcal{E}_L)$ is embedded in $\mathbb{R}^{d+1}$, we will only consider reflections with respect to reflection planes which are orthogonal to one of the cartesian vectors $\boldsymbol{e}_i$ for $i \in \{1, \ldots, d\}$ (and not $i=d+1$).

\begin{definition}\label{def:invariantunderreflection}
The weight function $U = (U_x)_{x \in \mathcal{T}_L}$,
which was defined in Definition \ref{def:measure},
 is \emph{invariant under reflections} if for any reflection plane $R$ through edges (which is orthogonal to one of the cartesian vectors 
 $\boldsymbol{e}_i$ for $i \in \{1, \ldots d\}$), it holds that,
$$
\forall x \in \mathcal{T}_L \quad \quad \Theta (U_x) = U_{ \Theta(x) },
$$
where   $\Theta$ is the reflection operator associated to the reflection plane $R$.
\end{definition}

The next {theorem} introduces an important tool. The {theorem} states that the random path model with weight function $U$ satisfying the assumptions in Definition \ref{def:measure} and
which is invariant under reflections, as defined in Definition \ref{def:invariantunderreflection}, is reflection positive.
\begin{theorem}[\textbf{Reflection positivity}]\label{theo:RP1}
Consider the torus $(\mathcal{T}_{L},\Ecal_{L})$ for $L\in2\N$. Let $R$ be a reflection plane through edges, which is orthogonal to one of the cartesian vectors $\boldsymbol{e}_i$, $i \in \{1, \ldots, d\}$, let $\Theta$ be the corresponding reflection operator. Consider the random path model with  $N \in \mathbb{N}_{>0}$,  $\lambda \in \mathbb{R}_{>0}$, and weight function $U$  invariant under reflections.
For any pair of {bounded} functions
$f, g \in \mathcal{A}^+$, we have that,
\begin{enumerate}
\item[(1)] $\mu_{L,N, \lambda, U} (f\Theta g)=\mu_{L,N, \lambda, U} (g \Theta f)$,
\item[(2)] $\mu_{L,N, \lambda, U} (f\Theta f)\geq 0$.
\end{enumerate}
From this we obtain that,
\begin{equation}\label{eq:RPCS}
\mu_{L,N, \lambda, U} \big (   f \,  \Theta g  \big )
\leq 
\mu_{L,N, \lambda, U} \big (   f  \,  \Theta f  \big )^{\frac{1}{2}}
\, \, 
\mu_{L,N, \lambda, U} \big (   g \,  \Theta g  \big )^{\frac{1}{2}}.
\end{equation}
\end{theorem}

\begin{proof}[\textbf{Proof of Theorem \ref{theo:RP1}}]
This proof is similar to the proof of Proposition 3.2 in \cite{Lees}, the difference is that here we deal with an extended torus in place of the graph $(\T_L, \E_L)$. 
To begin we introduce the notion of \textit{projection}.
We denote by $\mathcal{W}^R$ the set of  configurations $w = (m, c, \gamma)$ 
such that  $m_e=0$ whenever $e\notin \Ecal_L^R$ and, for all $x\in\mathcal{T}_L$,
$\gamma_x$ leaves all links touching $x$ unpaired.
We also denote by  $P_R:\Wcal \to \Wcal^R$ the projection such that, for any $w=(m,c, \gamma)\in\Wcal$, 
$P_R(w) = (m^R, c^R, \gamma^R)$ is defined as the  configuration such that  $m^R_e = \mathbbm{1}_{\{e \in \mathcal{E}_L^R \}} m_e$ and $c^R_e =   c_e$ if  $e \in \mathcal{E}_L^R$ and $c_e^R = \emptyset $ otherwise, 
and all links are unpaired at every vertex.
The following remark will be useful.
\begin{remark}\label{remark:observationidentif}
Recall the definition of restriction which was provided in Section \ref{sect:definitionsrandompath}. Given a triplet of configurations $w^\prime \in \mathcal{W}^R$, 
 $w_1 \in \mathcal{W}^{+}$,  $w_2 \in \mathcal{W}^{-}$
such that  $P_R(w_1) = P_R(w_2) = w^\prime$,
there exists a unique configuration $w \in \mathcal{W}$ such that 
$$w_{\mathcal{T}_L^+} = w_1, \quad 
w_{\mathcal{T}_L^-} = w_2, \quad 
P_R(w) = w^\prime.$$
This configuration is formed by concatenating $w_1$ and $w_2$
(concatenation includes  the pairing structures of each $w_j$).
\end{remark}
Through the proof we  write $\mu=\mu_{L,N, \lambda, U}$. To begin,  we note that \eqref{eq:RPCS} follows in the standard way as properties \textit{(1)} and \textit{(2)} show that we have a positive semi-definite, symmetric bilinear form.
To prove \textit{(1)}  we note that, by Definition \ref{def:measure} and due to the symmetries of the torus and the fact that $U$ is invariant under reflections, $\mu(w)=\mu(\Theta w)$ for any $w \in \mathcal{W}$. Hence
\begin{equation}
\begin{aligned}
 \mu(f\Theta g)&=\sum_{w\in\mathcal{W}}f(w)\Theta g(w)\mu(w)=\sum_{\Theta w\in\mathcal{W}}f(\Theta w)\Theta g(\Theta w)\mu(w)
 \\
 &=\sum_{\Theta w\in\mathcal{W}} g(w)\Theta f(w)\mu(w)=\sum_{ w \in \mathcal{W}} g(w)\Theta f(w)\mu(w)=\mu(g\Theta f).
 \end{aligned}
 \end{equation}
For \textit{(2)} we condition on the number of links in $w$ crossing the reflection plane and on their colours. 
We write
\begin{equation}\label{eq:decomposition}
\mu(f\Theta f)=\sum_{w \in\Wcal^R}\mu(f\, ; \,w),
\end{equation}
where, for any $w^{\prime}   \in\Wcal^R$,
\begin{equation}
\begin{aligned}
\mu(f\, ; \,w^{\prime}):=&\sum_{\substack{w\in\Wcal \\ P_R(w)=w^{\prime}}}f(w)\Theta f(w) \mu(w)
\\
=& \Big( \prod_{e\in\Ecal_L^R} \, \, \frac{m_e(w^{\prime})!}{\lambda^{m_e(w^{\prime})}} \, \Big )  \, \,  \sum_{\substack{w\in\Wcal \\ P_R(w)=w^{\prime}}}f(w) \, 
\, \, \Big ( \prod_{e\in \Ecal_L^+}   \,\frac{\lambda^{m_e(w)}}{m_e(w)!} \Big )\, \,  \Big (  \prod_{x\in\mathcal{T}_L^+}  \, U_x(w) \Big ) 
\\
&\qquad\qquad\qquad\qquad\qquad\qquad\qquad\Theta f(w) \, \, \Big (   \prod_{e\in\Ecal_L^-}\frac{\lambda^{m_e(w)}}{m_e(w)!} \, \, \Big ) \, \, 
\Big (   \prod_{x\in\mathcal{T}_L^-} U_x(w)\Big ).
\end{aligned}
\end{equation}
Now, any $w\in\Wcal$ such that $P_R(w)=w^{\prime}$ uniquely defines $w_{\mathcal{T}_L^{\pm}}$, the restriction of $w$ to $\mathcal{T}_L^\pm$.
Thus, from Remark \ref{remark:observationidentif} we deduce that we can split the sum over $w\in\Wcal$ with $P_R(w)=w^{\prime}$ as the product of two independent sums and continue:
\begin{equation}
\begin{aligned}
\mu(f\, ; \,w^{\prime})=& \Big ( \prod_{e\in\Ecal_L^R} \, \, \frac{m_e(w^{\prime})!}{\lambda^{m_e(w^{\prime})}}  \Big ) \,  \bigg ( \sum_{\substack{w_1\in\Wcal^+\\ P_R(w_1)=w^{\prime}} }f(w_1) \, 
\, \, \Big ( \prod_{e\in \Ecal_L^+ }   \,\frac{\lambda^{m_e(w_1)}}{m_e(w_1)!} \Big )\, \,  \Big (  \prod_{x\in\mathcal{T}_L^+}  \, U_x(w_1) \Big )  \bigg )
\\
&\qquad\qquad\qquad\qquad\qquad\qquad\qquad
\bigg ( \sum_{\substack{w_2\in\Wcal^-\\ P_R(w_2)=w^{\prime}} } \Theta f(w_2) \, 
\, \, \Big ( \prod_{e\in \Ecal_L^- }   \,\frac{\lambda^{m_e(w_2)}}{m_e(w_2)!} \Big )\, \,  \Big (  \prod_{x\in{\mathcal{T}_L^-}}  \, U_x(w_2) \Big ) \bigg ) \\
=& 
\Big ( \prod_{e\in\Ecal_L^R} \, \, \frac{m_e(w^{\prime})!}{\lambda^{m_e(w^{\prime})}}  \Big ) \,  
\bigg ( 
\sum_{\substack{w_1\in\Wcal^+\\ P_R(w_1)=w^{\prime}} }f(w_1) \, 
\, \, \Big ( \prod_{e\in \Ecal_L^+ }   \,\frac{\lambda^{m_e(w_1)}}{m_e(w_1)!} \Big )\, \,  \Big (  \prod_{x\in\mathcal{T}_L^+}  \, U_x(w_1) \Big ) \bigg )^2 .
\end{aligned}
\end{equation}
The last equality holds true by the  symmetry of the extended torus. Since the last expression is non-negative, from (\ref{eq:decomposition}) we conclude the proof of (2) and, thus, the proof of the {theorem}.
\end{proof}

\subsection{Chessboard estimate}
\label{sect:Chessboardscheme}
We now introduce the notion of support.
Contrary to the notion of domain, 
which was introduced in Section {\ref{sect:definitionsrandompath}},
the notion of support is defined only for subsets of the original torus.
We say that the function
 $f : \mathcal{W} \mapsto \mathbb{R}$ has \textbf{\textit{support}} in 
 $D \subset \T_L$ if it has domain in $D \cup   D^{(2)}$,
 where $D^{(2)}$ is defined as the set of sites which are `on top' of those in $D$, 
$$
  D^{(2)} : = 
  \{ z \in \mathbb{Z}^{d+1} \, \, : z - \boldsymbol{e}_{d+1} \in D    \}.
 $$
  Fix an arbitrary site $t \in \T_L$ and let  $t_0=o$, $t_1$, $\ldots$, $t_k = t$
be a self-avoiding nearest-neighbour path from $o$ to $t$,
and for any $i \in \{1, \ldots, k\}$, let $\Theta_i$ be the reflection with respect to the plane going through the edge $\{  t_{i-1}, t_{i}  \}$.
Let $f$ be a function having support in $\{o\}$ and define
$$
f^{ [t]} : = \Theta_k  \circ \Theta_{k-1} \, \ldots \, \circ \Theta_1 \, ( f ).
$$
Observe that the function $f^{ [t]}$ does not depend on the chosen path (a glance at  
Figure \ref{Fig:chessboard} might be useful).
\begin{figure}
  \centering
    \includegraphics[width=0.31\textwidth]{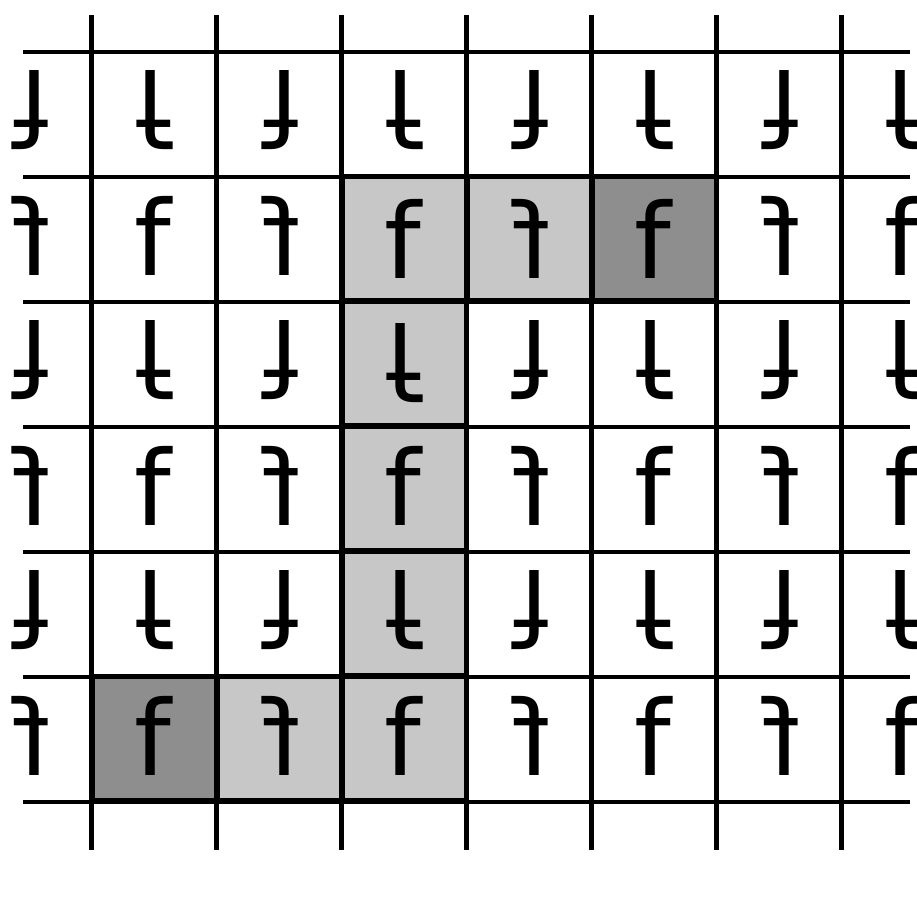}
      \caption{  The function $f^{ [t]  } : =  \Theta_k  \circ \Theta_{k-1} \, \ldots \, \circ \Theta_1 \, ( f )$ does not depend on the chosen path.   }\label{Fig:chessboard}
\end{figure}
\begin{proposition}[Chessboard estimate]\label{prop:chessboardabstract}
Let $f = (f_t)_{t \in \T_L}$ be a sequence of  real-valued functions
with support $\{o\}$ each and which are {all bounded.}
Under the same assumptions as in Theorem \ref{theo:RP1}, 
we have that,
$$
\mu_{L, N, \lambda, U} \Big(   \prod_{t \in \T_L} f^{[t]}_t  \Big   ) \, \leq \, 
\bigg (
\,  \, \prod_{t \in \T_L} \, \, 
   \mu_{L, N, \lambda, U}  \Big (   \prod_{s \in \T_L} f_t^{[s]}  \Big )   \, \, \, \,  \bigg )^{\frac{1}{|\T_L|}}
$$
\end{proposition}
The proof of Proposition \ref{prop:chessboardabstract} for a measure $\mu$ satisfying (\ref{eq:RPCS}) is classical and was first presented in \cite{FrohlichLieb}.
Since we only use reflections with respect to reflection planes which are orthogonal to the cartesian vectors $\boldsymbol{e}_i$, $i \in \{1, \ldots, d\}$ (not $i = d+1$), virtual vertices play no role in the proof and thus the same proof of
\cite{FrohlichLieb} applies to our case directly.
For the proof of Proposition \ref{prop:chessboardabstract} we refer to the original paper \cite{FrohlichLieb} or to the overviews  \cite[Theorem 5.8]{Biskup} or \cite[Theorem 10.11]{FriedliVelenik}.
We now introduce a \textit{central quantity.}
Recall that, for any vertex $x \in \mathcal{T}_L$ and any configuration $w \in \mathcal{W}$,  $u_x(w)$ denotes  the number of links touching $x \in \mathcal{T}_L$ which are unpaired at $x$. 
\begin{definition}[Central quantity]\label{def:Zh}
\textit{For any $L \in \mathbb{N}$, $\lambda \in \mathbb{R}_{\geq 0}$, $N \in \mathbb{N}_{>0}$, 
any $U$ as in Definition \ref{def:measure}, 
any vector of real numbers $\boldsymbol{h} = (h_x)_{x \in \mathcal{T}_L}$, we define}
\begin{equation}\label{eq:Zh}
{\mathcal{Z}_{L,N,\lambda,U}(\boldsymbol h):=
\sum_{w\in\mathcal W}\mu_{L,N,\lambda,U}(w)
\prod_{x\in\mathcal T_L}h_x^{u_x(w)}}
\end{equation}
\end{definition}
{Here $0^0=1$.} The function $h_x^{u_x}$ in Definition \ref{def:Zh} assigns a multiplicative factor $h_x$ to each link touching $x$ which is unpaired at $x$.
We  assume that the weight function $U$ is such that the quantity (\ref{eq:Zh}) is finite for any vector $\boldsymbol{h}$ as in Definition \ref{def:Zh} and for any $L \in 2 \mathbb{N}$.
The next proposition is an immediate consequence of Proposition 
\ref{prop:chessboardabstract}.

\begin{proposition}\label{prop:chessboardscheme}
Fix  arbitrary $L \in 2 \mathbb{N}$, $\lambda \geq 0$, $N \in \mathbb{N}_{>0}$.
Suppose that the weight function $U$ is invariant under reflections.
Let $\boldsymbol{h} = (h_z)_{z \in \mathcal{T}_L}$ be a real-valued vector
such that $|h_z| \leq 1$ for every $z \in \mathcal{T}_L$. For any $x \in \T_L$ define the new real-valued vector
$\boldsymbol{h}^{x} = ( h_z^{x})_{z \in \mathcal{T}_L}$ which is obtained from $\boldsymbol{h}$ by copying the value $h_x$ at each original vertex and the value 
$h_{x+ \boldsymbol{e}_{d+1}}$ at each virtual vertex, namely 
$$
\forall z \in \mathcal{T}_L \quad \quad h_z^{x} :
= 
\begin{cases}
h_x  &\mbox{ if $z \in \T_L$} \\
h_{x + \boldsymbol{e}_{d + 1}}  & \mbox{ if $z \in \T_L^{(2)}$}.
\end{cases}
$$
We have that,
$$
 \mathcal{Z}_{L, N, \lambda, U}( \boldsymbol{h} )
\leq 
\Big ( \, \, \, \prod_{x \in \T_L}   \mathcal{Z}_{L, N, \lambda, U}
\big  ( \boldsymbol{h}^x  \big ) \, \, \, \Big )^{\frac{1}{|\T_L|}}.
$$
\end{proposition}
\begin{proof}
The proof follows from an immediate application of Proposition \ref{prop:chessboardabstract}.
Define,
$$\forall x \in \T_L \quad \quad 
f_{\boldsymbol{h},x}:=  { \, \big  (h_x  \big )\, }^{u_o} \, \,  {\big ( \,  h_{{x+ \boldsymbol{e}_{d+1}}} \big ) \, } ^{u_{o+ \boldsymbol{e}_{d+1}}},$$
note that this function has support $\{o\}$
and which is bounded.
Moreover, note that for any $x \in \T_L$,
\begin{equation}\label{eq:reflectedh}
f^{[x]}_{\boldsymbol{h},x} = { \, \big  (h_x  \big )\, }^{u_x} \, \,  {\big ( \,  h_{{x+ \boldsymbol{e}_{d+1}}} \big ) \, } ^{u_{x+ \boldsymbol{e}_{d+1}}},
\end{equation}
which has support $\{x\}$.
From this we deduce that,
$$
 \mathcal{Z}_{L, N, \lambda, U}( \boldsymbol{h} ) = 
 \mu_{L, N, \lambda, U}  \Big( \prod_{x \in {\T_L}} 
f^{[x]}_{\boldsymbol{h},x}
\Big  )
$$
and that, for any $x \in \T_L$,
$$
 \mathcal{Z}_{L, N, \lambda, U}( \boldsymbol{h}^x ) = 
 \mu_{L, N, \lambda, U}  \Big( \prod_{z \in {\T_L}} 
f^{[z]}_{\boldsymbol{h},x}
\Big  )
$$
The claim now follows from a direct application of Proposition \ref{prop:chessboardabstract}.
\end{proof}

\subsection{Polynomial expansion}
\label{sect:relation}
This sub-section presents an important step of the proof of the Key Inequality, namely 
  Proposition \ref{prop:moments} below, which
 states a 
relation between the values of any vector $\boldsymbol{h}$, 
the partition function  $\mathcal{Z}( \varphi \boldsymbol{h})$
in the limit $\varphi \rightarrow 0$, where $\varphi \in \mathbb{R}$,
and the partition functions  which were defined in (\ref{eq:partitionlambda}).
To make this connection we  choose an appropriate weight function, which is  denoted by $H$ and is introduced in the next definition,  and  expand $\mathcal{Z}_{L,N, \lambda, H}(\varphi \boldsymbol{h})$ as a polynomial
in $\varphi$.
Recall that $n_x$ denotes the number of pairings at $x$ (i.e. one half the number of links touching $x$ which are paired at $x$ to another link touching $x$).
\begin{definition}\label{def:definitionH}
We define the  weight functions, $H = (H_x)_{x \in \mathcal{T}_L}$, as follows:
{\small\begin{align}\label{eq:weightoriginal2}
 \forall x \in \T_L \quad \quad H_x  &  := \begin{cases}
1 & \mbox{ if $n_x \leq 1$, $u_x \leq 2$, and no link on $\{x, x + \boldsymbol{e}_{d+1}\}$ is unpaired at $x$,} \\
\frac{1}{2} & \mbox{ if $n_x \leq 1$, $u_x \leq 2$, and precisely one link on $\{x, x + \boldsymbol{e}_{d+1}\}$ is unpaired at $x$,} \\
0 & \mbox{ otherwise. }
\end{cases} \\
 \forall x \in \T_L^{(2)} \quad \quad H_x & : = \mathbbm{1}_{ \{n_x = 0 \} }
\end{align}}
Moreover, we define  $\mathcal{W}^1$, the set of configurations $w \in \mathcal{W}$ such that 
$\prod_{x \in \mathcal{T}_L} H_x(w)  >0$.
\end{definition}
Each configuration $w \not\in \mathcal{W}^1$ has weight zero under $\mu_{ L, N, \lambda, H }$ and thus ignoring it  costs  nothing.
See Figure \ref{Fig:Pairingsexample0} for an example of two realisations $w$ which are not in $\mathcal{W}^1$.
The upper bound  $u_x \leq 2$ in Definition \ref{def:definitionH} is only necessary to guarantee that $|\mathcal{W}^1| < \infty$ and $2$ might be replaced by any other integer greater than two with no effect on the next pages. From the boundedness of $|\mathcal{W}^1|$ we deduce that $\mathcal{Z}_{L, N, \lambda, H}(\boldsymbol{h}) < \infty$ for any $L \in \mathbb{N}$, {$N\in\mathbb N_{>0}$ and $\lambda\in[0,\infty)$,}
and 
$\boldsymbol{h} \in \mathbb{R}^{\mathcal{T}_L   }$.
Note also that $H_x$ has domain $\{x\}$ and that 
$H = (H_x)_{x \in \mathcal{T}_L}$  is invariant under reflections, thus all the results stated in Sections \ref{sect:RP} and \ref{sect:Chessboardscheme} apply to $\mu_{ L, N, \lambda, U }$ under the choice of $U= H$.
As we will explain in Section \ref{sect:technicallemma}, the choice of $H$ is such that any closed path in $w$ lies entirely in the original torus and is vertex-self-avoiding, moreover closed paths are mutually vertex-disjoint
(paths will be defined later, but the reader might already try to have an intuition of what they are).
Contrary to closed paths, open paths are not entirely vertex-self-avoiding, since they are allowed to touch themselves or other paths  at their end-points.
The open paths might start (or end) at virtual vertices or at original vertices and they are allowed to touch the virtual vertices only at their end-points.
These details and further technical aspects are fundamental for the validity of the next proposition and will be discussed 
in  Section \ref{sect:technicallemma}.
 \begin{figure}
  \centering
    \includegraphics[width=\textwidth]{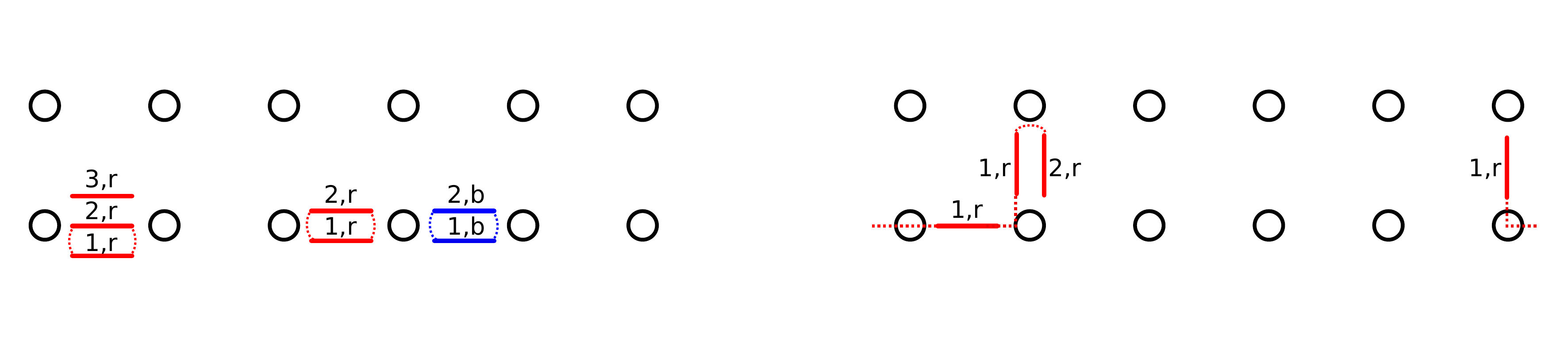}
      \caption{Two examples of realisations $w \in \mathcal{W} \setminus \mathcal{W}^1$ on the extended torus $(\mathcal{T}_L, \mathcal{E}_L)$ in dimension $d = 1$, where $L = 6$, with the upper row representing the virtual vertices. The realisation  on the left is not in $\mathcal{W}^1$ since there exists a vertex $x$ with $n_x = 2$, the realisation  on the right is not in $\mathcal{W}^1$ since there exists a virtual vertex $y$ with $n_y = 1$.
    }\label{Fig:Pairingsexample0}
\end{figure}
For the  statement  of the next proposition recall the definition of the partition functions (\ref{eq:partitionlambda}).
\begin{proposition}[\textbf{Polynomial expansion}]\label{prop:moments}
For any fixed $L \in 2 \mathbb{N}$,  $N \in \mathbb{N}_{ >0 }$, $\lambda \in \mathbb{R}_{>0}$,   any vector of real numbers $\boldsymbol{h} = (h_x)_{x \in \mathcal{T}_L}$,
and $\varphi \in \mathbb{R}$,
we have that, 
\begin{equation}\label{eq:expansion}
\mathcal{Z}_{L, N, \lambda, H} ( \varphi \, \boldsymbol{h} ) = \mathbb{Y}^\ell_{L, N, \lambda} \,  + \, \varphi^2 \mathcal{Z}^{(2)}_{L, N, \lambda, H} (\boldsymbol{h})  + o(\varphi^2),
\end{equation}
in the limit as $\varphi \rightarrow 0$,  where 
\begin{multline*}
\mathcal{Z}^{(2)}_{L, N, \lambda, H} (\boldsymbol{h})  :=  
\, N \, \lambda \, 
\mathbb{Y}^{\ell}_{L, N, \lambda} \,
\Big ( \big (\sum\limits_{\{x,y\} \in \E_L}   h_x h_y  \big ) \, + \,
\frac{1}{2} \sum\limits_{x \in \T_L}  h_x  {h_{x+\boldsymbol e_{d+1}}} 
\Big ) 
\\ + \, \, \,     N \, \frac{\lambda^2}{2} \,\sum\limits_{ \substack{x,y  \in \T_L }} 
\mathbb{Y}_{L, N, \lambda}(x,y)  \, \,
\big( \sum\limits_{  \substack{ q \in \mathcal{T}_L: \\ \{x,q\} \in \mathcal{E}_L  } } {h_q} \big ) \, \,
\big( \sum\limits_{  \substack{ r \in \mathcal{T}_L: \\ \{y,r\} \in \mathcal{E}_L  } } h_r  \big )
\end{multline*}
\end{proposition}
{The Key Inequality will follow from the non-positivity of the coefficient of $\varphi^2$ in $\mathcal Z_{L,N,\lambda,H}(\varphi\boldsymbol h^{\boldsymbol v})$, for every $\boldsymbol v\in\mathbb R^{\mathbb T_L}$, with $\boldsymbol h^{\boldsymbol v}$ given by (\ref{eq:hasfunctionofv}).}
Note that the terms in the expansion are slightly different than in (\ref{eq:expansion1}), since here we use the partition functions parametrised by $\lambda$,  which were defined in (\ref{eq:partitionlambda}), and the entries of the vector $\boldsymbol{h}$  are associated to the vertices of the extended torus (later we will relate the vector $\boldsymbol{h}$ to a vector $\boldsymbol{v}$, whose entries are associated to the vertices of the original torus, obtaining an expression  which is similar to (\ref{eq:expansion1})).
The remainder of the current subsection is devoted to the proof of Proposition \ref{prop:moments}.
Before presenting  the proof, we will provide some definitions and state a preparatory lemma.
All the definitions below are functional to the proof of Proposition \ref{prop:moments}. Section \ref{sect:proofinequailty}, which contains the proof of Theorem \ref{theo:keyinequality}, 
can be read independently from  what follows below in the current subsection.

\textbf{Paths.}
Given $w \in \mathcal{W}$, we use $(\{x, y\}, p)$ to denote the p-th link of $w$ which is on the edge $\{x,y\}$, with 
 $p \in \{1, \ldots,  m_{\{x,y\}}(w) \}$.
We say that a set of links \textit{$S$} in $w$,
$$
S = \big \{
(\{x_1, y_1\}, p_1),
(\{x_2, y_2\}, p_2),
\ldots 
{(\{x_\ell,y_\ell\},p_\ell)}
\big \},$$
is \textit{pairing-connected in  $w$} if, for any 
pair of links, $(\{x, y\}, p)$,
$(\{x^\prime,y^\prime\}, p^\prime) \in S$,
there exists an ordered sequence of links in  $S$,
 $\big (
(\{x^\prime_1, y^\prime_1\}, p^\prime_1),$
$(\{x^\prime_2, y^\prime_2\}, p^\prime_2),
\ldots
(\{x^\prime_k, y^\prime_k\}, p^\prime_k)
\big ) \subset S$
such that the following two conditions hold at the same time:
\begin{enumerate}[(i)]
\item $(\{x, y\}, p) = (\{x^\prime_1, y^\prime_1\}, p^\prime_1)$, and
$(\{x^\prime, y^\prime\}, p^\prime) = (\{x^\prime_k, y^\prime_k\}, p^\prime_k)$,
\item for any $i \in \{1, \ldots, k-1\}$,
$y^\prime_i = x^\prime_{i+1}$
and $(\{x^\prime_i, y^\prime_i\}, p^\prime_i )$
is paired to  $(\{x^\prime_{i+1}, y^\prime_{i+1}\}, p^\prime_{i+1} )$ at 
$y^\prime_i = x^\prime_{i+1}$.
\end{enumerate}
Paths are maximal pairing-connected sets. 
More precisely, a set of links $S$ of $w$ is a \textit{path} in $w$ if it is pairing-connected and there exists no pairing-connected set of links in  $w$, $S^\prime$, which is such that  $S^\prime  \supset S$  and $S^\prime \neq S$.
It is necessarily the case that all links belonging to the same path have the same colour.
For example, the configuration represented in Figure \ref{Fig:pairingexample} contains seven paths, 
 two of them are coloured by blue and five  by red. 

\textbf{$\ell$-loops, double links, $\ell$-walks, segments, extremal links.}
We will now distinguish between different types of paths.
A path $S$ of $w$ is called \textit{loop of links}, or just \textit{$\ell$-loop}, if it is such that any link $(\{x,y\}, p) \in S$  is paired to another link at both its end-points and $|S| > 2$.
A path $S$ of $w$ is called \textit{double link},  if it is such that any link $(\{x,y\}, p) \in S$  is paired at both its end-points and $|S| =  2$.
It is necessarily the case that both links belonging to the double link are on the same  edge.
A path $S$ of $w$ is called \textit{walk of links}, or just \textit{$\ell$-walk}, if  $|S| >  1$ and there exist precisely two distinct links in $S$ such that each of them is unpaired at  one end-point and paired at the other end-point.
Such two links will be called \textit{extremal links} for the $\ell$-walk
or extremal links for $w$.
A path $S$  of $w$ is called \textit{segment} if 
$|S| = 1$.
If $S$ is a segment, then  the unique link which belongs to $S$ is unpaired at both its end points.
From these definitions it follows that any path is either {an $\ell$-loop},  a double link,  {an $\ell$-walk}, or a segment. There are no other possibilities.
For example, the configuration $w$  in Figure \ref{Fig:pairingexample} is composed of one $\ell$-loop, two double links,
three segments, and one $\ell$-walk which is composed of two links.
The two links belonging to such $\ell$-walk are the only two extremal links 
of the configuration in Figure \ref{Fig:pairingexample}.

\textbf{Subsets of $\boldsymbol{\mathcal{W}}^1$.}
We now define several subsets of $\mathcal{W}^1 \subset \mathcal{W}$,
where the set $\mathcal{W}^1$ was defined in Definition \ref{def:definitionH}.
\begin{itemize}
\item Let $\mathcal{A}^\ell$ be the set of realisations 
$w \in \mathcal{W}^1$ such  that no path of $w$ is {an $\ell$-walk} or a segment. In other words,  each link of $w$ is paired at both its end-points. This also means that each path of $w \in \mathcal{A}^\ell$ is either {an $\ell$-loop} or a double link and, by definition of $H$, that  no link of $w \in \mathcal{A}^\ell$ is allowed to touch a virtual vertex.

\item 
For any $\{x,y\} \in \mathcal{E}_L$, let  $\mathcal{A}^s( \{x,y\}  )$ be the set of realisations $w \in \mathcal{W}^1$ such that  one (and not more than one) path of $w$ is a segment, such a segment  is composed of a link which is on the edge  $\{x,y\}$, and {no path of $w$ is an $\ell$-walk.}
In other words, each link of $w$ except for the one which belongs to the segment is paired at both its end-points. 
A realisation $w \in \mathcal{A}^s( \{x,y\}  )$ is represented in
 Figure \ref{Fig:Pairingsexample1}-left.
\item 
For any pair of (directed, not necessarily distinct) edges $(x,q)$, $(y,r) \in \mathcal{E}_L$, let  
$\mathcal{A}^w( (x,q), (y,r)   )$  be the set of realisations $w \in \mathcal{W}^1$ such that the following three conditions hold true at the same time:  \textbf{(1)} there exists a unique $\ell$-walk in $w$
\textbf{(2)}  the two extremal links of such a walk are on the edges 
$\{x,q\}$, $\{y,r\}$ respectively, one of them is unpaired at $q$ and the other one is unpaired at $r$,
 \textbf{(3)} no path of $w$ is a segment.
These three conditions and the definition of $H$ imply that the following properties hold for any $w \in \mathcal{A}^w( (x,q), (y,r))$:
\begin{enumerate}[(i)]
\item The unique $\ell$-walk in $w$ has end-points $q$ and $r$, where $q$ and $r$ might coincide (see some examples in Figure \ref{Fig:Pairingsexample1}-right,   Figure \ref{Fig:Pairingsexample2},  Figure \ref{Fig:Pairingsexample3},
where $x$ is taken to be the origin and $d = 1$), 
\item There are precisely two extremal links, which are on the edges $\{x,q\}$ and $\{y,r\}$ respectively (it is possible that $\{x,q\} = \{y,r\}$), and all the remaining links are paired at both their end-points.
\item Any link of $w$ which is not extremal is on edges in $\E_L$.
\item Both $x$ and $y$  belong to the original torus, $q$ and $r$ might be original or virtual.
\end{enumerate}
\end{itemize}
In the statement of the next lemma, recall that $(x,y)$ represents an edge directed from $x$ to $y$, while $\{x,y\}$ represents {an undirected edge}.
\begin{lemma}\label{lemma:componentsidentification}
Under the same assumptions as in Proposition \ref{prop:moments},
for any $(x, q)$, $(y,r)$, $\{u,b\} \in \mathcal{E}_L$, we have that,
\begin{align}
\label{eq:formula1}
\mu_{L, N, \lambda, H} \, \Big ( \,  \mathcal{A}^\ell  \,   \Big ) \, &  = \, \mathbb{Y}^{\ell}_{L, N,  \lambda},  \\
\label{eq:formula2}
\mu_{L, N, \lambda, H} 
 \big ( \, \mathcal{A}^s(\{u, b\}) \,   \big ) \, &  =
\begin{cases}
\lambda \, N \, \mathbb{Y}^{\ell}_{L, N,  \lambda}  &   \mbox{ if $\{u,b\} \in\E_L$, }   \\
  \frac{\lambda}{2} \, N  \, \mathbb{Y}^{\ell}_{L, N,  \lambda}   &     \mbox{ if $\{u,b\} \in   \mathcal{E}_L \setminus   \E_L$,} 
\end{cases}  \\
\label{eq:formula3}
\mu_{L, N, \lambda, H} \, \Big ( \, 
\mathcal{A}^w \big ( (x,q), (y,r)   \big ) \,   \Big ) \, &  = \,  
\begin{cases}
 \lambda^2 \, N \, \mathbb{Y}_{L, N,  \lambda}(x,y) \,  & \mbox{ if $x, y \in \T_L$ and $(x,q) \neq (y,r)$,} \\
\frac{\lambda^2}{2} \, N  \, \mathbb{Y}_{L, N,  \lambda}(x,x) \,   & \mbox{ if $x, y \in \T_L$ and $(x,q)= (y,r)$,} \\
0  & \mbox{ if $\{x,y\}  \cap \T_L^{(2)} \neq \emptyset$.} \\
\end{cases}
\end{align}
\end{lemma}
The proof of the lemma is postponed to Section \ref{sect:technicallemma} and is crucial.
We will now present the proof of Proposition \ref{prop:moments} given Lemma \ref{lemma:componentsidentification} .
\begin{proof}[\textbf{Proof of Proposition \ref{prop:moments} given Lemma \ref{lemma:componentsidentification}}]
Fix $L \in 2 \mathbb{N}$, $N \in \mathbb{N}_{>0}$, $\lambda > 0$ and a vector of real numbers
$\boldsymbol{h} = (h_x)_{x \in \mathcal{T}_L}$. We have that,
\begin{equation}\label{eq:consideration0}
\mathcal{Z}_{L, N, \lambda, H}( \varphi \boldsymbol{h} ) = \sum\limits_{ i= 0}^{\infty} \varphi^i \,  \mathcal{C}^{(i)}_{L, N, \lambda, H}(  \boldsymbol{h} ),
\end{equation}
where 
$$
\mathcal{C}^{(i)}_{L, N, \lambda, H}(  \boldsymbol{h} ) : = 
\mu_{L, N, \lambda, H } \Big ( \mathbbm{1}_{ \{  M = i   \} }  \big  ( \prod_{ z \in \mathcal{T}_L} h_z^{u_z}   \big ) \, \, \Big )
$$
and 
$
M : = \sum_{z \in \mathcal{T}_L} u_z
$
is the number of end-points of links which are unpaired in the whole graph.
First of all, note that 
\begin{equation}\label{eq:consideration1}
\forall i \in 2 \mathbb{N} + 1 \quad  \quad \mathcal{C}^{(i)}_{L, N, \lambda, H}(  \boldsymbol{h} ) = 0.
\end{equation}
since any path has either no link with unpaired end-points, or two links with precisely one unpaired end-point each, or one link with two unpaired end-points.
Thus, $M(w)$ is even for any $w \in \mathcal{W}^1$.
Moreover, note that,
\begin{equation}\label{eq:consideration2}
\mathcal{C}^{(0)}_{L, N, \lambda, H}(  \boldsymbol{h} ) = \mu_{L, N, \lambda, H } \Big (\, \mathcal{A}^\ell \,  \Big ) =\mathbb{Y}^\ell _{L, N, \lambda},
\end{equation}
where the first identity holds true since $w \in \{M = 0 \}$ if and only if each path of $w$ is a  $\ell$-loop or a double link and the second identity follows from  Lemma \ref{lemma:componentsidentification}.
Furthermore, note that $w \in \{M = 2 \}   \cap \mathcal{W}^1$ if and only if 
precisely one path of $w$ is {a segment} or {an $\ell$-walk} and all the remaining paths of $w$ are $\ell$-loops or double links. In the next expression, the first  term in the right-hand side corresponds to a sum over all possible edges on which the segment might be located, the second term in the  right-hand side corresponds to a sum over all (directed) edges on which the extremal links might be located (recall the definitions provided before the statement of Lemma \ref{lemma:componentsidentification}),
\begin{multline}\label{eq:consideration3}
\mathcal{C}^{(2)}_{L, N, \lambda, H}(  \boldsymbol{h} )  = \\ 
\sum\limits_{  \{x,y\}  \in \mathcal{E}_L  } \, \mu_{L, N, \lambda, H } \Big (  \mathcal{A}^s( \{x,y\}  )      \, \Big )  \, h_x \, h_y \,  
+  \sum\limits_{ \substack{ \{ (x,q), (y,r) \}  \subset \mathcal{E}_L}}
\mu_{L, N, \lambda, H } 
\Big (  \mathcal{A}^w \big ( (x,q), (y,r ) \big  )   \, \, \Big ) \, h_q \, h_r.
\end{multline}
Note that the second sum in the right-hand side is over all \textit{unordered} pairs of (not necessarily distinct) \textit{directed} edges.
Now we apply Lemma \ref{lemma:componentsidentification} and we re-write the second term in the right-hand side of the previous expression as follows,
\begin{multline}\label{eq:consideration4}
\frac{1}{2}  \sum\limits_{ \substack{x,y  \in \T_L  : \\ x \neq y}} \, 
\sum\limits_{ \substack{ q, r \in {\mathcal T_L}  : \\ \{x,q\}, \{y,r\} \in \mathcal{E}_L }} \, 
h_q \, h_r \, 
\mu_{L, N, \lambda, H } 
\Big (  \mathcal{A}^w \big ( (x,q), (y,r) \big  )   \Big ) \, \, + \, \, \\
\frac{1}{2} \,
\sum\limits_{x \in \T_L} \, 
 \sum\limits_{ \substack{ q, r \in {\mathcal T_L}  :  \\ \{ x,q\}, \{x,r\} \in \mathcal{E}_L, q\neq r } }
 h_q \, h_r \, 
\mu_{L, N, \lambda, H } 
\Big (  \mathcal{A}^w \big ( (x,q), (x,r) \big  )   \, \, \Big ) \\
 \, \, + \, \, 
\sum\limits_{x \in \T_L}
\sum\limits_{ \substack{ q \in {\mathcal T_L} : \\ \{x, q\} \in \mathcal{E}_L} }
 h^2_q \, 
\mu_{L, N, \lambda, H } 
\Big (  \mathcal{A}^w \big ( (x,q), (x,q) \big  )   \, \, \Big )  \\
 = \frac{1}{2} \, N \, \lambda^2 \,  \sum\limits_{x,y  \in \T_L}    \mathbb{Y}_{L, N, \lambda}(x,y) \, \big (   \sum\limits_{\substack{ q \in \mathcal{T}_L : \\ \{ x,q\} \in \mathcal{E}_L}  }  h_q  \big )   
 \big (   \sum\limits_{\substack{r \in \mathcal{T}_L : \\ \{ y,r\} \in \mathcal{E}_L}  }  h_r \big ) 
\end{multline}
where the factor one-half in the first two terms is a multiplicity factor due to the fact that we sum over  \textit{ordered} pairs of sites. 
{By replacing the second term in the right-hand side of (\ref{eq:consideration3}) by (\ref{eq:consideration4}),} applying Lemma \ref{lemma:componentsidentification} for the first  term in the right-hand side of (\ref{eq:consideration3}), using (\ref{eq:consideration1}) and 
 (\ref{eq:consideration2}), we conclude the proof of the proposition. 
\end{proof}

\subsubsection{Proof of Lemma \ref{lemma:componentsidentification}}
\label{sect:technicallemma}
In this section we prove Lemma \ref{lemma:componentsidentification},
which is a fundamental step in the proof of the Polynomial expansion.
The proof of (\ref{eq:formula1}) is the easiest. Indeed,  our choice of the weight function $H$ imposes that any configuration in the set $\mathcal{A}^\ell$ consists of mutually-vertex-disjoint $\ell$-loops and double links which lie entirely in the original torus and these can be identified with  loops and  double edges of the configurations in $\Omega^\ell$ taking the same positions.
The proofs of (\ref{eq:formula2}) and  (\ref{eq:formula3}) are  more elaborate.  The proof requires defining a map which maps sets of configurations in $\mathcal{A}^s(\{x,y\})$ to sets of configurations in $\Omega^\ell$ and sets of configurations in ${\mathcal A^w((x,q),(y,r))}$ to sets of configurations in $\Omega_{x,y}$ and consists of a comparison of the weights taken by such sets.
Informally the map works as follows: 
For the proof of (\ref{eq:formula2}), we take any configuration in $\mathcal{A}^s(\{x,y\})$ and `remove' the link  which is unpaired at both its end-points.
Such a removal has a cost $\lambda$ (whose corresponding factor appears in the right-hand side of (\ref{eq:formula2})) and leads to a configuration in $\mathcal{A}^\ell$. After that, we compare the sets of configurations $\mathcal{A}^\ell$ obtained after such a  `removal' 
with sets of configurations in $\Omega^\ell$ similarly to the previous case.
For the proof of (\ref{eq:formula3}) we `remove' from any configuration in ${\mathcal A^w((x,q),(y,r))}$ the two extremal links (which, by definition, are on $\{x,q\}$ and  on $\{y,r\}$ respectively and are unpaired at $q$ and $r$ respectively) paying a cost $\lambda^2$
(which appears in the right-hand side of (\ref{eq:formula3}))
 and obtain a configuration with {a self-avoiding path from $x$ to $y$, allowing a one-link path and the marked degenerate path when $x=y$,} and possibly double links and $\ell$-loops, with all such objects being vertex-self-avoiding, mutually-vertex-self-avoiding and lying entirely in the original torus by our choice $U = H$. Such objects ({the surviving path}, double links and $\ell$-loops) can be identified with corresponding objects of the configurations in  $\Omega_{x,y}$  (walk, double edges and loops respectively) taking the same positions. Such an identification allows the comparison of the weights of the 
 set ${\mathcal A^w((x,q),(y,r))}$  under $\mu$ and the weights taken by the configurations in $\Omega_{x,y}$ in partition function $\mathbb{Y}_{L, N, \lambda}(x,y)$.
 It is important for such a comparison to ensure that the `removal' of the links does not leave a `hole':  For this reason the definition of the weight function $H$ which we provided implies that 
the $\ell$-walk is not entirely vertex-self-avoiding, namely at the vertices where its two extremal links are unpaired, $q$ and $r$,
 it might `touch' itself or other paths. Here by `no hole' we mean that, when the two extremal links are `removed', one obtains configurations whose paths  are `free' to use the vertices  which are touched by the  links which get removed.
A further technical aspect in the proofs of (\ref{eq:formula2}) and {(\ref{eq:formula3})} is that such a `removal'  is a many-to-one map, since the {links which get} `removed' might occupy different positions on the same edge and the `removal'  maps several input configurations with different positions of such links to the same output.
For this reason we  need to compute the factor  corresponding to the number of such possible positions, which also depends on the pairing of the other links on that edge.  Fortunately for us, the factor $\frac{1}{m_e!}$ 
in the definition of the measure $\mu_{L, N, \lambda, U}$ assigns a higher weight to the configuration obtained after the `removal'  and such {an energy gain} matches the corresponding entropy loss \textit{perfectly}, giving a total factor  which equals precisely  one.

\begin{proof}[\textbf{Proof of Lemma \ref{lemma:componentsidentification}}]
For the formal proof it will be  convenient dealing with undirected  sub-graphs of the torus and for this reason we introduce the set $\Sigma$,
which  can be viewed as  an `intermediate object' between the sets  $\mathcal{W}^1$ and $\Omega \cup \Omega^\ell$, whose respective subsets must be compared.

\textbf{\textit{Definition of the set  $\Sigma$.}}
Let $\Sigma$ be the set of spanning sub-graphs of $(\T_L, \E_L)$ such that 
every vertex has degree zero, one or two. Any connected component of  $\sigma \in \Sigma$ is called \textit{monomer} if it consists of a single vertex, \textit{isolated edge} if it consists of two vertices connected by one edge, \textit{loop} if the set of its edges is isomorphic to a simple closed curve in $\mathbb{R}^d$, \textit{walk} if the set of its edges is isomorphic to an open simple curve in $\mathbb{R}^d$.
Thus, an isolated edge is also a walk.
For $x \neq y$, let $\Sigma_{x,y}$ be defined as the  set of graphs $\sigma \in \Sigma$ such that there exists a  walk with end-points   $x$ and $y$ and any other connected component is a monomer, {an isolated edge} or a loop,
let $\Sigma^\ell$ be defined as  the  set of graphs $\sigma \in \Sigma$ such that any connected component is  a monomer, {an isolated edge} or a loop,
let $\Sigma_{x,x}$ be the set of graphs $\sigma \in \Sigma^\ell$ such that $x$ is monomer.
For any $\sigma \in \Sigma$, let $\mathcal{L}(\sigma)$ be the  number of connected components in $\sigma$ which are not monomers (by a slight abuse of notation, since we already defined the related quantity $\mathcal{L}(\pi)$ in the introduction) let $\mathcal{D}(\sigma)$ be the  number of isolated edges in $\sigma$ 
 let 
$\mathcal{D}^\prime(\sigma)$ 
  be the  number of isolated edges in $\sigma$ which do not contain the origin,
let $| \sigma |$ be the  number of edges in $\sigma$.
Recall the definitions of the 
partition functions (\ref{eq:partitionlambda}) parametrised by $\lambda$.
We have that, for any $y \in \T_L \setminus \{o\}$,
\begin{align}
\label{eq:relationclosedpartition}
\mathbb{Y}_{L, N, \lambda}^{\ell} &  =  \sum\limits_{ \sigma \in \Sigma^\ell}  \, \,  \lambda^{ |\sigma|}  \, \, N^{ \mathcal{L}(\sigma)} \,  \,  (\frac{\lambda}{2})^{ \mathcal{D}(\sigma)}, \\
\label{eq:relationdirectedpartition}
 N \, \mathbb{Y}_{L, N, \lambda}(o,y) &  =    \sum\limits_{ \sigma \in \Sigma_{o,y}}  
 \lambda^{ |\sigma|}  \, \, N^{ \mathcal{L}(\sigma)} \,  (\frac{\lambda}{2})^{ \mathcal{D}^\prime(\sigma)} \\
 \label{eq:relationdirectedpartition2}
\mathbb{Y}_{L, N, \lambda}(o,o) &  =    \sum\limits_{ \sigma \in {\Sigma_{o,o}}}  
 \lambda^{ |\sigma|}  \, \, N^{ \mathcal{L}(\sigma)} \,  (\frac{\lambda}{2})^{ \mathcal{D}^\prime(\sigma)}.
\end{align}
To see why the previous relations hold true, note that there is an obvious correspondence between the elements $\pi \in \Omega^{\ell}$ and the elements $\sigma \in \Sigma^{\ell}$ and between the elements $\pi \in \Omega_{o,x}$ and the elements $\sigma \in \Sigma_{o,x}$.
Indeed, for each $\pi$, we obtain a unique element $\sigma$ which is associated to $\pi$ by replacing any double edge, directed loop or directed walk by {an isolated edge},  undirected loop or undirected walk respectively which is composed of the same edges and sites. 
We deduce  (\ref{eq:relationclosedpartition}) and  
 (\ref{eq:relationdirectedpartition})  from the definitions
 (\ref{eq:partitionlambda})
 considering that directed loops have two possible orientations and that double edges in $\pi$ consist of two (directed) edges while the isolated edges in $\sigma$  just of one edge.
Note  that the factor $N$ in the left-hand side of {(\ref{eq:relationdirectedpartition})} 
is due to the fact that $\mathcal{L}(\pi)$, which was defined in Section \ref{sect:rigorousresults}, does not count the walk, while $\mathcal{L}(\sigma)$ 
counts the number of connected components which are not monomers and thus also the walk.
Finally, note that in {(\ref{eq:relationdirectedpartition})}  and (\ref{eq:relationdirectedpartition2}) we have $\mathcal{D}^\prime$ in place of 
$\mathcal{D}$ since, if the walk consists of just one edge, {we do not want to assign} to it a factor $\frac{\lambda}{2}$.
 Now that the partition functions have been defined in terms of sums over elements of $\Sigma$, we can proceed with the comparison between the elements of $\mathcal{W}^1$ and the elements 
of $\Sigma$. This comparison will require introducing a map between such sets and studying its multiplicity properties. 

Below we will keep adopting the following terminology: \textit{double links}, \textit{$\ell$-loops}, \textit{$\ell$-walks},  and \textit{segments} for the paths of the realisations $w \in \mathcal{W}^1$; \textit{isolated edges}, \textit{loops}, \textit{walks} and \textit{monomers} for the connected components of the realisation $\sigma \in \Sigma$.
Moreover, we write that $\{x,y\} \in \sigma$
if $\{x,y\}$ belongs to the edge set of $\sigma \in \Sigma$.

\textbf{\textit{Definition and properties of the map $Q : \mathcal{W}^1 \mapsto \Sigma$.}}
For any $w \in \mathcal{W}^1$,  let $\mathcal{Q}(w)$ be the set of edges $\{x,y\} \in \E_L$ such that there exists  a link on $\{x,y\}$ in $w$ which is paired both at $x$ and $y$. We define a  map $Q$ which associates to each realisation $w \in \mathcal{W}^1$ the realisation $Q(w)  := (\T_L, \mathcal{Q}(w))$.
To begin note that, 
\begin{equation}\label{eq:containedin}
\forall w \in \mathcal{W}^1, \quad  Q(w) \in \Sigma.
\end{equation}
This holds true since, by definition of  $\mathcal{W}^1$, for each realisation $w \in \mathcal{W}^1$,  
each vertex of $Q(w)$ has degree zero, one or two.
For  any $\sigma \in \Sigma$, define  the set $Q^{-1}(\sigma) : = \{ w \in \mathcal{W}^1 \, : \,   Q(w) = \sigma  \}$. From the definition of the map $Q$ we deduce that,
for any pair of graphs $\sigma_1, \sigma_2 \in \Sigma$, 
\begin{equation}\label{eq:disjointness}
\sigma_1 \neq \sigma_2 \implies 
Q^{-1}(\sigma_1) \cap  Q^{-1}(\sigma_2) = \emptyset.
\end{equation}
Note that  for any $w \in \mathcal{W}^1$, 
a loop is present in  $Q(w)$ if and only if {an $\ell$-loop}
with precisely one link located on each edge of the loop 
is present in $w$. {An $\ell$-walk with extremal links $(x,q),(y,r)$, $x\ne y$, gives a walk in $Q(w)$ with one edge per non-extremal link. In particular, an isolated edge in $Q(w)$ may arise either from a double link or from an $\ell$-walk with one non-extremal link.}
Finally, suppose that  $x = y$. Note that, by definition of $H$, for any $w \in \mathcal{W}^1$,  {an $\ell$-walk} with 
extremal links $(x,q)$, $(x,r)$ can only consist 
 of two links which are paired to each other at $x$ and which are both extremal in $w$.
Thus, $Q(w)$ has a monomer at $x=y$ if and only if either {an $\ell$-walk} 
composed of just two links paired at $x$ and on the edges $\{x,q\}$, $\{x,r\}$ 
 for some $q,r \in \mathcal{T}_L$ (with possibly $q=r$) is
present in $w$ or if no link of $w$ is paired at $x = y$.
See also Figures \ref{Fig:Pairingsexample1}, \ref{Fig:Pairingsexample2}, \ref{Fig:Pairingsexample3}  for  examples.
From all these considerations we deduce that,
\begin{equation}\label{eq:goalmap0}
\forall w \in \mathcal{A}^\ell, \quad Q(w) \in \Sigma^\ell,
\end{equation}
\begin{equation}\label{eq:goalmap1}
\forall \{x,y\} \in \mathcal{E}_L \quad \forall w \in \mathcal{A}^s( \{x,y\} ), \quad Q(w) \in \Sigma^{\ell},
\end{equation}
\begin{equation}\label{eq:goalmap2}
\forall (x,q), (y,r)  \in \mathcal{E}_L \, : \,   x, y \in \T_L,  \quad \forall w \in \mathcal{A}^w( \, (x,q), (y,r)  \,  \big ), \quad Q(w) \in \Sigma_{x, y}.
\end{equation}
Moreover, by definition of $\mathcal{W}^1$ we also have that,
\begin{equation}\label{eq:goalmap3}
\forall (x,q), (y,r)  \in \mathcal{E}_L \, : \,  \{ x, y \} \cap \T^{(2)}_L \neq \emptyset,     \quad  \mathcal{A}^w( \, (x,q), (y,r)  \,  \big ) = \emptyset.
\end{equation}
We will now prove all the claims in the statement of Lemma \ref{lemma:componentsidentification}
one by one using such properties.

\textbf{\textit{Proof of (\ref{eq:formula1}).}}
From (\ref{eq:goalmap0}) and from the considerations made in the paragraph before (\ref{eq:goalmap0}) we deduce that,
\begin{equation}\label{eq:proofZ0}
\forall \sigma \in \Sigma^\ell,  \quad \quad 
\mu_{L, N, \lambda, H} \Big ( \mathcal{A}^\ell   \cap   \{ Q(w) = \sigma \}       \Big )  = { \Big   (\frac{1}{2} \Big )}^{\mathcal{D}(\sigma)} \, \, \lambda^{|\sigma| + \mathcal{D}(\sigma)} \, \, N^{ \mathcal{L}(\sigma)}.
\end{equation}
The factor $N^{\mathcal{L}(\sigma)}$ above {takes into account} the fact that
if  $w^\prime$  is obtained from $w$ by changing the colour of all the links belonging to  the same path, then $Q(w) = Q(w^\prime)$,
the term
$|\sigma| + \mathcal{D}(\sigma)$ corresponds to the  number of links in each configuration 
$w \in \mathcal{A}^\ell$ such that $Q(w) = \sigma$, and the factor 
$(1/2)^{\mathcal{D}(\sigma)}$ comes from the term $\frac{1}{m_e!}$ in the definition (\ref{eq:measure}).
Now note that, 
\begin{align*}
\mu_{L, N, \lambda, H} \big (  \mathcal{A}^\ell   \big ) 
& = 
\sum\limits_{  \sigma \in \Sigma^\ell}
\mu_{L, N, \lambda, H}  \Big (  \mathcal{A}^\ell  \, \cap  \,  \{ Q(w) = \sigma \}  \Big ) \\
& = 
\sum\limits_{  \sigma \in \Sigma^\ell}   { \Big   (\frac{\lambda}{2} \Big )}^{\mathcal{D}(\sigma)} \, \, \lambda^{|\sigma|} \, \, N^{ \mathcal{L}(\sigma)} =
\mathbb{Y}^{\ell}_{L, N, \lambda}.
\end{align*}
For the first identity we used (\ref{eq:disjointness}) and (\ref{eq:goalmap0}), {for the second we used (\ref{eq:proofZ0}), and for the last we used (\ref{eq:relationclosedpartition}).}
This concludes the proof of (\ref{eq:formula1}).

\textbf{\textit{Proof of (\ref{eq:formula2}).}}
Recall that,
if $\{x,y\}$ belongs to the edge set of $\sigma \in \Sigma$, we write $\{x,y\} \in \sigma$.
In the whole proof we fix an arbitrary undirected edge $\{x,y\} \in \mathcal{E}_L$. To begin, we claim that for any  $\sigma \in \Sigma^\ell$, 
\begin{equation}\label{eq:cons1}
\begin{aligned}
&\Big |   \big \{    w \in \mathcal{A}^s( \{x,y\}  )   \, : \, Q(w) = \sigma \,  \big \}   \,      \Big |
\\ &\qquad =
\begin{cases}
3 \,  \, N^{ \mathcal{L}(\sigma)  + 1 } \quad & \mbox{ if $\sigma$ has {an isolated edge} at $\{x,y\}$ and $\{x,y\} \in \E_L$} \\
2 \, \, N^{ \mathcal{L}(\sigma) +1  } \quad  & \mbox{ if $\{x,y\}$ belongs to a loop of $\sigma$ and $\{x,y\} \in \E_L$} \\
 1 \, \, N^{ \mathcal{L}(\sigma) +1 } \quad  & \mbox{ if $\{x,y\} \notin \sigma$.} \\
0 \quad  & \mbox{ otherwise.} \\
\end{cases}
\end{aligned}
\end{equation}
{We now prove} (\ref{eq:cons1}), starting from the \textit{fourth case }of  (\ref{eq:cons1}) (`otherwise'), which is when   $\{x,y\} \in \mathcal{E}_L \setminus \E_L$  and $\{x,y\}$ belongs to a loop or {an isolated edge} of $\sigma$.
In this case
 ${\mathcal{A}^s(\{x,y\})} \cap \{Q(w) = \sigma\} = \emptyset$, since 
for any $w \in \mathcal{W}^1$, no 
  double link or $\ell$-loop  is allowed to \textit{touch} a virtual vertex.
This explains why we get zero in the fourth case of (\ref{eq:cons1}).

We now consider the first three cases. To begin, note that the factor $N^{\mathcal{L}(\sigma)+1}$ in the first three cases {takes into account} the fact that if  $w^\prime$  is obtained from $w$ by changing the colour of all the links belonging to  the same path, then $Q(w) = Q(w^\prime)$.
 \begin{figure}
  \centering
    \includegraphics[width=\textwidth]{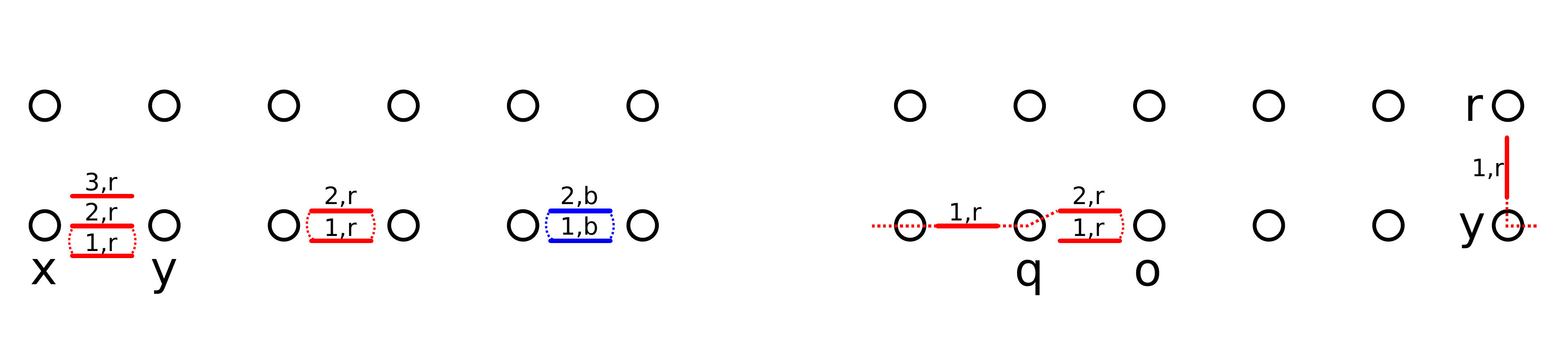}
      \caption{
      Two copies of the vertex set of the graph $(\mathcal{T}_L, \mathcal{E}_L)$ when $d = 1$, with
       $\mathcal{T}_L =  \{ -2, \dots, 3 \} \times \{1,2\}$.
On each copy a realisation $w \in \mathcal{W}^1$ is represented, 
each link has two possible colours, red or blue, and a dotted line connects end-points of paired links.
\textit{Left:}
 {A realisation }$w \in \mathcal{A}^s \big ( \{  x,y \}  \big )$, 
  such that $Q(w) \in \Sigma$ consists of three isolated edges and {no monomers}.
   \textit{Right:} A  realisation  $w \in \mathcal{A}^w \big ( (o,q), (y,r)  \big )$ such that $Q(w) \in \Sigma$ consists of one walk composed of {three edges and two monomers}. 
  }\label{Fig:Pairingsexample1}
\end{figure}
The factors $3$, $2$ or $1$ in the first three cases above {take into account} the number of possible labels of the link belonging to the  segment and which is on $\{x,y\}$. We explain this starting from the first case.
In the \textit{first case}, when $\sigma$ has {an isolated edge} at $\{x,y\}$,  each configuration $w\in Q^{-1}(\sigma) \cap \mathcal{A}^s( \{x,y\}   )$ has three links on $\{x,y\}$, where two of such three links are paired to each other and compose a double link, while the third link  is unpaired at both its end-points. Such an unpaired link might be the first, the second or the third link  on $\{x,y\}$. This situation is represented for example on the left of  Figure \ref{Fig:Pairingsexample1}.
Thus, the factor $3$ {takes into account} the fact that the unpaired link might have three distinct possible labels (in other words, it might occupy three distinct possible positions on $ \{x,y\}$), with each label corresponding to a distinct configuration $w$ such that $Q(w) = \sigma$.
In the \textit{second case},  when  $\{x,y\}$ belongs to a loop of $\sigma$,  each $w \in Q^{-1}(\sigma) \cap \mathcal{A}^s( \{x,y\})$ 
has  two links on $\{x,y\}$, with one link belonging to the segment and thus being unpaired at both its end-points and the other link being paired both at $x$ and $y$. Thus, the factor two {takes into account} the fact that there are two choices for which link on $\{x,y\}$ belongs to the segment and which link on $\{x,y\}$  is paired at both its end-points.
Finally, in  the \textit{third case} we have no entropy factor.
From these considerations and from the definition  of $\mu$,  which is given in
Definition \ref{def:measure},  and the definition of $H$, which is given in Definition \ref{def:definitionH},
 we also deduce that, for any
$\sigma \in \Sigma^\ell$, for any  $w \in  \mathcal{A}^s(\{x,y\} )$ such that $Q(w) = \sigma$,
\begin{multline}\label{eq:cons2}
\mu_{L, N, \lambda, H} (w) = 
\begin{cases}
\frac{1}{3!} \,  { \frac{1}{2^{\mathcal{D}(\sigma)-1}} } \,  \lambda^{|\sigma| + \mathcal{D}(\sigma) + 1} \quad &  \mbox{ if $\sigma$ has {an isolated edge} at $\{x,y\}$ and $\{x,y\} \in \E_L$,} \\
 \frac{1}{2} \,   { \frac{1}{2^{\mathcal{D}(\sigma)}} } \, \, \lambda^{|\sigma| + \mathcal{D}(\sigma)+1}  \quad &  \mbox{ if $\{x,y\}$ belongs to a loop of $\sigma$ and $\{x,y\} \in \E_L$,} \\
  \frac{1}{2} \,    {\frac{1}{2^{\mathcal{D}(\sigma)}} } \, \, \lambda^{|\sigma| + \mathcal{D}(\sigma)+1}   \quad  & \mbox{ if $\{x,y\} \not\in \sigma$ and $\{x,y\} \in \mathcal{E}_L \setminus \E_L$,} \\
    { \frac{1}{2^{\mathcal{D}(\sigma)} } }\, \, \lambda^{|\sigma| + \mathcal{D}(\sigma) + 1}   \quad  & \mbox{ if $\{x,y\} \not\in \sigma$ and $\{x,y\} \in \E_L$.} \\
\end{cases}
\end{multline}
In all the cases above, the last factor corresponds to the weight of the links, whose number is $|\sigma| + \mathcal{D}(\sigma) + 1$.
The first two factors in the first two cases, the second factor in the third case and the first factor in the last case 
{follow} from the term $\frac{1}{m_e!}$ in the definition of $\mu$,
the first factor $\frac{1}{2}$ in the third case comes from the fact that the weight function $H_x$, $x \in \T_L$ assigns a factor $\frac{1}{2}$ whenever it `sees' a link on $\{x,x + \boldsymbol{e}_{d+1}\}$
which is unpaired at $x$ and this can only happen when 
such a link is unpaired at $x$ and at 
${x+\boldsymbol e_{d+1}}$, thus being a segment.
From (\ref{eq:cons1}) and (\ref{eq:cons2}) we deduce that,
for any $w \in \mathcal{A}^s( \{x,y\})$, 
for any $\sigma \in \Sigma^\ell$,
\begin{equation}\label{eq:proofZ1}
\mu_{L, N, \lambda, H }\Big ( \mathcal{A}^s(\{x,y\})   \,\cap  \,   \{ Q(w) = \sigma \}    \, \,   \big )   =  
\begin{cases}
\lambda \,  N \, \, \lambda^{|\sigma| + \mathcal{D}(\sigma)} \, \, 
(\frac{1}{2})^{\mathcal{D}(\sigma) }
\, \, N^{ \mathcal{L}(\sigma)} & \mbox{ if $\{x,y\} \in \E_L$} \\
\frac{\lambda}{2} \, N \,  \, \, \lambda^{|\sigma| + \mathcal{D}(\sigma)} \, \,(\frac{1}{2})^{\mathcal{D}(\sigma) } \, \,  N^{ \mathcal{L}(\sigma)} & \mbox{ if  $\{x,y\} \in \mathcal{E}_L \setminus \E_L$.}
\end{cases}
\end{equation}
From (\ref{eq:disjointness}), (\ref{eq:goalmap1}),
(\ref{eq:cons2}), and (\ref{eq:proofZ1}) we deduce that, when $\{x,y\} \in \E_L$,
\begin{align*}
\mu_{L, N, \lambda, H } \Big (  \mathcal{A}^s(\{x,y\})  \Big ) & = 
 \sum\limits_{  \sigma \in \Sigma^\ell}
\mu_{L, N, \lambda, H  } \Big (\mathcal{A}^s( \{x,y\}   )  \cap   \{ Q(w) = \sigma \}  \Big ) \\
& = 
\lambda \, N  \,    \sum\limits_{  \sigma \in \Sigma^\ell}   { \Big   (\frac{\lambda}{2} \Big )}^{\mathcal{D}(\sigma)} \, \, \lambda^{|\sigma|} \, \, N^{ \mathcal{L}(\sigma)} = \, \lambda \, N \,
{\mathbb{Y}^\ell_{L, N, \lambda}}
\end{align*}
and that the same holds true with a factor one-half in front of the two last terms when 
 $\{x,y\} \in \mathcal{E}_L \setminus \E_L$.
 
\textbf{\textit{Proof of (\ref{eq:formula3}) when 
 $\{x,y\} \cap  \T_L^{(2)}  \neq \emptyset $.}}
 In this case, the proof follows immediately from (\ref{eq:goalmap3}).

 \textbf{\textit{Proof of (\ref{eq:formula3}) when $\{x,y\} \subset \T_L$}.}
Suppose that $\{x,y\} \subset \T_L$ (possibly $x = y$).
Without loss of generality (by translation invariance) fix  $x = o$.
From (\ref{eq:goalmap2}) and from the properties of the map $Q$
we claim that, under these assumptions,
for any $y \in \T_L$ and  $\sigma \in \Sigma_{o,y}$, we have that,
\begin{multline}\label{eq:cons3}
\Big |   \big \{    w \in \mathcal{A}^w( (o,q),  (y,r)  )   \, : \, Q(w) = \sigma \,   \big \}   \,      \Big | =  \\
\begin{cases}
  2^{ \mathbbm{1}   \{ \{o,q\} \in \sigma   \} } \, \, 
 2^{    \mathbbm{1}   \{ \{y,r\} \in \sigma   \} } \, \,  N^{ \mathcal{L}(\sigma)  }  & \mbox{ if $y \neq o$ \mbox{ and } $(y,r) \neq (q,o)$ } \\
 6 \,   N^{\mathcal{L}(\sigma)}  \, \,    &  \mbox{ if $y \neq o$, $(y,r) = (q,o)$ and $\{o,y\} \in \sigma$} \\
 2  \,  N^{\mathcal{L}(\sigma) } \, \,       & \mbox{ if $y \neq o$, $(y,r) = (q,o)$ and $\{o,y\} \not\in \sigma$,} \\
  N^{\mathcal{L}{(\sigma)} + 1 } \,  & \mbox{ if $y=o$.}
\end{cases}
\end{multline}
We now explain (\ref{eq:cons3}).
The  factors $N^{\mathcal{L}(\sigma)} $ and $N^{\mathcal{L}(\sigma)+1}$
in all the cases  above {take into account} the fact that
if  $w^\prime$  is obtained from $w$ by changing the colour of all the links belonging to  the same path, then $Q(w) = Q(w^\prime)$.
We now explain the remaining factors considering case by case.
 \begin{figure}
  \centering
    \includegraphics[width=\textwidth]{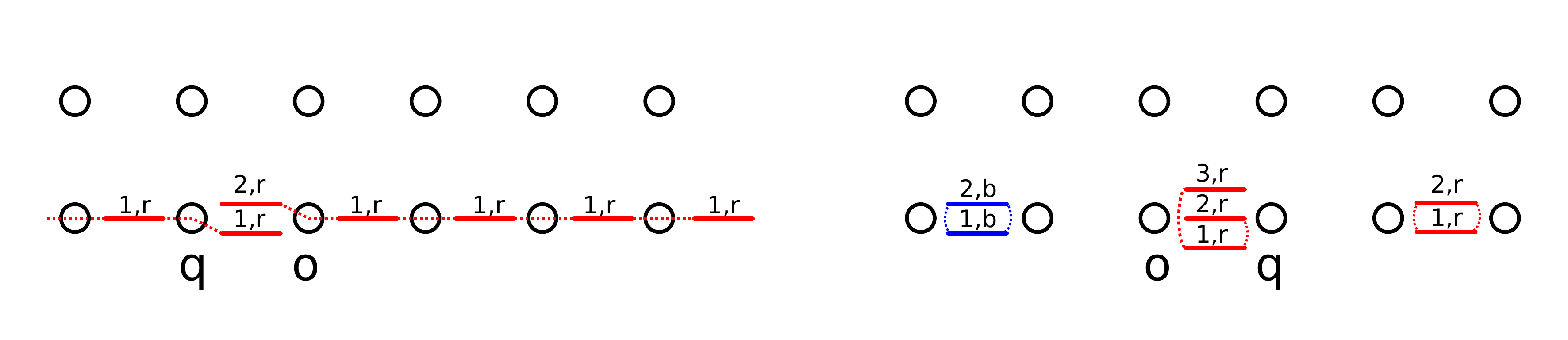}
      \caption{
     Same setting as in Figure \ref{Fig:Pairingsexample1}.      
{\textit{Left:}}
A realisation $w \in \mathcal{A}^w((o,q), (q,o))$
such that $\{ o,q \} \not\in Q(w)$ and $Q(w)$ consists of one walk composed of five edges and {no monomers}.
      {\textit{Right:}} A realisation $w \in \mathcal{A}^w((o,q), (q,o))$
such that $\{ o,q  \} \in Q(w)$ and such that $Q(w) \in \Sigma_{o,q}$ consists of three isolated edges.  }\label{Fig:Pairingsexample2}
\end{figure}
\begin{itemize}
\item Let us explain the first case: $y \neq o$, and $(y,r) \neq (q,o)$.
Note that, from the properties of the map $Q$, it follows that 
for any $w \in \mathcal{A}^w( (o,q),  (y,r)  )$ such that 
 $Q(w) = \sigma$, $\{o,q\} \in \sigma$ if and only if 
two links of the unique $\ell$-walk in $w$ are on  $\{o,q\}$,
one of which is extremal.
Note also that the same claim holds true 
if we replace $\{o,q\}$ by 
 $\{y,r\}$.
Thus, the factors $2^{ \mathbbm{1}   \{ \{o,q\} \in \sigma   \} }$
and $2^{ \mathbbm{1}   \{ \{y,r\} \in \sigma   \} }$
account for the fact that there are two possibilities for choosing {which of the two links is} the extremal one (the other link belongs to the $\ell$-walk, but it is not extremal). For example, if  $w^1$  is the configuration in the right of Figure \ref{Fig:Pairingsexample1}, $\sigma$ is such that $Q(w^1) = \sigma$, 
and $w^2$ is the configuration which is obtained from $w^1$ by exchanging 
the pairing at the vertex $q$ in such a way that the link $(\{q,o  \}, 1)$
is paired at $q$ to  the link $(\{ q-\boldsymbol{e}_1, q  \}, 1)$
and $(\{q,o  \}, 2)$ is unpaired at $q$,
then also $Q(w^2) = \sigma$.
From these considerations we also deduce that, if 
$y \neq o$, and $(y,r) \neq (q,o)$, for any 
$\sigma \in \Sigma_{o,y}$ and 
$w\in {\mathcal A^w((o,q),(y,r))}$ such that $ Q(w) = \sigma$,
\begin{equation}\label{eq:consA}
\mu_{L, N, \lambda, H}(w) = 
\frac{1}{2^{\mathbbm{1}   \{ \{o,q\} \in \sigma\} + \mathbbm{1}   \{ \{y,r\} \in \sigma   \}}}
 \frac{1}{2^{   \mathcal{D}^\prime(\sigma) }}  \, \, 
\lambda^2 \, \,  
\lambda^{ | \sigma | + \mathcal{D}^\prime(\sigma)}
\end{equation}
where the first and the second {factors} {follow} from the term $\frac{1}{m_e!}$ in the definition of $\mu$,  the factor $\lambda^2$ corresponds to the weight of the two extremal links, and the last factor corresponds to the weight of all the remaining links.

\item Let us explain the second case: $y \neq o$, $(y,r) = (q,o)$ and $
\{o,q\} \in \sigma$. 
In this case, any $w \in \mathcal{A}^w( (o,q),  (q,o))$ is such that the $\ell$-walk consists of three links which are  on $\{o,q\}$
and there are precisely three links on $\{o,q\}$.
Thus, one link of the $\ell$-walk must be paired at both its end-points to the two other links of the $\ell$-walk, while the two remaining links are paired at  one end-point and unpaired at the other end-point.
An example of such configuration is represented in  Figure \ref{Fig:Pairingsexample2} - right.
The factor six in the right-hand side of (\ref{eq:cons3}) 
accounts for the fact that there are three distinct possibilities 
for choosing which of such three links is paired at both end-points
and, once this has been chosen, there are two possibilities for choosing which of the two remaining links is paired at $o$ and unpaired at $q$.
From these considerations we also deduce that, 
for any $\sigma \in \Sigma_{o,q}$, such that $\{o, q\} \in \E_L$ and
$\{o,q\} \in \sigma$,
for any $w \in \mathcal{A}^w( (o,q),  (q,o))$ such that $Q(w) = \sigma$,
\begin{equation}\label{eq:consB}
\mu_{L, N, \lambda, H}(w) = 
\frac{1}{3!} \, 
\frac{1}{2^{   \mathcal{D}^\prime(\sigma) }}  \, \, 
\lambda^2 \, \,  
\lambda^{ | \sigma | + \mathcal{D}^\prime(\sigma)}
\end{equation}
where the first and the second {factors} {follow} from the term $\frac{1}{m_e!}$ in the definition of $\mu$,  the factor $\lambda^2$ corresponds to the weight of the two extremal links of the $\ell$-walk, 
and the last factor corresponds to the weight of all the remaining links.

\item Let us explain the third case: $y \neq o$, $(y,r) = (q,o)$ and $
\{o,q\} \notin \sigma$. 
In this case,  any $w \in \mathcal{A}^w( (o,q),  (q,o))$ is such that two links are on $\{o,q\}$, where one of them is unpaired  at $o$ and is paired to another link of the walk at $q$, while the second one is unpaired at $q$ and it is paired to another link of the walk at $o$. 
An example of such configuration is represented in Figure \ref{Fig:Pairingsexample2} - left.
The factor $2$ in the right-hand side of (\ref{eq:cons3}) accounts for the fact that there are two  possibilities for choosing which of the two links is paired at $o$ and which at $q$.
From these considerations we also deduce that, 
for any $\sigma \in \Sigma_{o,q}$, such that $\{o, q\} \in \E_L$ and
${\{o,q\}\notin\sigma}$,
for any $w \in \mathcal{A}^w( (o,q),  (q,o))$ such that  $Q(w) = \sigma$,
\begin{equation}\label{eq:consC}
\mu_{L, N, \lambda, H}(w) = 
\frac{1}{2!} \, 
\frac{1}{2^{   \mathcal{D}^\prime(\sigma) }}  \, \, 
\lambda^2 \, \,  
\lambda^{ | \sigma | + \mathcal{D}^\prime(\sigma)}
\end{equation}
where the first and the second {factors} {follow} from the term $\frac{1}{m_e!}$ in the definition of $\mu$,  the factor $\lambda^2$ corresponds to the weight of the two extremal links of the $\ell$-walk, 
and the last factor corresponds to the weight of all the remaining links.

\item Let us explain the last case: $y = o$.
An example of configuration
$w \in  \mathcal{A}^w( (o,q),  (o,r)  )$
is represented in the left of Figure 
 \ref{Fig:Pairingsexample3} when
 $q \neq r$
and in the right of Figure \ref{Fig:Pairingsexample3} when  $q = r$.
\begin{figure}
  \centering
    \includegraphics[width=\textwidth]{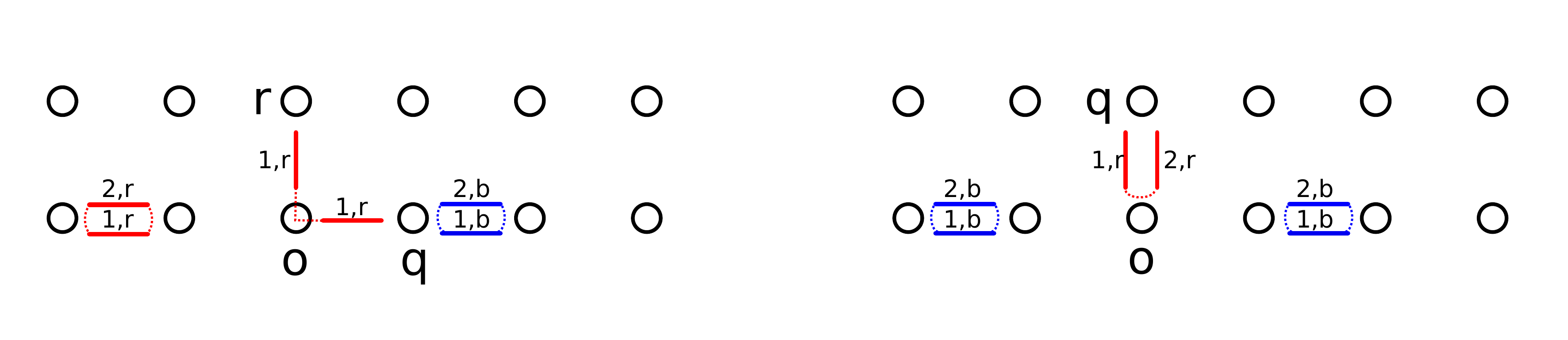}
      \caption{Same setting as in Figure \ref{Fig:Pairingsexample1}.  
 \textit{Left:} A realisation $w \in {\mathcal A^w((o,q),(o,r))}$, $r \neq q$,
 such that $Q(w)$ consists of two isolated edges and {two monomers}.
 \textit{Right:} A realisation $w \in {\mathcal A^w((o,q),(o,q))}$
 such that $Q(w) \in \Sigma_{o,o}$ consists of two isolated edges and {two monomers}.
    }\label{Fig:Pairingsexample3}
\end{figure}
In this case, for any $w \in   \mathcal{A}^w( (o,q),  (o,r)  )$
the unique $\ell$-walk in $w$ consists of just two links which are paired to each other at $o$.
When $q = r$, these links are the only two links on $\{o,q\} = \{o,r\}$,
while when $q \neq r$, each link of the two is the unique link on 
$\{o,q\}$ and $\{o,r\}$.
Since all the other paths are double links or $\ell$-loops,  we deduce (\ref{eq:cons3}).
From these considerations we also deduce that,
for any $\sigma \in \Sigma_{o,o}$,
for any $w\in {\mathcal A^w((o,q),(o,r))}$ {such that $Q(w)=\sigma$},
we have that,
\begin{equation}\label{eq:consD}
\mu_{L, N, \lambda, H}(w) = 
\begin{cases}
\frac{1}{2^{   \mathcal{D}^\prime(\sigma) }}  \, \, 
\lambda^2 \, \,  
\lambda^{ | \sigma | + \mathcal{D}^\prime(\sigma)}
& \mbox{ if $q \neq r$,} 
 \\ 
\frac{1}{2!} \, 
\frac{1}{2^{   \mathcal{D}^\prime(\sigma) }}  \, \, 
\lambda^2 \, \,  
\lambda^{ | \sigma | + \mathcal{D}^\prime(\sigma)}
& \mbox{ if $q = r$ } .
\end{cases}
\end{equation}
where the first factor in the first case and the first two factors in the second case {follow} from the term $\frac{1}{m_e!}$ in the definition of $\mu$,  the factor $\lambda^2$ corresponds to the weight of the two unique links the $\ell$-walk is composed of,  
and the last factor corresponds to the weight of all the remaining links.
\end{itemize}
Now that the multiplicity properties of the map and that the weights assigned by $\mu$ to the configurations $w$ in each of the four cases above have been considered, we can put all the cases together to conclude the proof of (\ref{eq:formula3}).
Below, we use  the general properties of the map $Q$, (\ref{eq:disjointness}) and {(\ref{eq:goalmap2})}, for the first identity,
 (\ref{eq:cons3}),
(\ref{eq:consA}),  (\ref{eq:consB}),
 (\ref{eq:consC}),  (\ref{eq:consD}),
 for the three cases of the second identity, 
 (\ref{eq:relationdirectedpartition}) and
  (\ref{eq:relationdirectedpartition2}) 
   for the three cases of the third and last identity, obtaining that, for any pair of directed edges 
 $(o,q), (y,r) \in \mathcal{E}_L$,
\begin{align*}
& \mu_{ L, N, \lambda, H} \Big (  \mathcal{A}^w  \big (    (o,q), (y,r)   \big )      \Big )  =
 \sum\limits_{  \sigma \in \Sigma_{o,y}}
\mu_{L, N, \lambda, H  } \Big ( \mathcal{A}^w  \big (    (o,q), (y,r)   \big )  \cap   \{ Q(w) = \sigma \}  \Big ) \\
& = 
\begin{cases}
 \lambda^2 \,   \sum\limits_{  \sigma \in \Sigma_{o,y}}   { \Big   (\frac{\lambda}{2} \Big )}^{\mathcal{D}^\prime(\sigma)} \, \, \lambda^{|\sigma|} \, \, N^{ \mathcal{L}(\sigma)} = \, \lambda^2 \, N \,
\mathbb{Y}_{L, N, \lambda}(o,y) \quad &  \mbox{ if $y \neq o$,} \\
 N \lambda^2 \,   \sum\limits_{  \sigma \in \Sigma_{o,y}}   { \Big   (\frac{\lambda}{2} \Big )}^{\mathcal{D}^ \prime(\sigma)} \, \, \lambda^{|\sigma|} \, \, N^{ \mathcal{L}(\sigma)} = \, \lambda^2 \, N \,
\mathbb{Y}_{L, N, \lambda}(o,y)\quad  & \mbox{ if $y = o$, $(o,q) \neq  (y,r)$, }  \\
\frac{N}{2} \, \, \lambda^2 \,   \sum\limits_{  \sigma \in \Sigma_{o,y}}   { \Big   (\frac{\lambda}{2} \Big )}^{\mathcal{D}^\prime(\sigma)} \, \, \lambda^{|\sigma|} \, \, N^{ \mathcal{L}(\sigma)} = \, \frac{ \lambda^2 }{2} \,  N \,
\mathbb{Y}_{L, N, \lambda}(o,y) \quad  &  \mbox{ if $y = o$, $(o,q) = (y,r)$.} 
\end{cases}
\end{align*}
This concludes the proof of Lemma \ref{lemma:componentsidentification}.
\end{proof}

\subsection{Proof of Theorem \ref{theo:keyinequality}}
\label{sect:proofinequailty}
All the ingredients for the proof of Theorem \ref{theo:keyinequality} have been introduced and we can now combine them to present the proof of Theorem \ref{theo:keyinequality}.
\begin{proof}[Proof of Theorem \ref{theo:keyinequality}]
Fix  arbitrary  finite integers $L \in 2 \mathbb{N}_{>0}$, $N \in \mathbb{N}_{>0}$,
fix an edge-parameter $\lambda \in (0, \infty)$.
Recall  that $x \sim y$ denotes that $x$ and $y$ are nearest neighbours in
$(\T_L,\E_L)$, 
and recall that $\sum_{(x,y) \in \mathcal{E}_L}$ is the sum over directed edges while 
$\sum_{\{x,y\} \in \mathcal{E}_L}$ is the sum over undirected edges.
Recall also that $(\T_L,\E_L)$ corresponds to the torus $\mathbb{Z}^d / L \mathbb{Z}^d$
while  $(\mathcal{T}_L,\mathcal{E}_L)$ is the extended torus.
For any real-valued vector $\boldsymbol{v} = (v_x)_{x \in \T_L}$, let 
$\boldsymbol{h}^{\boldsymbol{v}} = ( h_x^{\boldsymbol{v}}   )_{x \in \mathcal{T}_L}$
be obtained from $\boldsymbol{v}$ as follows:
\begin{equation}\label{eq:hasfunctionofv}
\forall x \in \mathcal{T}_L  \quad \quad
h^{\boldsymbol{v}} _x : = 
\begin{cases}
v_x & \mbox{ if $x \in \T_L$,} \\
-2\, d \, v_{x - \boldsymbol{e}_{d+1}} & \mbox{ if $x \in \T_L^{(2)}$.} 
\end{cases}
\end{equation}
Using the fact that for any real-valued vector $\boldsymbol{v} =  (v_x)_{x \in \T_L}$, 
\begin{equation}\label{eq:cardinalityedges}
2 \,  d \,  \sum\limits_{x \in \T_L} v_x^2  = 
 \sum\limits_{\{x,y\} \in \E_L} (v_x^2  + v_y^2)
\end{equation}
we deduce that,
\begin{multline}\label{eq:firstconversion}
\sum\limits_{  \{x,y\} \in \E_L  }
h^{\boldsymbol{v}}_x h^{\boldsymbol{v}}_y 
+ 
\frac{1}{2}  \sum\limits_{ x \in \T_L }
h^{\boldsymbol{v}}_x h^{\boldsymbol{v}}_{x + \boldsymbol{e}_{d+1}   }
=  \\
\sum\limits_{  \{x,y\} \in \E_L  }
v_x v_y 
-  
d \sum\limits_{  x \in \T_L  }
v_x^2 = 
\frac{1}{2} 
\sum\limits_{\{x, y\} \in \E_L}  ( 2 v_x v_y 
- v_x^2 - v_y^2) 
=
- \frac{1}{2} \sum\limits_{  \{ x,y \} \in \E_L  } 
\big (  
v_x - v_y
\big )^2.
\end{multline}
Moreover, 
\begin{equation}\label{eq:secondconversion}
\sum\limits_{q \in \mathcal{T}_L : (x,q) \in \mathcal{E}_L} h^{\boldsymbol{v}}_q =  ( \triangle \boldsymbol{v} )_{x}.
\end{equation}
From (\ref{eq:firstconversion}), (\ref{eq:secondconversion}) and from the definition in Proposition \ref{prop:moments} we deduce that,
for any $\boldsymbol{v} = (v_x)_{x \in \T_L}$,
\begin{equation}\label{eq:beforehomogenized}
\mathcal{Z}_{L, N, \lambda, H}^{(2)}(\boldsymbol{h}^{\boldsymbol{v}})
= - \frac{\lambda N}{2} \mathbb{Y}^\ell_{L, N, \lambda} \sum\limits_{ \{x,y\} \in  \E_L} (v_y-v_x)^2 \, \, 
+ \, \, 
\frac{\lambda^2 \, N}{2} \sum\limits_{x, y \in \T_L} \mathbb{Y}_{L, N, \lambda}(x,y) 
(\triangle \boldsymbol{v})_x \,
(\triangle \boldsymbol{v})_y.
\end{equation}
Moreover, recall that, as defined in Section \ref{sect:Chessboardscheme},
for any original vertex $x \in \T_L$, 
$( \boldsymbol{h}^{\boldsymbol{v}})^x$ is defined as the vector 
which is obtained from $\boldsymbol{h}^{\boldsymbol{v}}$
by copying the value $h^{\boldsymbol{v} }_x = v_x$ at any original vertex and the value
$h^{\boldsymbol{v} }_{ x  + \boldsymbol{e}_{d+1}} = -2d v_x$
at any virtual vertex and deduce from this and from (\ref{eq:beforehomogenized})
that, 
\begin{equation}\label{eq:homogenized}
\forall \boldsymbol{v} = (v_z)_{z \in \T_L}  \quad 
\forall x \in \T_L, \quad  \quad \mathcal{Z}_{L, N, \lambda, H}^{(2)} \Big (( \boldsymbol{h}^{\boldsymbol{v}})^{x} \Big )
= 0.
\end{equation}
We have that, in the limit as $\varphi \rightarrow 0$,
\begin{align*}
\mathcal{Z}_{L, N, \lambda, H}(\varphi \boldsymbol{h}^{\boldsymbol{v}}) & = \mathbb{Y}^{\ell}_{L, N, \lambda}
+ \varphi^2 \mathcal{Z}^{(2)}_{L, N, \lambda, H}(\boldsymbol{h}^{\boldsymbol{v}}) + o(\varphi^2) \\
& \leq \Big ( \prod_{x \in \T_L}
\mathcal{Z}_{L, N, \lambda, H} \big ( (\varphi \boldsymbol{h}^{\boldsymbol{v}})^x \big )
 \Big )^{\frac{1}{|\T_L|}} \\
& = \Big ( \prod_{x \in \T_L} \big ( \mathbb{Y}^{\ell}_{L, N, \lambda}
+ o(\varphi^2) \big )    \Big )^{\frac{1}{|\T_L|}} \\
& = \mathbb{Y}^{\ell}_{L, N, \lambda} + o(\varphi^2),
\end{align*}
For the first step above we used  Proposition \ref{prop:moments}, for the second step above we used Proposition \ref{prop:chessboardscheme},
for the third step above we used Proposition \ref{prop:moments} and (\ref{eq:homogenized}),
for the last step we perform the Taylor expansion around $x=0$ of the function:  $(1 + x)^{1 / |\T_L|}  = 1 + x / |\T_L| + O(x^2)$, where in our case $x = o(\varphi^2)$.
Thus we proved that, for any  $ \boldsymbol{v} \in \mathbb{R}^{\T_L}$, in the limit as $\varphi \rightarrow 0$, 
$$
 \mathbb{Y}^{\ell}_{L, N, \lambda}
+ \varphi^2 \mathcal{Z}^{(2)}_{L, N, \lambda, H}(\boldsymbol{h}^{\boldsymbol{v}}) + o(\varphi^2) \leq   \mathbb{Y}^{\ell}_{L, N, \lambda} + o(\varphi^2),
$$
where $\boldsymbol{h}^{\boldsymbol{v}}$ was defined in (\ref{eq:hasfunctionofv}) as a function of $\boldsymbol{v}$,
and this can only hold true if 
\begin{equation}\label{eq:2lessthan0}
 \mathcal{Z}^{(2)}_{L, N, \lambda, H}(\boldsymbol{h}^{ \boldsymbol{v}})  \leq 0.
\end{equation}
By replacing  (\ref{eq:beforehomogenized}) in the left hand-side of 
(\ref{eq:2lessthan0}), dividing the whole expression by $\frac{ \lambda N }{2} \mathbb{Y}^\ell_{L,N, \lambda}$ and plugging in 
(\ref{eq:twopointlambdarelation}), we deduce that, for any finite strictly positive $\lambda$,
$$
\sum\limits_{x,y \in \T_L} \mathbb{G}_{L, N, \frac{1}{\lambda}}(x,y) (\triangle v )_x \, 
(\triangle v)_{y} \,  
\leq  \, \sum\limits_{\{x,y\} \in \E_L} \big (v_y - v_x \big )^2.
$$
Since the previous relation holds for any strictly positive $\lambda$ and since for any finite $L$, 
$ \lim_{\lambda \rightarrow \infty}$ $ \mathbb{G}_{L, N, 1  / \lambda}(x,y)$ $  = 
\mathbb{G}_{L, N, 0}(x,y)$, we deduce that the same inequality holds true also with $\frac{1}{\lambda}$ replaced by $0$ and thus the proof is concluded.
\end{proof}

\section{A version of the Infrared bound}
\label{sect:InfraredBound}
{
The main goal of this section is to prove a strictly positive lower bound
for the odd-sublattice Ces\`aro average at zero activity. We first retain
the high-frequency estimate, valid also at positive activity, and then
divide the long-range-order argument into $d>3$ and $d=3$.

}
Recall the definition of the \textit{odd} and \textit{even} sub-lattices, (\ref{eq:oddevensublattice}) and define  the \textit{odd} and   \textit{even two-point functions},
\begin{equation}\label{eq:oddtwopoint}
\mathbb{G}^{o}_{L,  N, \rho}(x, y) : = \mathbb{G}_{L,  N, \rho}(x, y) \, \, \mathbbm{1}_{\{ {y-x\in\T_L^o}   \}},
\end{equation}
\begin{equation}\label{eq:eventwopoint}
\mathbb{G}^{e}_{L,  N, \rho}(x, y) : = \mathbb{G}_{L,  N, \rho}(x, y) \, \, \mathbbm{1}_{\{ {y-x\in\T_L^e}    \}}.
\end{equation}
We will use the notation
$$
\mathbb{G}_{L,  N, \rho}(x) : = \mathbb{G}_{L,  N, \rho}(o, x) \quad \quad \mathbb{G}^o_{L,  N, \rho}(x) : = \mathbb{G}^o_{L,  N, \rho}(o, x) \quad \quad \mathbb{G}^e_{L,  N, \rho}(x) : = \mathbb{G}^e_{L,  N, \rho}(o, x),
$$
for any $x \in \T_L$, and we will omit the sub-scripts when appropriate.
To begin, we define the central quantity,
\begin{equation}\label{eq:varepsilonfunction}
\forall k \in \T_L^*, \quad \quad  \varepsilon(k) := 
2 \sum\limits_{j=1}^{d} \big ( \,  1 - \cos (k_j) \,  \big ).
\end{equation}
Recall also the definition of Fourier transform and inverse Fourier transform which were provided in Section \ref{sect:ingredientspart2}.
\begin{proposition}[\textbf{High frequency upper bound}]\label{prop:highfrequencybound}
{For every $d,N\in\mathbb N_{>0}$, $\rho\ge0$ and even $L\ge4$,}
\begin{equation}\label{eq:highfrequencybound}
 \forall k \in \T_L^* \setminus \{o\} \quad \quad 
 \hat{ \mathbb{G}}_{L, N, \rho}(k)\, \, =  \hat{ \mathbb{G}}^o_{L, N, \rho}(k)\ +  \hat{ \mathbb{G}}^e_{L, N, \rho}(k)\  \leq  \, \, 
\frac{1}{ \varepsilon(k)}.
\end{equation}
\end{proposition}
\begin{proof}
To begin, we  fix an arbitrary $k \in \T^*_L \setminus \{o\}$ and choose the vector $\boldsymbol{v} = (v_x)_{x \in \T_L}$ such that, for any $x \in \T_L$,
$ v_x  := \cos( k \cdot x).
$
We note that under this choice the following facts hold true,
\begin{enumerate}[(i)]
\item For any $x \in \T_L$,
$({\triangle v})_x  = - \varepsilon(k) \, \, v_x,$
\item $\sum_{\{x,y\} \in \E_L   } (v_y - v_x)^2 =  \varepsilon(k) \, \,  \sum_{x \in \T_L} v_x^2,$
\item %
$
\sum_{x, y \in \T_L} \, v_x  \, v_y  \,\mathbb{G}(x,y) = \hat{\mathbb{G}}(k) \, \sum_{x \in \T_L}  v_x^2.
$
\end{enumerate}
These computations are classical and we present their proof  in the appendix of this paper.
The proof of Proposition \ref{prop:highfrequencybound}
follows from Theorem \ref{theo:keyinequality} and from such computations.
We first apply  (i) 
to the left-hand side  of (\ref{eq:starting point infrared}), 
then we apply  (ii)  to the right-hand side of 
(\ref{eq:starting point infrared}),  thus obtaining that 
$$
\varepsilon^2(k) \,   \sum\limits_{x,y \in \T_L} \, v_x \, v_y  \, \mathbb{G}(x,y)\,  \leq \,
\varepsilon(k) \, \sum\limits_{x \in \T_L} \,  v_x^2.
$$
Now we apply (iii) to the left-hand side of the previous expression and we divide everything by $\varepsilon^2(k) \sum_{x \in \T_L} v_x^2$. This concludes the proof.
\end{proof}
The next lemma states some properties of the  two-point functions and of their Fourier transforms.
\begin{lemma}\label{lemma:Fourierproperties}
Let $\mathbbm{U}$ be the set of vectors $\boldsymbol{u} : =(u_1, \ldots, u_d) \in \mathbb{Z}^d$ such that  $|u_i| = 1$ for any coordinate $i$. The following properties hold for any  $\boldsymbol{u} \in \mathbbm{U}$,
\begin{enumerate}[(i)]
\item For any $k \in \T_L^*$, $\hat{\mathbb{G}}_{L, N, \rho}(k), \hat{\mathbb{G}}_{L, N, \rho}^e(k), \hat{\mathbb{G}}_{L, N, \rho}^o(k) \in \mathbb{R}$.
\item If $k, k + \pi \boldsymbol{u} \in \T_L^*$, then $\hat{\mathbb{G}}_{L, N, \rho}^o(k  + \pi \boldsymbol{u}) = - \hat{\mathbb{G}}_{L, N, \rho}^o(k),$ 
\item If $k, k +  \pi \boldsymbol{u} \in \T_L^*$, then $\hat{\mathbb{G}}_{L, N, \rho}^e(k + \pi  \boldsymbol{u}) =  \hat{\mathbb{G}}_{L, N, \rho}^e(k).$ 
\item For any
$
L \in 2 \mathbb{N}$ and $ x \in \T_L$, we have that $\mathbb{G}^e_{L, N, 0}(o,x) = 0.
$
\end{enumerate}
\end{lemma}
\begin{proof}
The first property follows from the definition of Fourier transform and the symmetries of  $\mathbb{Z}^d / L \mathbb{Z}^d$. The properties {(ii) and (iii)} follow from the definition of Fourier transform and the fact that,  if $x \in \T_L^o$, then $\sum_{i=1}^d x_i {\in} 2 \mathbb{Z}+1$ and 
if $x \in \T_L^e$,  then $\sum_{i=1}^d x_i {\in} 2 \mathbb{Z}$.
The fourth property holds true since, if the  walk in $\pi \in \Omega$ ends at an even site, then it contains an odd number of sites and,
since  the total number of sites in $\T_L$ is even and since each loop or double edge contains an even number of sites, this implies that 
at least one monomer is present in $\pi$ and thus that the weight of $\pi$ is zero since $\rho = 0$.
\end{proof}
{
\subsection{Long-range order in dimensions \texorpdfstring{$d>3$}{d>3}}
\label{sect:corrected-higher-d}
The proof crucially uses the sign-change symmetry of the Fourier transform
of the odd two-point function under $k\mapsto k+(\pi,\ldots,\pi)$ modulo $2\pi$,
while at zero activity the even two-point function and its Fourier
transform vanish identically (Lemma~\ref{lemma:Fourierproperties}(ii),(iv)).
Recall $h_d$ from \eqref{eq:admissible-N}, and set
\[
 A_{L,N}:=\frac1{|\T_L^o|}\sum_{x\in\T_L^o}\mathbb G_{L,N,0}(x),
 \qquad
 D_d(k):=\frac1d\sum_{j=1}^d\cos k_j
        =1-\frac{\varepsilon(k)}{2d},
\]
with $\varepsilon$ as in \eqref{eq:varepsilonfunction}.

\begin{proposition}\label{prop:corrected-higher-d}
For every $d\ge3$ and positive integer $N$,
\begin{equation}\label{eq:corrected-higher-d-limit}
 \liminf_{\substack{L\to\infty\\L\in2\mathbb N}}A_{L,N}
 \ge\frac1{2d}\left(\frac2N-h_d\right).
\end{equation}
This lower bound is positive when $Nh_d<2$, in particular for $N=2$
in every dimension $d\ge4$ and for $N=1$ when $d=3$.
\end{proposition}

\begin{proof}
Fix $\rho=0$ and even $L\ge4$, and omit the subscripts of $\mathbb G$ as above.
With $p$ as in Lemma~\ref{lemma:relationmodus}, averaging
\eqref{eq:fourierlemma} over coordinate directions (and using that
$\widehat{\mathbb G}$ is even in $k$, by the symmetry $x\mapsto-x$, so that
$e^{ik_j}$ may be replaced by $\cos k_j$) yields
\begin{equation}\label{eq:corrected-neighbour-Fourier}
 \mathbb G(\boldsymbol e_1)=A_{L,N}
 +\frac1{L^d}\sum_{k\in\T_L^*\setminus\{o,p\}}
 D_d(k)\widehat{\mathbb G}(k).
\end{equation}
Since $D_d(k+p)=-D_d(k)$, pair each remaining mode with its translate
by $p$ (modulo $2\pi$) and apply \eqref{eq:highfrequencybound} at the
member with $D_d(k)\ge0$. Indeed, since also
$\widehat{\mathbb G}(k+p)=-\widehat{\mathbb G}(k)$
(Lemma~\ref{lemma:Fourierproperties}(ii),(iv) at $\rho=0$), the two members of
a pair give the same contribution $D_d(k)\widehat{\mathbb G}(k)$; for the
member with $D_d(k)\ge0$, \eqref{eq:highfrequencybound} and
$\varepsilon(k)=2d(1-D_d(k))$ give
$2D_d(k)\widehat{\mathbb G}(k)\le|D_d(k)|/[d(1-|D_d(k)|)]$, and the sum over
the $(L^d-2)/2$ pairs is $\frac12\sum_{k\ne o,p}|D_d(k)|/[d(1-|D_d(k)|)]$.
Using \eqref{eq:boundfirstterm} gives
\begin{equation}\label{eq:corrected-higher-d-finite}
 A_{L,N}\ge\frac1{dN}-\frac{J_L(d)}{2d},\qquad
 J_L(d):=\frac1{L^d}\sum_{k\in\T_L^*\setminus\{o,p\}}
                  \frac{|D_d(k)|}{1-|D_d(k)|}.
\end{equation}
These full-zone punctured Riemann sums converge for $d\ge3$: near
$o,p$ the integrand is of order at most distance$^{-2}$, and summing
over grid shells bounds the contribution within distance $\delta$
by $O(\delta^{d-2})$, uniformly in $L$.

The random-walk Fourier formula and the symmetry
$k\mapsto k+p$ imply
\begin{equation}\label{eq:rw-green-integral}
 \frac1{(2\pi)^d}\int_{[-\pi,\pi]^d}\frac{\mathrm dk}{1-D_d(k)^2}
 =\frac1{(2\pi)^d}\int_{[-\pi,\pi]^d}\frac{\mathrm dk}{1-D_d(k)}=1+r_d.
\end{equation}
Indeed, $D_d(k)$ is the characteristic function of one step of the
simple random walk, so that
$P^d(S_n=o)=(2\pi)^{-d}\int_{[-\pi,\pi]^d}D_d(k)^n\,\mathrm dk$; summing over
$n\ge0$ gives the second identity, since
$\sum_{n\ge0}P^d(S_n=o)=1+E^d(N_+)=1+r_d$ (to exchange sum and integral one
inserts a factor $s^n$, with $s\in(0,1)$, and lets $s\uparrow1$: the
integrand $[1-sD_d(k)]^{-1}$ is increasing in $s$ where $D_d(k)\ge0$ and
bounded by one where $D_d(k)<0$). The first identity follows from
$\frac1{1-D_d(k)^2}=\frac12\big[\frac1{1-D_d(k)}+\frac1{1+D_d(k)}\big]$ and
$D_d(k+p)=-D_d(k)$.
Consequently Cauchy--Schwarz gives
\[
 \lim_{L\to\infty}J_L(d)
 =\frac1{(2\pi)^d}\int_{[-\pi,\pi]^d}
 \left(\frac{D_d(k)^2}{1-D_d(k)^2}+\frac{|D_d(k)|}{1-D_d(k)^2}\right)\mathrm dk
 \le r_d+\sqrt{r_d(1+r_d)}=h_d,
\]
which proves \eqref{eq:corrected-higher-d-limit}.
Here we used $\frac{|D_d(k)|}{1-|D_d(k)|}
=\frac{D_d(k)^2+|D_d(k)|}{1-D_d(k)^2}$; by \eqref{eq:rw-green-integral}, the
first term integrates to $(1+r_d)-1=r_d$, and the Cauchy--Schwarz inequality
is applied to the second term, written as the product of
$|D_d(k)|\,(1-D_d(k)^2)^{-1/2}$ and $(1-D_d(k)^2)^{-1/2}$, whose squares
integrate to $r_d$ and $1+r_d$ respectively.

\bis{The function $D_d(k)$ is the average over $j=1,\ldots,d$ of
$(d-1)^{-1}\sum_{i\ne j}\cos k_i$\blu{, namely of $D_{d-1}$ evaluated at the
vector obtained from $k$ by omitting the coordinate $k_j$}. Thus Jensen's
inequality for
$t\mapsto(1-t)^{-1}$ and integration \blu{over $k$, together with}
\eqref{eq:rw-green-integral} \blu{in dimensions $d$ and $d-1$,}
give $r_d\le r_{d-1}$ for $d\ge4$.
By Hara, Slade and Sokal~\cite[Table~IV]{HaraSladeSokal},
$1+r_4<4/3$, so $r_d<1/3$ and hence $h_d<1$ for all $d\ge4$.
Watson's estimate $0.51<r_3<0.52$~\cite{Watson} gives $1<h_3<2$.
Finally, the expansion in \cite[Appendix~A]{HaraSladeSokal} gives
$r_d=(2d)^{-1}+O(d^{-2})$, and therefore $h_d\sim(2d)^{-1/2}$ as $d\to\infty$.}
\end{proof}
\rev{For $d=3$ and $N=2$ the bound \eqref{eq:corrected-higher-d-limit} is
void, since $h_3>1$ makes its right-hand side negative; this case is treated
in the next subsection\blu{, where the bound \eqref{eq:highfrequencybound}
is improved by the factor $2/3$}.}
}

{
\subsection{Long-range order in dimension three}
\label{rev3d-section}
We study the dimer model on $\T_L$ by viewing it as $L$ horizontal
layers $V\times\{t\}$, where $V=(\mathbb Z/L\mathbb Z)^2$ and
$0\le t<L$; throughout this subsection $d=3$ and $L\ge4$ is even.
Accordingly, in this subsection a point of $\T_L$ is denoted by
$z=(x,t)$, where $x\in V$ is its horizontal coordinate and
$t\in\{0,\ldots,L-1\}$ is its vertical coordinate (its layer), and a point of
$\T_L^*$ is denoted by $k=(q,\theta)$. The letters $x,y$ are reserved for
points of $V$ and the letter $z$ for points of $\T_L$; thus
$\Xi_L(z)=\Xi_L((x,t))$ is the monomer-monomer correlation
\eqref{eq:monomermonomercorr}, with one monomer at the origin and the other
one at horizontal displacement $x$ and vertical displacement $t$.
The horizontal wavevectors lie in
\[
 Q_L=\frac{2\pi}{L}\{0,\ldots,L-1\}^2.
\]
For $q\in Q_L$, define
\[
 F_q(t)=\sum_{x\in V}e^{-iq\cdot x}\Xi_L((x,t)),\qquad
 \widehat\Xi_L(q,\theta)=\sum_{t=0}^{L-1}e^{-i\theta t}F_q(t),
\]
where $\theta\in(2\pi/L)\{0,\ldots,L-1\}$ and wavevector sums are
understood modulo $2\pi$. Thus $F_q(t)$ is the
\textit{transverse Fourier transform} of the two-point function: we
transform the horizontal separation $x$ between the monomers while
keeping their vertical separation $t$ fixed. Transforming also in $t$
gives the full Fourier transform $\widehat\Xi_L(q,\theta)$, which is
the Fourier transform \eqref{eq:Fourierdefinition} of $\Xi_L$ evaluated at
$k=(q,\theta)$; we write indifferently $\widehat\Xi_L(k)$ or
$\widehat\Xi_L(q,\theta)$. In particular
$\widehat\Xi_L(0,0)=\sum_{z\in\T_L}\Xi_L(z)$. Both are real by
reflection symmetry.
We work directly with $\Xi_L$ and use
\eqref{eq:relationdimerpermutationstat} for the case $N=2$ of lattice
permutations. Put
\begin{equation}\label{rev3d-mean-definition}
 A_L=\frac1{|\T_L^o|}\sum_{z\in\T_L^o}\Xi_L(z)
     =\frac2{L^3}\sum_{z\in\T_L}\Xi_L(z)
     =\frac2{L^3}\,\widehat\Xi_L(0,0).
\end{equation}
\rev{The second identity holds since $|\T_L^o|=L^3/2$ and $\Xi_L$ vanishes
on the even sublattice $\T_L^e$, as noted after
\eqref{eq:monomermonomercorr}.}
The matching condition and cubic symmetry give $\Xi_L(\boldsymbol e_i)=1/6$.
Lemma~\ref{lemma:Fourierproperties}(ii),(iv) and
\eqref{eq:relationdimerpermutationstat} give
\begin{equation}\label{rev3d-parity}
 \widehat\Xi_L(k+\pi\mathbf1)=-\widehat\Xi_L(k),\qquad k\in\T_L^*,
\end{equation}
where $\mathbf 1=(1,1,1)$.

\paragraph{Strategy.}
\bis{Counting matchings layer by layer gives a spectral representation
of $F_q(t)$ with nonnegative coefficients. Jensen's inequality then compares
two Fourier modes that differ in just one coordinate. Applying this
comparison along the three axes gives the isotropic infrared bound
\eqref{bis-isotropic}. Its coefficient $2/3$ is small enough to give a
positive lower bound for $A_L$ in the Fourier sum of
Section~\ref{sect:corrected-higher-d}\blu{: with the coefficient one of
\eqref{eq:highfrequencybound} that argument requires $h_3<1$, which is
false, while with the coefficient $2/3$ it requires $\frac23h_3<1$}.}

\subsubsection{Layer matching counts}\label{rev3d-layers}
We encode matchings and monomer insertions by a layer transfer matrix
and compute the local contact term that enters the spectral estimate.

For $A\subseteq V$, let $\rev{\nu}(A)$
count perfect matchings of the induced graph on the retained vertices $A$;
in particular $\rev{\nu}(\varnothing)=1$. Define
\[
 w(S,T)=\mathbf1_{S\cap T=\varnothing}\,\rev{\nu}(V\setminus(S\cup T)),
 \qquad S,T\subseteq V.
\]
Here $S,T$ record the vertical dimers on the two sides of a layer.
Gluing layers gives the finite sum
\begin{equation}\label{rev3d-eq-addendum-row-partition}
 |\mathcal D(\emptyset)|
 =\sum_{S_0,\ldots,S_{L-1}\subseteq V}\prod_{j=0}^{L-1}w(S_j,S_{j+1}),
 \qquad S_L=S_0.
\end{equation}
We regard $W=(w(S,T))_{S,T\subseteq V}$ as a square matrix whose rows
and columns are indexed by the subsets of $V$. All the kernels introduced
below are matrices of this kind; they are multiplied as matrices,
$(AB)(S,T)=\sum_{U\subseteq V}A(S,U)B(U,T)$, and
$A^{\mathsf T}(S,T)=A(T,S)$ denotes the transpose. In this notation
\eqref{rev3d-eq-addendum-row-partition} reads
$|\mathcal D(\emptyset)|=\operatorname{Tr}(W^L)$.
The real symmetric matrix $W=(w(S,T))$ is invertible: complementing
the columns gives $w(S,V\setminus T)=\mathbf1_{S\subseteq T}\rev{\nu}(T\setminus S)$,
which is upper triangular, with diagonal one, when subsets are ordered
by cardinality.

The insertion kernel
\[
 b_x(S,T)=\mathbf1_{\{x\notin S,\ T=S\cup\{x\}\}}
\]
satisfies $b_xW=Wb_x^{\mathsf T}$, namely the matrices $b_xW$ and
$Wb_x^{\mathsf T}$ have the same entry $(S,T)$ for all $S,T\subseteq V$.
Indeed, $b_x(S,U)$ is different from zero only if $x\notin S$ and
$U=S\cup\{x\}$, so that the sums defining the two matrix products reduce to
a single term,
\[
\begin{aligned}
 (b_xW)(S,T)&=\sum_{U\subseteq V}b_x(S,U)\,w(U,T)
   =\mathbf1_{\{x\notin S\}}\,w(S\cup\{x\},T),\\
 (Wb_x^{\mathsf T})(S,T)&=\sum_{U\subseteq V}w(S,U)\,b_x(T,U)
   =\mathbf1_{\{x\notin T\}}\,w(S,T\cup\{x\}),
\end{aligned}
\]
and, by the definition of $w$, both right-hand sides are equal to
\[
 \mathbf1_{\{S,T,\{x\}\ \text{pairwise disjoint}\}}
 \rev{\nu}(V\setminus(S\cup T\cup\{x\})).
\]
In words, if $b_x$ is inserted between two layers and $W$ is applied
afterwards, the product $b_xW$ is the weight of a layer in which $x$ is a
monomer: $x$ is removed from the vertices to be matched horizontally, and it
carries no vertical dimer on either side; the same weight is obtained by
applying $W$ first and $b_x^{\mathsf T}$ afterwards. Two insertions at
distinct sites produce a layer with two monomers, while $b_x^2=0$,
consistently with $\Xi_L(o)=0$. The identity $b_xW=Wb_x^{\mathsf T}$ is the
\rev{key} property of the dimer model used in the spectral argument of
Lemma~\ref{bis-lem-representation}: together with the symmetry of $W$ it
is responsible for the nonnegativity of the coefficients in
\eqref{rev3d-eq-transverse-representation}.
For real $h:V\to\mathbb R$, define the weighted insertion matrix
\begin{equation}\label{rev3d-eq-weighted-insertion}
 b_h:=\sum_{x\in V}h_xb_x.
\end{equation}
We call the following weighted sum of two-monomer correlations,
with monomers in layers $0$ and $t$, the \textit{source sum}:
\begin{equation}\label{rev3d-eq-addendum-row-source}
\begin{split}
 F_h(t)&:=\sum_{x,y\in V}h_xh_y\Xi_L((y-x,t))\\
 &=\frac1{|\mathcal D(\emptyset)|}\sum_{S_0,\ldots,S_{L+1}\subseteq V}
 b_h(S_0,S_1)\left[\prod_{j=1}^{t}w(S_j,S_{j+1})\right]
 b_h(S_{t+1},S_{t+2})
 \left[\prod_{j=t+2}^{L+1}w(S_j,S_{j+1})\right],
\end{split}
\end{equation}
where $0\le t<L$, $S_{L+2}=S_0$, and empty products are one.
In matrix notation, the second line of
\eqref{rev3d-eq-addendum-row-source} is
$\operatorname{Tr}\big(b_hW^{t}\,b_hW^{L-t}\big)/|\mathcal D(\emptyset)|$.
To see why the second line of \eqref{rev3d-eq-addendum-row-source} counts
dimer covers with two monomers, \bis{first take $1\le t<L$.} \bis{The} factor $b_h(S_0,S_1)w(S_1,S_2)$
produces, with weight $h_x$, a layer (call it layer $0$) in which $x$ is a
monomer and whose vertical dimers are $S_0$ below and $S_2$ above; the
factors $w(S_2,S_3),\ldots,w(S_t,S_{t+1})$ produce the ordinary layers
$1,\ldots,t-1$; the factor $b_h(S_{t+1},S_{t+2})w(S_{t+2},S_{t+3})$ produces
layer $t$ with a monomer at $y$ and weight $h_y$; the remaining $L-t-1$
factors are the ordinary layers $t+1,\ldots,L-1$, and the condition
$S_{L+2}=S_0$ closes the torus. There are $L$ factors $w$ in total, the two
monomers sit in layers $0$ and $t$, and translation invariance in $V$ gives
the first line. For $t=0$ the two insertions act on the same layer, and $b_x^2=0$
correctly gives no contribution from $x=y$.

\bis{The contact sum $K_h$ in \eqref{rev3d-eq-addendum-row-contact}
is determined by the occupation probability of one vertical edge.}
For later use define the kernel
\[
 k_h(S,T)=\sum_U[b_h(S,U)b_h(T,U)-b_h(U,S)b_h(U,T)]
 =\mathbf1_{S=T}\sum_xh_x^2(1-2\mathbf1_{x\in S}).
\]
The terms with distinct $x,y$ cancel; the remaining identity follows
by checking whether $x\in S$. A specified vertical edge crossing a seam
is occupied with probability $1/6$. Consequently its cyclic sum is
\begin{equation}\label{rev3d-eq-addendum-row-contact}
 K_h:=\frac1{|\mathcal D(\emptyset)|}\sum_{S_0,\ldots,S_L\subseteq V}
 k_h(S_L,S_0)\prod_{j=0}^{L-1}w(S_j,S_{j+1})
 =\frac23\sum_xh_x^2.
\end{equation}
Indeed, since $k_h$ is diagonal, the sum in
\eqref{rev3d-eq-addendum-row-contact} equals
$\sum_xh_x^2\,|\mathcal D(\emptyset)|^{-1}\sum_S(1-2\cdot\mathbf 1_{x\in S})
(W^L)(S,S)$, and $|\mathcal D(\emptyset)|^{-1}(W^L)(S,S)$ is the probability
that the set of vertical dimers crossing a given seam equals $S$. Thus
$|\mathcal D(\emptyset)|^{-1}\sum_S\mathbf 1_{x\in S}(W^L)(S,S)=1/6$, by the
symmetry among the six edges incident to a vertex, and
$K_h=\sum_xh_x^2(1-2\cdot\frac16)$.

Taking $h_c(x)=\cos(q\cdot x)$ and $h_s(x)=\sin(q\cdot x)$ gives
\begin{equation}\label{rev3d-eq-cosine-sine-identities}
 F_{h_c}(t)+F_{h_s}(t)=L^2F_q(t),\qquad
 K_{h_c}+K_{h_s}=\frac{2L^2}{3}.
\end{equation}
The first identity follows from translation invariance and
$\cos a\cos b+\sin a\sin b=\cos(a-b)$; the second follows from
\eqref{rev3d-eq-addendum-row-contact} and $h_c(x)^2+h_s(x)^2=1$.

{\color{black}
\subsubsection{An isotropic infrared bound from the layer representation}
\begin{lemma}[Layer representation]\label{bis-lem-representation}
For each $q\in Q_L$ there is a finite family of coefficients
\blu{$(a_{q,\alpha},r_{q,\alpha})_\alpha$} such that\blu{, for every
$t\in\{0,\ldots,L-1\}$,}
{\color{black}
\begin{equation}\label{rev3d-eq-transverse-representation}
\begin{aligned}
 F_q(t)&=\sum_\alpha a_{q,\alpha}
   \bigl(r_{q,\alpha}^{t+1}+r_{q,\alpha}^{L-t+1}\bigr),
 &a_{q,\alpha}&\ge0,\quad r_{q,\alpha}\in[-1,1],\\
 \sum_\alpha a_{q,\alpha}(1-r_{q,\alpha}^L)(1-r_{q,\alpha}^2)
 &=\frac23.
\end{aligned}
\end{equation}
}
\end{lemma}
}
\begin{proof}
Fix a real function $h$ on $V$. \rev{Since $W$ is real symmetric, there
exists an orthonormal basis $(\phi_i)_i$ of eigenvectors of $W$; we denote by
$\lambda_i\in\mathbb R$ the eigenvalue of $\phi_i$, and we recall that
$\lambda_i\ne0$, since $W$ is invertible. We write
$(b_h)_{ij}=\langle\phi_i,b_h\phi_j\rangle$ for the entries of the matrix
\eqref{rev3d-eq-weighted-insertion} in this basis.}
Then $|\mathcal D(\emptyset)|\rev{\,=\operatorname{Tr}(W^L)}=\sum_i\lambda_i^L>0$, and
$b_hW=Wb_h^{\mathsf T}$ gives
$\lambda_j(b_h)_{ij}=\lambda_i(b_h)_{ji}$ (take the matrix element of
both sides between \rev{$\phi_i$ and $\phi_j$}).
If $\mathsf p_i=\lambda_i^L/|\mathcal D(\emptyset)|$,
\eqref{rev3d-eq-addendum-row-source} becomes
\[
 F_h(t)=\sum_{i,j}\mathsf p_i(\lambda_j/\lambda_i)^t (b_h)_{ij}(b_h)_{ji},
\]
since $\operatorname{Tr}(b_hW^tb_hW^{L-t})
=\sum_{i,j}(b_h)_{ij}\lambda_j^{t}(b_h)_{ji}\lambda_i^{L-t}$: the eigenvalue
$\lambda_j$ acts on the $t$ layers on one side of the two monomers and
$\lambda_i$ on the $L-t$ layers on the other side.
Each unordered off-diagonal pair \bis{contributes $a(r^{t+1}+r^{L-t+1})$ to $F_h(t)$}, with
$a=\mathsf p_i (b_h)_{ij}^2$ and $r=\lambda_j/\lambda_i$, after choosing
$|\lambda_i|\ge|\lambda_j|$. Indeed, with this choice $r\in[-1,1]$,
$(b_h)_{ji}=r(b_h)_{ij}$ and $\mathsf p_j=r^L\mathsf p_i$, so that the two
ordered terms $(i,j)$ and $(j,i)$ add up to
$\mathsf p_i(b_h)_{ij}^2\,(r^{t+1}+r^{L-t+1})$; the coefficient is
nonnegative because $L$ is even. A diagonal term is represented by
$r=1$ and $a=\mathsf p_i (b_h)_{ii}^2/2$. \bis{The sum of $a(1-r^L)(1-r^2)$ over these terms equals} the
contact sum $K_h$ of \eqref{rev3d-eq-addendum-row-contact}: writing
$K_h=|\mathcal D(\emptyset)|^{-1}\operatorname{Tr}\big((b_hb_h^{\mathsf T}
-b_h^{\mathsf T}b_h)W^L\big)=\sum_{i,j}\mathsf p_i\big((b_h)_{ij}^2-(b_h)_{ji}^2\big)$,
the diagonal terms vanish and each unordered pair contributes
\[
 (\mathsf p_i-\mathsf p_j)((b_h)_{ij}^2-(b_h)_{ji}^2)
 =a(1-r^L)(1-r^2).
\]
By \eqref{rev3d-eq-cosine-sine-identities}, pooling the finite lists for
$h_c$ and $h_s$ and dividing their coefficients by $L^2$ gives
\bis{\eqref{rev3d-eq-transverse-representation}, including for $q=0$.}
\end{proof}

{\color{black}
\begin{proposition}[Isotropic infrared bound]\label{bis-prop-isotropic}
For every even $L\ge4$ and every $k\in\T_L^*\setminus\{o\}$,
\begin{equation}\label{bis-isotropic}
 \widehat\Xi_L(k)\le
 \left(\frac{1}{\widehat\Xi_L(0,0)}+\frac{\varepsilon(k)}{K}\right)^{-1}
 \le\frac{2}{3\varepsilon(k)},\qquad K=\frac23 .
\end{equation}
\end{proposition}
\begin{proof}
Fix a horizontal frequency $q$. In \eqref{rev3d-eq-transverse-representation}
write $a_\alpha,r_\alpha$ for its coefficients, and put
$\kappa_\alpha=a_\alpha(1-r_\alpha^L)(1-r_\alpha^2)$, so
$\sum_\alpha\kappa_\alpha=K$. First suppose $|r_\alpha|<1$ for all $\alpha$.
At a vertical frequency $\theta$, set $s=2(1-\cos\theta)$ and
$u_\alpha=r_\alpha/(1-r_\alpha)^2$. Summing the two geometric series in
\eqref{rev3d-eq-transverse-representation} and using $e^{iL\theta}=1$, we obtain
\[
 \widehat\Xi_L(q,\theta)
 =\sum_\alpha\kappa_\alpha
       \frac{r_\alpha}{1-2r_\alpha\cos\theta+r_\alpha^2}
 =\sum_\alpha\kappa_\alpha\frac{u_\alpha}{1+s u_\alpha},
 \qquad
 \widehat\Xi_L(q,0)=\sum_\alpha\kappa_\alpha u_\alpha .
\]
\blu{The first identity follows from
\[
 \sum_{t=0}^{L-1}e^{-i\theta t}\bigl(r^{t+1}+r^{L-t+1}\bigr)
 =(1-r^L)\left[\frac{r}{1-re^{-i\theta}}
              +\frac{r^2e^{i\theta}}{1-re^{i\theta}}\right]
 =(1-r^L)(1-r^2)\,\frac{r}{1-2r\cos\theta+r^2},
\]
and the second one from $1-2r\cos\theta+r^2=(1-r)^2+sr$.}
Since $u_\alpha+1/4=(1+r_\alpha)^2/[4(1-r_\alpha)^2]>0$ and
$0\le s\le4$, the function $u\mapsto u/(1+su)$ is concave on
$(-1/4,\infty)$. Jensen's inequality with weights $\kappa_\alpha/K$ gives
\[
 \widehat\Xi_L(q,\theta)
 \le \frac{\widehat\Xi_L(q,0)}{1+s\widehat\Xi_L(q,0)/K}.
\]
The denominator is positive\blu{, since $\widehat\Xi_L(q,0)/K$ is an
average of the numbers $u_\alpha>-1/4$ and $s\le4$}. Hence, if $\widehat\Xi_L(q,\theta)>0$,
then $\widehat\Xi_L(q,0)>0$ and
\begin{equation}\label{bis-reciprocal}
 \frac{1}{\widehat\Xi_L(q,\theta)}-
 \frac{1}{\widehat\Xi_L(q,0)}\ge\frac{2(1-\cos\theta)}{K}.
\end{equation}
If some $r_\alpha=\pm1$, replace each ratio in the finite sum by
$(1-\delta)r_\alpha$ and recompute $K$. The Fourier sums and $K$ converge
to their original values as $\delta\downarrow0$. When the mode at
$\theta$ is positive in the limit, its perturbed value is positive for
small $\delta$, and the perturbed mode at zero is at least as large by
\eqref{bis-reciprocal}. Thus both reciprocals pass to the limit.

By cubic symmetry \eqref{bis-reciprocal} holds with any coordinate as the
layer direction. Starting at a mode $k$ with $\widehat\Xi_L(k)>0$, set
its coordinates to zero one at a time. Each resulting mode remains
positive, so the three inequalities add to
\[
 \frac1{\widehat\Xi_L(k)}-\frac1{\widehat\Xi_L(0,0)}
 \ge\frac{2}{K}\sum_{j=1}^3(1-\cos k_j)
 =\frac{\varepsilon(k)}{K}.
\]
This proves \eqref{bis-isotropic} for positive modes. For nonpositive modes
it holds as well, since $\widehat\Xi_L(0,0)>0$ by
\eqref{rev3d-mean-definition} and $\Xi_L(\boldsymbol e_1)=1/6$.
\end{proof}
}

\begin{proposition}[Positive spatial average in dimension three]\label{rev3d-main}
For $d=3$ there exists $c_{\mathrm{av}}>0$ such that
\begin{equation}\label{rev3d-positive-mean}
 \liminf_{\substack{L\to\infty\\L\ \mathrm{even}}}
 \frac1{|\T_L^o|}\sum_{z\in\T_L^o}\Xi_L(z)\ge c_{\mathrm{av}}.
\end{equation}
By \eqref{eq:relationdimerpermutationstat}, the same conclusion holds for
$\mathbb G_{L,2,0}$.
\end{proposition}
{\color{black}
\begin{proof}
Apply the pairing of modes in \eqref{eq:corrected-neighbour-Fourier}
with $N=2$ and $d=3$. By \eqref{rev3d-parity}, the two contributions of
each pair $k,k+(\pi,\pi,\pi)$ are equal. At the member with $D_3(k)\ge0$,
\eqref{bis-isotropic} improves the upper bound used in
\eqref{eq:corrected-higher-d-finite} by the factor $2/3$. Thus
\[
 A_L\ge\frac16\left(1-\frac23 J_L(3)\right).
\]
Section~\ref{sect:corrected-higher-d} gives $\lim_{L\to\infty}J_L(3)\le h_3$.
Since $r_3<0.52$~\cite{Watson}, we have
$h_3=r_3+\sqrt{r_3(1+r_3)}<1.41$. It follows that
\[
 \liminf_{\substack{L\to\infty\\ L\ {\rm even}}}A_L
 \ge\frac16\left(1-\frac23 h_3\right)>\frac1{100}.
\]
\end{proof}
}

}

\section{Proof of Theorems \ref{theo:theo2} 
and \ref{theo:theo3}}
\label{sect:proofs}
In this section we present the proofs of Theorems  \ref{theo:theo2} 
and \ref{theo:theo3}.

\begin{proof}[\textbf{{\textit{Proof of  (\ref{eq:permpositivityaverage}) in Theorem \ref{theo:theo3}.}}}]
To begin, we claim that, for any {even $L\ge4$},
\begin{equation}\label{eq:precisevalue}
 \mathbb{G}_{L, N, 0}(o,\boldsymbol{e}_1) = \frac{1}{d\, N}.
\end{equation}
To see why this is true, define the map
$\Pi : \Omega_{o, \boldsymbol{e}_1}  {\longrightarrow} \{ \pi \in \Omega^\ell \,  : \,   \, \, {(\boldsymbol e_1,o)}  \in E_\pi \}$ which associates to any $\pi \in \Omega_{ o, \boldsymbol{e}_1}$ an element
 $ \Pi(\pi)$ which is obtained from $\pi$ by adding to $\pi$ an edge directed from $\boldsymbol{e}_1$ to $o$.
Note that, by definition of $\Omega_{o, \boldsymbol{e}_1}$, such a directed edge cannot be already present in $\pi  \in \Omega_{o, \boldsymbol{e}_1}$ (but an edge directed from $o$ to $\boldsymbol{e}_1$ might be present!), and that this map is one-to-one. {It is onto the displayed set, since removing the indicated edge from its directed cycle produces the required directed open walk.} Thus, we
 deduce that,
 \begin{align}\label{eq:relationtwopointprobability}
 \mathbb{Z}_{ L,  N, \rho }(o, \boldsymbol{e}_1) & =   \sum\limits_{ \pi   \in  \Omega_{o, \boldsymbol{e}_1}    }
\rho^{ \mathcal{M}(\pi)  } \, \, ( \frac{N}{2}) ^{ \mathcal{L}(\pi)  } \notag\\
&=\frac{2}{N} 
 \sum\limits_{\substack{   \pi   \in  \Omega^\ell : \\  
{(\boldsymbol e_1,o)} \in E_\pi }   }
\rho^{ \mathcal{M}(\pi)  } \, \, ( \frac{N}{2}) ^{ \mathcal{L}(\pi)  } \notag\\
&= \frac{2}{ N } \,  \frac{1}{ 2d } \, 
 \sum\limits_{\substack{   \pi   \in  \Omega^\ell : \\  {\exists y\sim o:\ (y,o)\in E_\pi} }   }
\rho^{ \mathcal{M}(\pi)  } \, \, ( \frac{N}{2}) ^{ \mathcal{L}(\pi)  } 
 \end{align}
where  $\mathcal{L}(  \Pi(\pi) ) = \mathcal{L}(\pi) + 1$, and the last step follows from reflection and rotational symmetry. {The $2d$ possible incoming edges at the origin are mutually exclusive in a closed configuration.}
From this and {(\ref{eq:twopoint})} we deduce that,
$$
 \mathbb{G}_{L, N, \rho} \big (o,\boldsymbol{e}_1 \big ) = \frac{1}{d \, N } \,  \frac{
 \sum\limits_{\substack{   \pi   \in  \Omega^\ell : \\  {\exists y\sim o:\ (y,o)\in E_\pi} }   }
\rho^{ \mathcal{M}(\pi)  } \, \, ( \frac{N}{2}) ^{ \mathcal{L}(\pi)  } 
}{
 \sum\limits_{\substack{   \pi   \in  \Omega^\ell }   }
\rho^{ \mathcal{M}(\pi)  } \, \, ( \frac{N}{2}) ^{ \mathcal{L}(\pi)  } 
}
$$
Since for any finite {even $L\ge4$}, the second factor equals one when $\rho = 0$ (the origin is not a monomer almost surely),  the proof of (\ref{eq:precisevalue}) is concluded.
{
We now use the two dimension-dependent bounds proved in
Section~\ref{sect:InfraredBound}.

\paragraph{Dimensions $d>3$.}
Proposition~\ref{prop:corrected-higher-d}, together with
(\ref{eq:precisevalue}), gives a strictly positive lower bound for the
lower limit of the spatial average whenever $Nh_d<2$.  This includes
all $N\in\mathcal N_d$, since $h_d<1$ for $d\ge4$.
Thus \eqref{eq:permpositivityaverage} holds for some
$c_{\mathrm{av}}(d,N)>0$ in these dimensions.

\paragraph{Dimension $d=3$.}
For $N=1$, the same conclusion follows from
Proposition~\ref{prop:corrected-higher-d}, since $h_3<2$.
For $N=2$, use (\ref{eq:relationdimerpermutationstat}) and
Proposition~\ref{rev3d-main}.  The latter proves that the lower limit
of the spatial average is strictly positive, and we choose
$c_{\mathrm{av}}(3,2)>0$ accordingly.  These are precisely the
three-dimensional cases in Theorem~\ref{theo:theo3}.
}
\end{proof}

{
\begin{proof}[\textbf{\textit{Proof of (\ref{eq:permpointwise}) in Theorem \ref{theo:theo3}}.}]
Represent coordinates in $(-L/2,L/2]$.
The site-monotonicity bound in \cite[Theorem 2.4 and equation (2.24)]{Lees},
together with \eqref{eq:precisevalue} and the vanishing of
$\mathbb G_{L,N,0}(o,z)$ on the even sublattice, gives
\begin{equation}\label{eq:monotonicity}
 0\le\mathbb G_{L,N,0}(o,z)\le\frac1{dN}
 \qquad(z\in\mathbb T_L).
\end{equation}
Moreover, \cite[Theorem 2.4 and Proposition 4.1]{Lees} and reflection
symmetry imply that, whenever $z_i$ and $n$ are odd and
$1\le n\le |z_i|\le L/2$,
\begin{equation}\label{eq:corrected-site-projection}
 \mathbb G_{L,N,0}(o,z)
 \le\mathbb G_{L,N,0}(o,|z_i|\boldsymbol e_i)
 \le\mathbb G_{L,N,0}(o,n\boldsymbol e_i).
\end{equation}

For an odd integer $1\le n\le L/2$, set
\[
 \mathbb S_{n,L}:=\{z\in\mathbb T_L^o:
                   \exists i,\ |z_i|<n\}.
\]
There are $2n-1$ excluded values per coordinate, and each coordinate
section contains $L^{d-1}/2$ odd sites. Hence
\[
 \frac{|\mathbb S_{n,L}|}{|\mathbb T_L^o|}
 \le\frac{d(2n-1)}{L}.
\]
Every $z\in\mathbb T_L^o\setminus\mathbb S_{n,L}$ has an odd
coordinate of absolute value at least $n$. Thus
\eqref{eq:corrected-site-projection} and coordinate symmetry give
$\mathbb G_{L,N,0}(o,z)\le
\mathbb G_{L,N,0}(o,n\boldsymbol e_j)$ for every $j$.
Decomposing the spatial average and using \eqref{eq:monotonicity},
we obtain
\[
 \mathbb G_{L,N,0}(o,n\boldsymbol e_j)
 \ge\frac1{|\mathbb T_L^o|}
 \sum_{z\in\mathbb T_L^o}\mathbb G_{L,N,0}(o,z)
 -\frac{2n-1}{NL}.
\]
Fix $\varphi\in(0,c_{\rm av}(d,N))$ and choose $c_3\in(0,1/2)$
such that $2c_3/N<c_{\rm av}(d,N)-\varphi$.
Equation~\eqref{eq:permpositivityaverage} and the preceding inequality
then imply the claim for all odd $n\in(0,c_3L)$ and all sufficiently
large even $L$.
\end{proof}
}

\begin{proof}[\textbf{\textit{Proof of Theorem \ref{theo:theo2}}}]
Theorem \ref{theo:theo2} is an immediate consequence of 
Theorem \ref{theo:theo3}. For any $A \subset \T_L$, we have that,
\begin{align*}
 \mathbb{P}_{L, N, 0}( X \in A ) &= \sum\limits_{x \in A}   \mathbb{P}_{L, N, 0}( X =x ) \\
&=  \frac{\sum\limits_{x \in A} \mathbb{Z}_{L, N, 0}(o,x)}{\sum\limits_{x \in \T_L} \mathbb{Z}_{L, N, 0}(o,x)} =
 \frac{\sum\limits_{x \in A} \mathbb{G}_{L, N, 0}(o,x)}{\sum\limits_{x \in \T_L} \mathbb{G}_{L, N, 0}(o,x)},
\end{align*}
where the last identity follows after dividing the numerator and the denominator by $\mathbb{Z}^\ell_{L, N, 0}$.
Now the claim follows from  (\ref{eq:permpositivityaverage}), which provides a lower bound for the denominator in the right-most term, and from {(\ref{eq:monotonicity})}, which provides an upper bound for the numerator in the right-most term. {Indeed, for all sufficiently large even $L$, the denominator is bounded below by a strictly positive constant times $L^d$, since the even-displacement two-point function vanishes at $\rho=0$. The numerator is at most $|A|/(dN)$.} Using both bounds we obtain  {(\ref{eq:mainestimate})}.
\end{proof}

\section*{Acknowledgements}
This work started as the author was affiliated at Technische Universit\"at Darmstadt and was funded by DFG (grant number: BE 5267/1), 
it continued as the author was affiliated at the University of Bath
and funded by {EPSRC} (grant number: EP/N004566/1), and it was concluded as the author was affiliated at the Weierstrass Institute for Applied Analysis and Stochastics, Berlin.
The author thanks Volker Betz,  Ron Peled, Thomas Spencer and Daniel Ueltschi for intriguing discussions and Volker Betz for introducing  the model of random lattice permutations to him.

As explained in the \hyperlink{remark:correction}{remark} at the end
of Section~\ref{sect:Introduction}, the author found an error in the computation
of the infrared-bound constant in version~3 of the preprint~\cite{TaggiPreprint}.
The author acknowledges assistance from ChatGPT (Astra~6.0) in developing
and editing the corrected argument, presented on
pages~\pageref{prop:corrected-higher-d}--\pageref{rev3d-main},
and has carefully checked all arguments.
The remainder of the paper was written in 2018, before the development
of this technology.

\section{Appendix}

\begin{proof}[Proof of (i), (ii) and (iii) in the proof of Proposition \ref{prop:highfrequencybound}]
These computations are classical and  they can be extracted for example from the computations in  \cite{UeltschiMarseille}. We present them  for the reader's convenience.
The proof of  (i) consists of the following computation,
\begin{align*}
( \triangle v)_x & =
\sum\limits_{y \sim o } \Big ( \,  \cos \big ( (x+y) \cdot k   \big) \, - \, 
v_x \, \Big ) = 
\sum\limits_{y\sim o } \Big ( \,  \cos \big ( x \cdot k   \big) \, 
\cos \big ( y \cdot k   \big) \, \,-  
 \sin \big ( x \cdot k   \big) \, 
 \sin \big (y \cdot k   \big)
- v_x \,  \Big ) \\ 
& = 
\sum\limits_{y\sim o } \Big ( \, v_x 
\cos \big ( y \cdot k   \big) \, 
 -  \, v_x \,  \Big )   = - \varepsilon(k) \, \, v_x.
\end{align*}
The proof of  (ii) follows from the first Green identity, which states that, for any pair of real-valued vectors, $\boldsymbol{a} = (a_x)_{x \in \T_L}, \boldsymbol{b} = (b_x)_{x \in \T_L}$, when $(\T_L, \E_L)$ is the torus,
$$
  \sum\limits_{ \{x,y\}\in \E_L} ( b_y - b_x )\, \, (a_y - a_x) =  
 -  \sum\limits_{x \in \T_L}   {a}_x  \, \, (\triangle b)_x.
   $$  
   The  proof of such an identity can be found for example in \cite{FriedliVelenik}[Lemma 8.7].
   Applying such an identity with $ \boldsymbol{a} = \boldsymbol{b} = \boldsymbol{v}$ and using (i), we obtain (ii).   
It remains to prove (iii).
For this, we use the fact that, by lattice symmetries, $\hat{\mathbb{G}}(k)$ is real and we obtain:
\begin{align*}
& \sum\limits_{x,y \in \T_L} \cos(k \cdot x)  \cos(k \cdot y) \mathbb{G}(x,y)  =\sum\limits_{x\in \T_L}  \Big ( \cos(k \cdot x)  Re \Big [ \sum\limits_{y \in \T_L} \cos(k \cdot y)  \mathbb{G}(y-x)  \Big ] \Big )  \\  
 & =  \sum\limits_{x\in \T_L}  \Big ( \cos(k \cdot x)  Re \Big [  e^{ i k \cdot x }\sum\limits_{y \in \T_L} e^{i k \cdot (y-x)}  \mathbb{G}(y-x)  \Big ] \Big ) = 
\sum\limits_{x\in \T_L}  \big ( \cos(k \cdot x) \,  Re [ e^{ i k \cdot x }   {\hat{\mathbb G}(k)} ] \,  \big ) \\ & = \sum\limits_{x\in \T_L} \,  \cos^2(k \cdot x) \,      \hat  {\mathbb{G}}(k) .
\end{align*}
This concludes the proof.

\end{proof}

\end{document}